\documentclass{amsart}    
\usepackage{amsmath,amscd,amsfonts,amssymb,amsthm,enumitem,newlfont,exscale,mathrsfs, multicol,xcolor,xspace}
\UseRawInputEncoding
\usepackage[all]{xy}
\usepackage{amsmath,amsthm,amssymb,mathrsfs,stmaryrd,color}
\usepackage{comment}

\title{Lim Cohen-Macaulay sequences of modules}
\author{Bhargav Bhatt$^1$, Mel Hochster$^1$, and Linquan Ma$^1$}
\date{\today}
\theoremstyle{plain}
\newtheorem{theorem}{Theorem}[section]
\newtheorem{corollary}[theorem]{Corollary}
\newtheorem{proposition}[theorem]{Proposition}
\newtheorem{lemma}[theorem]{Lemma}

\newtheorem{mofact}[theorem]{Motivating Fact}
\newtheorem*{theorem*}{Theorem}

\theoremstyle{remark}
\newtheorem{remark}[theorem]{Remark}

\newtheorem{notation}[theorem]{Notation}
\newtheorem{discussion}[theorem]{Discussion}

\theoremstyle{definition}
\newtheorem{definition}[theorem]{Definition}
\newtheorem{example}[theorem]{Example}
\newtheorem{examples}[theorem]{Examples}

\newcommand{\0}{^\circ}

\newcommand{\Ann}{\mathrm{Ann}}
\newcommand\arwf[1]{\xrightarrow{#1}}

\newcommand\bem[1]{\bibitem[#1]{#1}}
\newcommand{\ben}{\begin{enumerate}}
\newcommand{\bena}{\ben[label=(\alph*)]}

\newcommand{\benn}{\ben[label=(\arabic*)]}
\newcommand{\benr}{\ben[resume*]}
\newcommand\bensp[1]{\ben[label=#1]}
\newcommand\blank{\underline{\hbox{\ \ }}}
\newcommand\bu{\bullet}
\newcommand{\cl}{\cln}
\newcommand{\cla}{^{*\cM,\alpha}}
\newcommand{\clb}{^{*\cM,\beta}}
\newcommand{\cln}{^{\natural}}

\newcommand{\Coker}{\mathrm{Coker}}
\newcommand{\C}{\mathbb{C}}

\newcommand{\cB}{\mathcal{B}}
\newcommand{\cC}{\mathcal{C}}

\newcommand{\cF}{\mathcal{F}}
\newcommand{\cG}{\mathcal{G}}

\newcommand{\cI}{\mathcal{I}}
\newcommand{\cJ}{\mathcal{J}}
\newcommand{\cK}{\mathcal{K}}
\newcommand{\cL}{\mathcal{L}}
\newcommand{\cM}{\mathcal{M}}
\newcommand{\CM}{Cohen-Macaulay\xspace}

\newcommand{\cP}{\mathcal{P}}

\newcommand{\cV}{\mathcal{V}}

\newcommand{\cZ}{\mathcal{Z}}
\newcommand{\dm}{\mathrm{dim}}
\newcommand{\depth}{\mathrm{depth}}

\newcommand{\disp}{\displaystyle}

\newcommand{\een}{\end{enumerate}}

\newcommand{\Ext}{\mathrm{Ext}}

\newcommand{\gr}{\mathrm{gr}}
\newcommand{\fA}{\mathfrak{A}}
\newcommand{\fB}{\mathfrak{B}}

\newcommand{\fL}{\mathfrak{L}}

\newcommand{\fN}{\mathfrak{N}}
\newcommand{\fp}{\mathfrak{p}}
\newcommand{\fq}{\mathfrak{q}}
\newcommand{\hgt}{\mathrm{height}}
\newcommand{\Hom}{\mathrm{Hom}}
\newcommand{\id}{\mathrm{id}}
\newcommand{\Img}{\mathrm{Im}}
\newcommand{\imp}{\Rightarrow}
\newcommand{\inc}{\subseteq}

\newcommand{\Ker}{\mathrm{Ker}}
\newcommand{\lCM}{lim Cohen-Macaulay\xspace}
\newcommand{\m}{\mathfrak{m}}
\newcommand{\MaxSpec}{\mathrm{MaxSpec}}
\newcommand{\MIN}{\mathfrak{Min}}
\newcommand{\n}{\mathfrak{n}}
\newcommand{\nn}{{}^n}
\newcommand{\N}{\mathbb{N}}
\newcommand{\pchar}{prime characteristic $p>0$\xspace}

\newcommand{\oo}{\mathrm{o}}
\newcommand{\OO}{\mathrm{O}}
\newcommand{\ov}{\overline}
\newcommand{\Q}{\mathbb{Q}}
\newcommand{\R}{\mathbb{R}}
\newcommand{\rank}{\mathrm{rank}}

\newcommand{\rmk}{(R,\,\mathfrak{m},\,K)}

\newcommand{\smo}{\setminus\{0\}}

\newcommand{\Spec}{\mathrm{Spec}}
\newcommand{\sop}{system of parameters\xspace}
\newcommand{\surj}{\twoheadrightarrow}

\newcommand{\Tor}{\mathrm{Tor}}
   
\newcommand{\und}{\underline}

\newcommand{\ux}{\und{x}}
\newcommand{\UX}{\und{X}}
\newcommand{\uy}{\und{y}}

\newcommand\vect[2]{#1_1,\,\ldots,\, #1_{#2}}

\newcommand{\wh}{\widehat}
\newcommand{\wt}{\widetilde}
\newcommand{\Z}{\mathbb{Z}}
\def\todo#1
{\par\vskip 7 pt \noindent {\textcolor{yellow}\footnotesize \color{yellow} \fbox{\parbox{4.8in}{\color{black}
\textbf{To do: } {\color{blue}#1}}}} \vskip 3 pt \par}
\def\forth#1
{\par\vskip 7 pt \noindent {\textcolor{yellow}\footnotesize \color{yellow} \fbox{\parbox{4.8in}{\color{black}
\textbf{For thought: } {\color{blue}#1}}}} \vskip 3 pt \par}

\begin{document}

\begin{abstract} We introduce the notion of a lim Cohen-Macaulay sequence of modules.  We prove
the existence of such sequences in positive characteristic, and show that their existence in mixed characteristic
implies the long open conjecture about positivity of Serre intersection multiplicities for all regular local rings, as well
as a new proof of the existence of big \CM modules.
We describe how such a sequence leads to a notion of closure for submodules of finitely generated
modules:  this family of closure operations includes
the usual notion of tight closure in characteristic $p>0$, and all of them have the property of
capturing colon ideals. In fact they satisfy axioms formulated by G.~Dietz from which it follows
that if a local ring $R$ has a \lCM\ sequence then it has a big Cohen-Macaulay module.
We also prove the existence of \lCM\ sequences for
certain rings of mixed characteristic.

  \end{abstract}

\subjclass[2020]{Primary 13}


\thanks{$^1$ The first author was partially supported by grants from the National Science Foundation DMS--1801689, 1952399, 1840234, the Packard Foundation, and the Simons Foundation, the second author was partially supported by grants from the National Science Foundation DMS--1902116, 2200501, and the third author was partially supported by grants from the National Science Foundation DMS--2302430, 1952366, 1901672, and a fellowship from the Sloan Foundation when preparing this article. The authors thank the referee for their comments.}

\maketitle

\pagestyle{myheadings}
\markboth{BHARGAV BHATT,  MEL  HOCHSTER, AND LINQUAN MA}{LIM COHEN-MACAULAY SEQUENCES OF MODULES}

\section{Introduction}\label{intro} %%1

Throughout, all rings are commutative, associative, with identity, and homomorphisms are
assumed to preserve the identity.  By a {\it local ring} $\rmk$ we mean a Noetherian ring
$R$ with a unique maximal ideal $\m$ and residue class field $K$.   In \S \ref{ss:ExecSummary}, we provide a concise introduction to the results of this paper. A more leisurely introduction with background and context is provided in \S \ref{ss:NotIntro}, \S \ref{ss:BCMIntro}, \S \ref{ss:ObjIntro}, \S \ref{ss:SmallCMIntro} and \S  \ref{ss:SerreIntro}.

\subsection{Executive summary}
\label{ss:ExecSummary}
A fundamental conjecture in commutative algebra predicts that complete local domains have
small Cohen-Macaulay modules. This conjecture has many consequences, but is wide open in general. Our objective is to show that a greatly weakened version of this conjecture suffices both to give new proofs of homological results like
the existence of big Cohen-Macaulay modules and to prove the positivity part of the Serre conjecture
on multiplicities, which  has been an open
question for over fifty years.  The new conjecture is phrased in terms of the existence
of {\it lim Cohen-Macaulay sequences}, which is the main new object introduced in this paper.

\begin{definition}
\label{def:limCMintro}
A sequence of nonzero finitely generated modules
$\cM = \{M_n\}_n$ over a local ring of $\rmk$ of Krull dimension $d$ is called {\it \lCM}  if
for some (equivalently, every) system of parameters $\ux = \vect x d$ for $R$,
$\ell\bigl(H_i(\ux; M_n)\bigr) =\oo\bigl(\nu(M_n)\bigr)$ for all $i \geq 1$,  where $\ell(H)$ denotes
the length of $H$,  $H_i(\ux;\, M)$ denotes
the $i\,$th Koszul homology of $M$, and $\nu(M)$ denotes the least number of generators
of $M$, i.e. $\dm_K(K \otimes_R M)$.
\end{definition}

 In \S\ref{limdef}
we show that this condition is independent of the system of parameters and that
when $R$ is a domain,  $\nu(M_n)/\rank(M_n)$ is bounded for such a sequence,
so that the sequence $\cM$ is \lCM\ if and only if
$\ell\bigl(H_i(\ux; M_n)\bigr) = \oo\bigl(\rank(M_n)\bigr)$ for all $i \geq 1$. An important class of examples is provided by the following result, ensuring that lim \CM sequences always exist over F-finite local rings of prime characteristic $p >0$.

 \begin{theorem*}[\bf\ref{nMlimcm}] Let $\rmk$ be an F-finite local ring of prime characteristic $p >0$
of Krull dimension $d$,  and let
$M$ be any $R$-module of Krull dimension $d$.  Then the sequence $\{F^n_*(M)\}_n$ is a lim \CM sequence of
$R$-modules \end{theorem*}

One of our motivations for exploring Definition~\ref{def:limCMintro} is Serre's intersection multiplicity conjecture. In fact, one consequence of our results is a proof that the Serre multiplicity \cite{Ser65} behaves correctly\footnote{I.e., is nonnegative and positive if and only if the intersection is proper.} in the equal characteristic case that avoids the technique of reduction to the diagonal.  Instead,  our proof in  characteristic $p$ proof uses \lCM sequences, and the
equal characteristic 0 case follows by reduction to the case of \pchar, a technique pioneered in \cite{PS74}.  More precisely, in \S\ref{limcmSerre}, we prove that if one can establish the existence of lim \CM sequences for complete local domains of mixed characteristic with algebraically closed residue field, then Serre's conjecture follows; the characteristic $p$ case of the latter then follows from Theorem~\ref{nMlimcm} discussed above.  The key in this direction is Theorem~\ref{pos} (and  Remark~\ref{geq1}):

 \begin{theorem*}[\bf \ref{pos}] Let $(T,\,\m,\,K)$ be a regular local ring of Krull dimension $d$.
 Let $P$ and $Q$ be prime ideals of $T$ such that $\dm(T/P) + \dm(T/Q)  = d$
 and $P+Q$ is $\m$-primary.  Assume that $R := T/P$ and $S := T/Q$ admit lim Cohen-Macaulay
 sequences $\{M_n\}_n$ and $\{N_n\}_n$, respectively. Then the multiplicity
 \[ \chi(R,\,S) := \sum_{i=0}^d (-1)^i \ell \left( \mathrm{Tor}_i^T(R,S) \right)\]
  is positive.  In fact,
 $$\chi(R,\,S) =  \lim_{n \to \infty} \frac{\ell(M_n \otimes_T N_n)}{\rank(M_n)\rank(N_n)} \geq 1.$$ \end{theorem*}

 In \S\ref{op} we describe a very general method for defining a closure operation from a
 sequence\footnote{One may also use a {\it net} of modules: a family indexed by a directed set. See  Remark~\ref{net}.}
 of finitely generated  $R$-modules (such as a lim \CM sequence) and an $\N_+$-valued function $\alpha$ defined on
 the  modules in these sequences.  Typically, $\alpha$ is the least number of generators or rank.  We show that
 integral closure for ideals arises in this way (Example \ref{intclop}), as well as tight closure of modules for a reduced F-finite equidimensional
 local rings (Theorem~\ref{tclimcm}).  We also show that under surprisingly weak assumptions, closures arising from a sequence of modules
 ({\it not} necessarily a lim \CM sequence) have the property that integrally closed ideals are closed in the ring
 (Theorem \ref{cloint}).

 In \S\ref{cap} we prove that the closure operations arising from lim \CM sequences have several colon-capturing
 properties like those of tight closure.  We use these to show that the direct summand theorem\footnote{For background,
 see \cite{Ho73a, Ho83, Heit02, And18, Bha18}. For information about related homological conjectures, see
 \cite{Aus63, Bass63, Du87, DHM85, EvG81, Heit93, Ho75a, PS74, PS76, Rob76, Rob80, Rob85, Rob87, Rob89}.}  follows from the existence of lim \CM sequences, and we also show that for a regular local ring, every submodule of every
 module is closed with respect to every lim \CM closures.

 In \S\ref{big} we prove that every lim \CM closure on a local ring
 is what is called a {\it Dietz closure}:
 it satisfies axioms developed by G.~Dietz  \cite{Di10}. This implies that every local ring
 that has a lim \CM sequence of modules has a big \CM module.  Hence, whenever we
 can prove that a local ring $R$ has a lim \CM sequence, we get a new proof that $R$ has
 a big \CM module. These results were developed
 in the hope of obtaining a proof of the existence of big \CM modules in mixed characteristic involving
 considerably less machinery than the existing proofs (which use almost mathematics and perfectoid
 geometry).  However, our current results on the existence of lim \CM sequences, while satisfactory
 in prime characteristic $p >0$, are very limited in mixed characteristic.

 In \S\ref{strong}, we define and prove some results about {\it strong}\footnote{In an earlier version of this
 paper, this was the definition of lim \CM sequence.  The weaker notion used here seems more natural
 and yields the same main results about positivity of Serre multiplicities and existence of big \CM modules.}
lim  \CM sequences, including the
 non-obvious fact that they are, as one expects from the name, lim \CM sequences.  We prove that over
 an F-finite local ring $R$, if a finitely generated module $M$ has the same Krull dimension as $R$,
 then $\{F^n_*(M)\}_n$ is strongly lim \CM.   See Corollary~\ref{strongimp}.

 Finally, in \S\ref{mixed} we give some examples of rings in mixed characteristic for which we can construct lim \CM sequences of algebras,
 and we show that some of these examples do not admit small \CM algebras. See Example~\ref{ex:LimCMNoFrob} and Corollary~\ref{NonExistSmallCM}.

In the rest of this section, we fix some notation and terminology and then
 discuss what is known about big and small \CM modules, some history
 for Serre's conjecture on intersection multiplicities, and the connection between the existence of small \CM
 modules and Serre's conjecture.

\subsection{Terminology and notation}
\label{ss:NotIntro}
As mentioned earlier, by a \emph{local ring} $\rmk$ we mean a Noetherian ring
$R$ with a unique maximal ideal $\m$ and residue class field $K$.
 We often assume that given rings are complete local domains,
which suffices in all the applications we have in mind.
 For certain purposes we may assume that $K$ is perfect or even algebraically closed: the latter case
suffices for the major applications that we have in mind.

The {\it rank} of a module $M$ over a domain $D$ with fraction field $\cK$ is its torsion-free rank,
i.e., $\dm_\cK (\cK \otimes_D M)$.  We denote this $\rank(M)$. More general notions of rank are
consider in Subsection~\ref{rank}.

We use the notations $\N_+$, $\N$, $\Z$, $\Q$, $\R$ and $\C$ for the positive integers, nonnegative integers,
integers, rational numbers, real numbers, and complex numbers, respectively.

\subsection{Big and small Cohen-Macaulay modules}
\label{ss:BCMIntro}
A {\it big Cohen-Macaulay module}\footnote{The term ``big Cohen-Macaulay module" is also used in the
literature for a module such that one system of parameters is a regular sequence, and then the term {\it balanced
big Cohen-Macaulay module} is used when every \sop is a regular sequence.  However, it is shown in
\cite{BaSt83}  that if one system of parameters for $R$ is a regular sequence on $M$, then the $m$-adic
completion of $M$ is a balanced big Cohen-Macaulay module and, since it is easy to make a transition from a
module satisfying the weaker condition to a module satisfying the stronger condition, we use the term
``big \CM module" in this manuscript to be synonymous with ``balanced big \CM module."}
over a local ring $(R,\m,K)$ is a (not necessarily finitely
generated) module $M$ such that $\m M \not= M$   and every system of parameters for
$R$ is a regular sequence\footnote{As part of the definition, in order for
$\vect x d$ to be a {\it regular sequence} on $M$, we require that $(\vect xd)M \not= M$ as well
as the condition $(*)$ that $x_{i+1}$ not be a zerodivisor on $M/(\vect x i)M$ for $0 \leq i \leq d-1$.
If $(*)$ holds but we allow the possibility that $(\vect x d)M =M$, we refer to $\vect x d$ as
a {\it possibly improper regular sequence} on $M$.}.

 If $M$ is finitely generated, then
$M$ is a big Cohen-Macaulay module for $R$ iff $M \not=0$ and {\it one} system of parameters
is a regular sequence on $M$.  In this case,  $M$ is called a {\it small} Cohen-Macaulay module for $R$.  In this terminology,  a small Cohen-Macaulay module
is always a big Cohen-Macaulay module and is a {\it maximal} finitely generated Cohen-Macaulay
module in the sense that $\depth_\m(M) = \dm(R)$.

   For a long time,
big Cohen-Macaulay modules and algebras were known to exist only in equal characteristic and if dim$(R) \leq 3$,
Cf. \cite{Ho75a, Ho75b, Ho94,  HH92, HH95, Heit02, Ho02}.
There has been an explosion in this area recently utilizing almost mathematics (\cite{Fal02, GaRa03}) and perfectoid geometry (\cite{Sch12}), and big Cohen-Macalay algebras are now known to exist in
general \cite{And18,HeitMa18, And20, Bha20}.

However, small Cohen-Macaulay modules are only known to exist if dim$(R) \leq 2$, if $R$ is $\N$-graded
over a perfect field of characteristic $p$ and has an isolated non-Cohen-Macaulay point at the origin,
and in a handful of other cases.  Results of this type are discussed, for example, in
\cite{Ho75b, Hanes99, Hanes04, HY23}.
Small \CM modules are not known to exist in equal characteristic 0 nor in prime characteristic $p >0$, even
over local rings at maximal ideals of affine domains of dimension 3 over algebraically closed field.

%% 1.3
\subsection{The objectives of this paper}
\label{ss:ObjIntro}
We therefore seek a weaker condition than the existence of small Cohen-Macaulay
modules that suffices both to prove the positivity of the Serre intersection multiplicity
for ramified regular local rings of mixed characteristic (the remaining open case)
and to give a new proof of the direct summand theorem.  The existence of lim Cohen-Macaulay sequences,
introduced in  \S\ref{limdef}, suffices for both, and, unlike small Cohen-Macaulay modules,
we can prove easily that such sequences exist in positive characteristic.

In \cite{Ma23},  the third author has already used a variation on the ideas presented here to prove Lech's conjecture
in equal characteristic when the base ring is a finitely generated standard graded algebra over a perfect field. The
equal characteristic 0 case is proved  by reduction to prime characteristic $p > 0$.  In characteristic $p$ the
result is proved by constructing a sequence that is lim Cohen-Macaulay in a sense (called ``weakly
lim Cohen-Macaulay")  that we discuss in subsection \ref{wklimdef},
and that also, asymptotically, consists of modules that approach the condition of being
Ulrich. These sequences are called {\it weakly lim Ulrich}, and are also utilized in \cite{IMW22}.
These methods are, in a way, a generalization of techniques developed by Hanes in \cite{Hanes99, Hanes04},
where maximal \CM modules that asymptotically approach the condition of being Ulrich are used.
 It is shown in  \cite{Yhee23} that weakly lim Ulrich sequences do not always exist for local rings of affine
 semigroup domains, even in prime characteristic $p >0$. But these rings have integral closures that
 are \CM, and lim \CM sequences do always exist, for example, for all F-finite local domains in \pchar:  see
 Section~\ref{Fflcm}.

%%1.4
\subsection{Applying small Cohen-Macaulay modules to positivity of Serre multiplicities}
\label{ss:SmallCMIntro}

Let $(T, \n, K)$ be a regular local ring of dimension $n$ and let $M, N$ be nonzero finitely generated modules such
that $\ell(M \otimes_T N) < \infty$,   i.e., such that $\text{Supp}(M) \cap \text{Supp}(N) = \{\n\}$.
Here,  $\ell(\ )$
denotes length.   {\it Serre's
intersection multiplicity} is defined by the formula
$$\chi^T(M,N) = \chi(M,N) := \sum_{i=0}^{\dim(T)} (-1)^i \ell\bigl(\Tor^T_i(M,N)\bigr).$$
The function $\chi(M,N)$ is {\it bi-additive} in $M$ and $N$ when it is defined on all the pairs occurring.
Since $M$ and $N$ have finite filtrations in which all factors are prime cyclic modules $T/P$,
the behavior of $\chi$ is determined by its behavior on pairs of such modules $T/P$, $T/Q$,   where
$P$, $Q$ are prime and $P+Q$ is $\n $-primary.
This is a formal situation analogous to
studying the intersection of two varieties near an isolated point of intersection.

In equal characteristic, $T$ is regular iff its completion is a formal power series ring over a field.
In mixed characteristic, it may be formal power series over a complete DVR $(V,pV)$    (like
the $p$-adic numbers)   whose maximal ideal is generated by the characteristic $p$ of the
residue class field.
In the (frequently more difficult) {\it ramified} case the ring has the form
$V[[X_1, \, \ldots, X_d]]/(p- F)$
where $F$ is in the square of the maximal ideal.    Such a ring, in general, is regular but {\it not} a formal
power series ring over a DVR.
 E.g., we may have $T = V[[X,\,Y,\,Z]]/(p - X^3 - Y^5 - Z^7)$.

Statements about results of Serre in this section all refer to \cite{Ser65}. Serre proved that if
$(T,\, \n, K)$ is regular local and its completion is formal power series over a field or
a DVR,    then the following hold for finitely generated nonzero modules  $M$, $N$ when
$\ell(M \otimes_TN) < \infty$
(keep in mind the case where $M = T/P$ and  $N = T/Q$ are domains, and $P+Q$ is $\n$-primary: it
implies the other cases and is closer to the geometric situation ):    \medskip

\bena
\item  $\dm(M)  +\dm(N) \leq \dm(T)$.
\item  $\dm(M)  + \dm(N) < \dm(T) \imp \chi(M,N) = 0$.
\item $\dm(M)  +\dm(N) = \dm(T) \imp \chi(M,N) > 0$.
\een

Serre also proved  that (a) holds for {\it every} regular local ring
$T$, and,   essentially, conjectured    (``Il est naturel de $\underline{\hbox{\rm conjecturer}}$")
that  (b) and (c) hold for all regular rings as well.    The remaining case is the ramified case in mixed
characteristic.   {\it It has been an open question for over fifty years.}

Serre also proved the case when either $M$ or $N$ is $T/(f_1, \, \ldots, f_h)T$ where $f_1, \, \ldots, f_h$
is a regular sequence, i.e., the case where one of the modules is a {\it complete intersection}.

\subsection{Progress on Serre's multiplicity conjecture}
\label{ss:SerreIntro}

Paul Roberts  \cite{Rob85} and, independently, H. Gillet and C. Soul\'e \cite{GS85} proved
part (b), i.e.,  $\dm(M)  + \dm(N) < \dm(T) \imp \chi(M,N) = 0$.

O.~Gabber, using  De Jong's results on alterations \cite{deJo96}, was able to prove that  $\chi(M,N) \geq 0$
in the ramified case.  There is an exposition by Berthelot \cite{Ber97} containing Gabber's
result.

The case where $\dm(T) \leq 4$ is settled in \cite{Ho73a}.  There are many results settling cases of the
conjecture and considering related conjectures in the work of S.~P.~Dutta
\cite{Du83a, Du83b, Du87, Du88, Du93, Du00, Du05}. and there are related results, both on
intersection multiplcities and on intersection theorems of various sorts in  \cite{PS74, PS76}
and \cite{Rob87, Rob89} for the mixed characteristic case.
There is a counterexample to a generalized form of the conjecture, in which just one of the
modules has finite propjective dimension, in \cite{DHM85}, where an example is given where
the Serre multiplicity is negative.

Perhaps the most tantalizing problem that remains is to prove Serre's original conjecture part (c)
on the strict positivity of  $\chi(M,N)$ in the case of ramified regular local rings.
{\it This remains open in all dimensions $\geq 5$.}   One may assume
that $T$ is complete with a perfect (or even algebraically closed) residue field.    Also, because of
the bi-additivity it suffices to prove the result when $M = T/P$ and $N = T/Q$ are prime cyclic modules.

We next want to explain the relevance
of the existence of small Cohen-Macaulay modules to Serre's multiplicity conjecture:  their
existence implies the remaining case, part (c), for ramified regular local rings.  The argument
below will be modified later, in \S\ref{limcmSerre}, to show that the small Cohen-Macaulay modules can be
replaced by lim Cohen-Macaulay sequences.

Let $P$ and $Q$ be prime ideals of a regular local ring $(T, \, \n, \, K)$ such that
$\dm(T/P) + \dm(T/Q) = \dm(T)$ and $P+Q$ is $\n$-primary.

\begin{mofact} Let $P$ and $Q$ be prime ideals of a regular local ring $(T, \, \n, \, K)$ such that
$\dm(T/P) + \dm(T/Q) = \dm(T)$ and $P+Q$ is $\n$-primary.  If $M$ is a small Cohen-Macaulay
module for $R$ and $N$ is a small Cohen-Macaulay module for $S$, then $\chi(R,\, S) > 0$.
\end{mofact}

\begin{proof}Suppose $M$ is a small Cohen-Macaulay module for $R  =T/P$
 of (torsion-free) rank $r$   and that $N$ is a small Cohen-Macaulay module for $S = T/Q$ of rank  $s$.
 $M$ has a finite filtration with $r$ factors equal to $R$ and other factors of {\it smaller} dimension.
  $N$ has a finite filtration with $s$ factors equal to $S$ and other factors of {\it smaller} dimension.
  Using the bi-additivity of $\chi$ and the fact that the vanishing part (b) of the
  conjecture holds, one obtains that $\chi(M,N) = rs\,\chi(R, S)$.    But
  when $M, N$ are Cohen-Macaulay {\it the higher Tors vanish}\footnote{Let us give a derived category argument for this higher Tor vanishing; a more explicit argument can be extracted from \S \ref{strong}. First, observe that 
  \[ M \otimes_T^L N \simeq R\Gamma_{\mathfrak{n}}(T) \otimes^L_T M \otimes^L_T N \simeq  R\Gamma_{\mathfrak{n}}(M) \otimes_T^L R\Gamma_{\mathfrak{n}}(N)\]
  as $M \otimes_T^L N$ is $\mathfrak{n}$-primary. Now  $R\Gamma_{\mathfrak{n}}(M)$ and  $R\Gamma_{\mathfrak{n}}(N)$ are concentrated in degrees $\dim(T/P)$ and $\dim(T/Q)$ respectively by the Cohen-Macaulayness assumption. As the $\mathrm{Tor}$-dimension of $T$ is $d=\dim(T/P) + \dim(T/Q)$, the right side of the above expression lies in cohomological degrees $\geq 0$; as the left side lies in cohomological degrees $\leq 0$, it follows that both must be in degree $0$.}
   and $\chi(M,N) =
  \ell(M \otimes_T N) >0 \imp rs\chi(R,S) > 0 \imp \chi(R,S)$. \end{proof}

  We hope that understanding this argument well help to motivate the much more difficult
  proof in \S\ref{limcmSerre}.

\section{Alternative notions of length and rank}\label{alt}

 In this section, we consider some alternative notions of length and rank that
 we need later.

Let $R$ be any ring and $\cV$ a family of $R$-modules.  For any $R$-module $M$,
we define $\ell_\cV(M)$, the length of $M$ with respect to $\cV$, to be 0 if $M = 0$,
to be $+\infty$ if $M$ has no finite filtration with all factors in $\cV$,  and otherwise to
be the length of a shortest filtration of $M$ with all factors isomorphic to an element
of  $\cV$.

If $\cV = \{R/m: m \in \MaxSpec(R)\} \cup \{0\}$ this is the usual notion of length.  If $I$ is an ideal
of $R$ and $\cV = \{R/J: I \inc J\}$ (up to isomorphism, the cyclic modules killed by $I$),
then this is the notion of {\it quasi-length} with respect to $I$ introduced in \cite{HH09} and studied further in
\cite{HZ18}.  In this case we shall write $\cL_I(M)$ for $\ell_{\cV}(M)$.

We note that, in general, $\ell_\cV$ is not additive on short exact sequences of
modules $0 \to A \to B \to C \to 0$ but that if $\ell_\cV (A)$ and $\ell_\cV(C)$ are both
finite then so is $\ell_\cV(B)$ and $\ell_\cV(B) \leq \ell_\cV(A) + \ell_\cV(C)$.  More
generally,  if $M$ has a finite filtration with factors $N_i$ such every $\ell_{\cV}(N_i)$ is finite,
then $\ell_{\cV}(M)$ is finite and $\ell_\cV(M) \leq \sum_i \ell_\cV(N_i)$.

We are also interested here in the notion $\fL_I(M)$, which we define to be
$\ell_\cV$  for $\cV = \{J'/J: I \inc J \inc J' \inc R\}$.  Up to isomorphism, these are the
submodules of cyclic modules killed by $I$, and this choice of $\cV$ is closed, up to isomorphism, under
taking submodules and quotient modules.

Note that if $\cV$ consists of all cyclic $R$-modules, then $\ell_\cV(M)$ is
the least number of generators of $M$, which we denote $\nu_R(M)$ or $\nu(M)$.

 If $R$ is a domain, $\cV$ consists of all ideals of $R$, and if $M$ is a finitely generated torsion-free $R$-module, then
 $\ell_\cV(M) = \rank(M)$, the torsion-free rank of $M$.

\subsection{Subadditivity}\label{subadd}  %% 2.1
Suppose that $\cV$ is closed under taking
submodules and quotient modules,
up to isomorphism.  Then if $N$ is a submodule  or quotient of module $M$ with $\ell_\cV(M)$ finite,
we have that $\ell_\cV(N) \leq \ell_\cV(M)$ and, in particular, $\ell_\cV(N)$ is finite.  Consequently, the same
is true for subquotients of $M$.

Throughout the rest of this section, let $\lambda$ be a function from $R$-modules to $\N \cup \{+\infty\}$ such that
\ben
\item If $N$ is a subquotient of $M$ then $\lambda(N) \leq \lambda(M)$.
\item If $0 \to A \to B \to C \to 0$ is exact then $\lambda(B) \leq \lambda(A) + \lambda(C)$.
\een
Note that we then have that if $A \to B \to C$ is exact at $B$, then
\benr
\item $\lambda(B) \leq \lambda(A) + \lambda(C).$
\een
Suppose that $\cF$ is a functor from $R$-modules to $R$-modules such that:
\bensp{($\dagger$)}
\item If the sequence $0 \to A \to B \to C \to 0$ is exact then there is an exact sequence
$\cF(A) \to \cF(B) \to \cF(C)$ or
$\cF(C) \to \cF(B) \to \cF(A)$,  depending on whether $F$ is covariant or contravariant
(the map on the left need not be injective nor the map on the right surjective).
\een
Functors satisfying $(\dagger)$
include  $\Tor^R_i(\blank, M)$,  $H^i_J(\blank)$, $\Ext^i_R(\blank,\, N)$
and $\Ext^i_R(N, \, \blank)$.  Then if $0 \to A \to B \to C \to 0$ is exact we have
$$\lambda\big(\cF(B)\bigr) \leq \lambda\bigl(\cF(A)\bigr)  + \lambda\bigl(\cF(C)\bigr),$$  and an
easy induction shows that if $B$ has a finite filtration with factors $B_i$,  then if all
the $\lambda\bigl(\cF(B_i)\bigr)$ are finite then so is $\lambda\bigl(\cF(B)\bigr)$ and
$$(*)\qquad \lambda\bigl(\cF(B)\bigr)  \leq \sum_i \lambda\bigl(\cF(B_i)\bigr).$$
Hence:

%% 2.1
\begin{proposition}\label{filt} Let $\lambda$ as above satisfy conditions (1) and (2).
Let $G$ be a functor  from pairs of $R$-modules to $R$-modules that satisfies
condition $(\dagger)$ as a functor of each variable when the other is held fixed.
Then if $A$ has a finite filtration with factors $A_i$ and $B$ has a finite filtration
with factors $B_j$, we have
$$(\#)\qquad \lambda\bigl(G(A,\,B)\bigr) \leq \sum_{i,j} \lambda\bigl(G(A_i,\,B_j)\bigr)$$
whenever all the terms in the sum on the right are finite.

In particular, $(\#)$ holds for $\lambda =\ell_{\cV}$ when $\cV$ is closed under taking submodules and quotients,
and, hence,  when $\lambda$ is the usual notion $\ell$ of length, and also when $\lambda$ is
$\fL_I$.
 \end{proposition}

%% 2.2
\begin{corollary}\label{ltor} If $(T, \m, K)$ is a local ring and $M, N$ are modules of finite length, then for all $t$ we have
$$\ell\big(\Tor^T_t(M,N)\big) \leq \ell(M)\ell(N)\ell\big(\Tor^T_t(K,K)\big).$$ \end{corollary}

\subsection{$\cL_I$ and $\fL_I$ for one-dimensional cases}  %% 2.2

For many one-dimensional Noetherian rings
$R$  there is a nonnegative integer
$\fN_R$ that is the greatest value of $\nu_R(I)$ for any ideal $I  \inc R$.
The existence of such a bound implies that $R$ has dimension at most one.  This is well known
in the local case \cite{Sal78, Gott93}, and the semilocal case follows easily.  In fact:

%% 2.3
\begin{proposition} Let $R$ be a Noetherian ring of dimension at most one. Suppose that one of the following
conditions holds:
\ben
\item $R$ is local.
\item $R$ is semilocal.
\item For every minimal prime $P$ of $R$,  the singular locus in $\Spec(R/P)$ contains a non-empty
open set.
\item $R$ is excellent.
\een
Then $\fN_R$ is finite, i.e., there is a finite bound for the number of generators of all ideals of $R$. \smallskip

Moreover, if $R = S/I$ is such that $\fN_R$ is finite, then for every finitely generated
 $S$-module $M$ killed by a power of $I$,
$\fL_I(M) \leq \cL_I(M) \leq \fN_R \, \fL_I(M)$.
\end{proposition}
\begin{proof} In the local case, we may complete:  $I$ and $I\wh{R}$ have the same number of generators.
$R$ has a finite filtration with factors $R/P$. This induces a filtration on each ideal of $R$ whose factors are
subquotients of the factors $R/P$.  Hence, we may assume that $R$ is a complete local domain, which
means it is a field (an obvious case) or a free module over a complete discrete valuation ring  $V$.  For any ideal $I$, $\nu_R(I) \leq \nu_V(I)$ which is at most the torsion-free rank of $R$ over $V$, since $I$ must be free over $V$.

If $R$ is semilocal with maximal ideals $\vect m k$ and $I$ is any ideal, for every $i$ we may choose at most
 $\fN_{R_{m_i}}$ elements of $I$ whose images generate $IR_{m_i}$.  These elements generate $I$, since
 that is true locally at every maximal ideal. Thus $\fN_R \leq \sum_{i=1}^k \fN_{R_{m_i}}$.

 In case (3),  $R$ has a finite filtration consisting of factors of the form $R/P$, where $P$ is a minimal prime,
 or of the form $R/m$ for $m$ maximal.  Thus, we may reduce to the case where $R$ is a field (which is
 clear) or $R$ is a one-dimensional
 domain and there is an element $a \in R \smo$  such that $R_a$ is regular, i.e., a Dedekind domain or
 field.  Let $V(a) = \{\vect m k\}$.  Given $I \inc R$,  $IR_a$ needs at  most two generators, which may
 be taken from $I$,  and for every $i$,  $IR_{m_i}$ needs at most $\fN_{R_{m_i}}$ generators, which
 may be taken to be images of elements of $I$.  Then we have at most $2 + \sum_{i=1}^k \fN_{R_{m_i}}$
 elements of $I$ that generate $I$, since these elements generate $I$ after localization at any maximal ideal.
 Condition (4) suffices since  $(4) \imp (3)$. \smallskip

 In the final statement, the inequality on the left is obvious, while
 the inequality on the right follows because $M$ has a filtration by at most $\fL_I(M)$
subquotients of $R$,  and each of these has a filtration by at most $\fN_R$ cyclic $R/I$-modules.
  \end{proof}

%% 2.4
\begin{corollary}\label{nurho} If $R$ is a domain of dimension at most one with $\fN_R$ finite,
then for every finitely generated torsion-free module $M$, $\rank(M) \leq \nu(M) \leq \fN_R\rank(M)$. \end{corollary}
\begin{proof} We may assume $M \not=0$. By tensoring with the fraction field of $R$,
we see that $\Hom_R(M,\,R) \not=0$,  and since there is a nonzero map  $M \to R$
there is a surjection $M \surj I$, where $I \inc R$ is a nonzero ideal.  Let $N$ be the kernel.
It follows by induction on the rank that $M$ has a filtration with $\rank(M)$ factors, each
of which is an ideal of $R$. \end{proof}

\subsection{Rank}\label{rank}% 2.3

If $R$ is a domain with fraction field $L$ and $M$ is an $R$-module, we define
$\rank(M) = \dim_L L \otimes_R M$, which is finite whenever $M$ is finitely generated
(although finite generation is not necessary).

When $R$ is local and not necessarily a domain, we define rank more generally for certain finitely generated modules.
Let $\MIN(R)$ denote the set of minimal primes $\fp$ of $R$ such that $\dim(R/\fp) = \dim(R)$.
We say that a finitely generated $R$-module $M$ has {\it rank} $r$ and write $r = \rank(M)$ if for every minimal prime
$\fp \in \MIN(R)$, we have that  $\ell_{R_{\fp}}(M_{\fp}) = r\,\ell_{R_{\fp}}(R_{\fp})$.  This notion is defined on a
larger class of modules than the notion used in \cite[p.~183]{BHU87}, where it is required that for every associated
prime $\fp$ of $R$ that
$M_\fp$ be $R_\fp$-free of rank $r$.  This rank agrees with the rank defined in
the paragraph above when $R$ is a local domain, and in that case is defined on all finitely generated $R$-modules $M$.

%%2.5
\begin{remark}\label{nzrk}  If $\rank(M)$ is defined for a finitely generated $R$-module $M$, it is clear that
$\rank(M) \not=0$ if and only if $\dim(M) = \dim(R)$, for that is what is needed for some $\fp \in \MIN(R)$
to be in the support of $M$. \end{remark}

%%2.6
\begin{proposition}\label{ranknu} If $R$ is a local ring and $M$ is a finitely generated $R$-module for which rank is
defined, $\rank(M) \leq \nu(M)$.  \end{proposition}
\begin{proof}  For any $\fp \in \MIN(R)$ we have $\ell_{R_\fp}(M_\fp) = r\ell(R_\fp)$, where $r = \rank(M)$.
We have a surjection $R^{\nu(M)} \surj  M$ and, hence, we have $R_{\fp}^{\nu(M)} \surj M_{\fp}$.
Consequently,  $r\ell_{R_\fp}(R_\fp) = \ell_{R_\fp}(M_{\fp})  \leq \nu(M)\ell_{R_\fp}(R_\fp)$ and so $r \leq \nu(M)$.
\end{proof}

%% 2.7
\begin{discussion}\label{mult}{\bf Hilbert-Samuel multiplicity.}\
If $\rmk$ is local, $M\not=0$ is a finitely generated $R$-module of Krull dimension $d$,  and $I$ is $\m$-primary,
then $\ell(M/I^{t+1}M)$ agrees for $t \gg 0$ with a polynomial in $t$ of degree $d$, the Hilbert-Samuel polynomial,
and its leading coefficient has the form $e/d!$, where $e$ is a positive integer, the {\it Hilbert-Samuel multiplicity}
of $M$ with respect to $I$. We often simply refer to this as the {\it multiplicity} of $M$ with respect to $I$.
If $\dim(M) \leq d$ we write $e_d(I; \, M)$ for $\disp d! \lim_{t \to \infty}\ell(M/I^{t+1}M)/t^d$, which is
an additive function on short exact sequences of finitely generated $R$-modules.  We also have that
$e_d(I; \, M)  = \disp d! \lim_{t \to \infty} \ell(M/I^tM)/t^d$,  since $(t-1)^d/t^d$ has limit $1$ as
$t \to \infty$.

Note that $e_d(I;\,M)$  is nonnegative
and is $0$ if and only if $\dim(M) < d$.  We write $e(I; \ M)$ for $e_d(I;\, M)$ when $M \not=0$ and $d = \dim(M)$.
If $I = \m$, we may write $e_d(M):= e_d(\m; \, M)$ and $e(M):= e(\m;\,M)$.
See Theorem~\ref{Koszul}(n) for a characterization using Koszul homology.  \end{discussion}

From the additivity of $e_d(I;\, \blank)$ we have at once:

%% 2.8
\begin{proposition}\label{multrank} Let $M$ be a finitely generated module of Krull dimension $d$ over the local ring
$\rmk$,  where $\dim(R) = d$.  If $M$ has rank $r$,   then for every $\m$-primary ideal  $I$ of $R$,  we have
$e(I;\,M) = r\,e(I;\,R)$. \end{proposition}

\section{Multiple Tor and Koszul homology}\label{multtor}%% 3

In this section we discuss some facts about multiple Tor, and also about the behavior of
Koszul homolgy.  We note that triple Tors are useful in setting up spectral sequences that
may be viewed as providing a kind of associativity for iterated Tor.  Such spectral
sequences originate, so far as we know, in \cite{CaEi56},  and are utilized effectively in
\cite{Ser65}, where Grothendieck is credited for some of the arguments.
The spectral sequences for quadruple Tor play an essential role in the
proof, in \S\ref{limcmSerre},  that the existence of lim \CM sequences implies the positivity
conjecture for Serre multiplicities.

If $T$ is a ring and $\cF_\bu^{(1)}, \, \ldots, \cF_\bu^{(s)}$ are left complexes of flat $T$-modules,
we denote by $\Tor^T_i(\cF_\bu^{(1)}, \, \ldots, \cF_\bu^{(s)})$ the $i\,th$ homology module
of  the total complex obtained by tensoring together the $s$ complexes $\cF_\bu^{(j)}$.
We also define this $s$-tuple $\Tor$ when some or all of the $\cF^{(j)}_\bu$ are modules instead
of complexes
by replacing each of the modules by a flat resolution of that module over $T$.  The result is independent
of the choice of flat resolution and this agrees, when $s = 2$,
with usual definition of $\Tor^T_i(M,\,N)$ for modules.  The values of the multiple
Tors are independent of the order of the $s$ complexes, up to natural isomorphism.
A subset consisting of $h$ of the input complexes may evidently be replaced by the total complex
obtained by tensoring them together.  This replaces the $s$-tuple Tor by an $(s-h+1)$-tuple Tor.
Note that if $s=1$,   $\Tor^T_i(\cF_\bu^{(1)})$ is simply  $H_i(\cF_\bu^{(1)})$.

%% 3.1
\begin{discussion}\label{specTor}{Spectral sequences for multiple Tor.}\label{mulTor}
In  the case of two flat complexes $\cF_\bu,\, \cG_\bu$ the two spectral sequences of the double
complex obtained by tensoring $\cF_\bu$ and $\cG_\bu$ together yield:
$$ (*) \quad H_i\bigl(H_j(\cF_\bu) \otimes_T \cG_\bu \bigr) \imp \Tor^T_{i+j}(\cF_\bu,\,\cG_\bu),$$
which we may also write as
$$  \quad \Tor^T_i\bigl(H_j(\cF_\bu),\, \cG_\bu\bigr) \imp  \Tor^T_{i+j}(\cF_\bu,\,\cG_\bu).$$
Here, $E^2_{ij} = \Tor^T_i\bigl(H_j(\cF_\bu),\, \cG_\bu\bigr)$, and
$d^r:E^r_{ij} \to E^r_{i-r,j+r-1}(\cF_\bu)\bigr)$ for $r \geq 2$.
Similarly,
$$ (**) \quad H_j\bigl(\cF_\bu \otimes_T H_i(G_\bu)\bigr) \imp \Tor^T_{i+j}(\cF_\bu,\,\cG_\bu)$$
which we may also write as
$$ (**) \quad \Tor^T_j\bigl(\cF_\bu,\, H_i(\cG_\bu)\bigr) \imp  \Tor^T_{i+j}(\cF_\bu,\,\cG_\bu).$$
In this case, $E^2_{ji} = \Tor^T_j\bigl(\cF_\bu,\, H_i(\cG_\bu)\bigr)$, and
$d^r:E^r_{ji} \to E^r_{j-r,i+r-1}$ for $r \geq 2$.
We may omit $T$ from the notation if it is clear from context. If there are two complexes as above
and $\cG_\bu$ is a flat resolution of $N$,  the spectral sequence $(*)$ yields
$$\Tor^T_i(\cF_\bu,\, N) \cong H_i(\cF_\bu \otimes_T N).$$
It follows that if $\cF_\bu$ is the total tensor product of  $\cF^{(1)}, \, \ldots, \cF^{(s)}$ then
$$ \Tor^T_i(\cF^{(1)}, \, \ldots, \cF^{(s)}, N) \cong H_i(\cF_\bu \otimes_T N).$$
\end{discussion}

%%3.2
\begin{remark}\label{diag}  Consider, for example, the spectral sequence of a double complex.  Suppose
that on the $E^r$ page there are finitely many nonzero terms  on the diagonal $i+j =n$ and all of them have
finite length.  Then the sum
of those lengths bounds the sum of the lengths for the diagonal $i+j = n$ when $r = \infty$, since in the
transition from $E^s$ to $E^{s+1}$ for $s \geq r$,
the terms on the diagonal $i+j = n$  for $E^{s+1}$ are subquotients of the terms on the diagonal $i+j = n$
for $E^s$. \end{remark}

We shall say that a complex is $\cF_\bu$ is in $D^{[a,b]}$ if $a \leq b$ are integers and $H^i(\cF_\bullet) = 0$ 
whenever $i \notin [a,b]$.  
%% 3.3
\begin{remark}\label{torann}  
Suppose that $\cF_\bu^{(j)}$ is in $D^{[a,b]}$ and that $f \in T$ is any element that kills the 
$H_i(\cF_\bu^{(j)})$ for all $i$. Then $f^{b-a+1}$ kills  $\Tor^T_h(\cF_\bu^{(1)}, \, \ldots, \cF_\bu^{(s)})$ for all $h$.  
We may assume that $j = 1$ and we may replace the sequence consisting of the other
flat complexes by their total tensor product.  Thus, we may assume that there 
are only two flat complexes.  The result then follows from the spectral sequence $(*)$ in Discussion~\ref{mulTor},
which yields a filtration of each $\Tor^T_h(\cF_\bu^{(1)}, \, \cF_\bu^{(2)})$ with at most $b-a+1$ factors.
\end{remark}

If $x \in T$, the Koszul complex $\cK_\bu(x; T)$ is the left complex  $0 \to T \arwf{x \cdot} T \to 0$ with
the two copies of $T$ in degrees 0 and 1.
Let  $\ux = \vect x d \in T$. The Koszul complex $\cK_\bu(\ux;\, M)$ may be defined as
the total tensor product of the complexes $\cK_\bu(x_j;\, R)$ tensored with the module
$M$,  and the Koszul homology $H_i(\ux;\, M)$  is $H_i\bigl(\cK_\bu(\ux;\,M)\bigr)$.
Evidently, if  $\cK^{(j)}_\bu$ denotes $\cK_\bu(x_j; R)$ then
$$ H_i(\ux;\, M) \cong \Tor_i(\cK^{(1)}_\bu, \, \ldots, \cK^{(d)}_\bu, \, M).$$

The following result summarizes some well-known properties of Koszul homology.
We refer the reader to \cite{Ser65}, \cite{BruH93}, and \cite{Licht66} for complete proofs, although
we have indicated some of the arguments if they are brief. We first make the
following definition:   if $H_j(\ux; M)$ has finite length for $j \geq i$,
then $\chi_j(\ux;\, M) = \sum_{j=i}^d (-1)^{j-i} \ell\bigl(H_j(\ux;\,M)\bigr)$ and  $\chi(\ux;\, M)
= \chi_0(\ux;\, M)$.  If it is necessary to indicate the ring $T$, it may be used as a superscript,
e.g., one may write $\chi^T(\ux;\,M)$.

%% 3.4
\begin{theorem}[Properties of Koszul homology.]\label{Koszul} Let $T$ be a ring, let
$\ux = \vect x d$ $\in T$, and let $M'$, $M$, $M''$ be $T$-modules. If $d \geq 1$ let
$\ux^{-} = \vect x {d-1}$.   Let $I$ be the ideal $(\ux)T$.
\bena
\item $H_0(\ux; M) \cong M/IM.$   %%a
\item $H_d(\ux; M) \cong \Ann_M I.$  %%b
\item There is a long exact sequence $$\cdots \to H_i(\ux^-;\, M)\arwf{\pm x_i \cdot}  H_i(\ux^-;\, M)
\to  H_i(\ux;\, M) \to \qquad\qquad\qquad$$
$$\qquad\qquad\qquad\qquad H_{i-1}(\ux^-;\, M) \arwf{\mp x_i \cdot}  H_{i-1}(\ux^-;\, M) \to \cdots\,.$$ %%c
\item For all $i$,  there is a short exact sequence
$$0 \to \frac{H_i(\ux^-;\, M)}{x_dH_i(\ux^-;\, M)}  \to H_i(\ux;\,M) \to \Ann_{H_{i-1}(\ux^-;\,M)} x_d \to 0.$$ %%d
\item If $\vect x d$  is a possibly improper regular sequence on $M$, then $H_i(\ux;\,M) = 0$ for all $i \geq 1$. %%e
\item Let $\theta:A \to T$ be any ring homomorphism and $\vect Xd \in A$ be  such
that $\vect Xd$ is a regular sequence in $A$ and $X_i \mapsto x_i$, $1 \leq i \leq d$.  For example
if $\theta_0:\Lambda \to T$ is any ring homomorphism
 (e.g., we may always take $\Lambda = \Z$ or
$\Lambda = T$ with $\theta_0$ the identity map), we may let $A$ be the polynomial ring $\Lambda[\vect X d]$
in $d$ variables over $\Lambda$, and extend $\theta_0$ to $\theta:A \to T$ such that
$X_j \mapsto x_j$,  $1 \leq j \leq d$.  Let $\ov{A}$ denote that $A$-module $A/(\vect X d)A$,
which is $\Lambda$ when $A$ has the form $\Lambda[\vect X d]$.  $M$ is an $A$-module via restriction of scalars
using $\theta:A \to T$.
Then $$H_i(\ux; \,M) \cong H_i(\vect X d;\, M) \cong \Tor^A_i(\ov{A}, \, M).$$  %%f
\item $I$ and $\Ann_TM$ kill every  $H_i(\ux;\, M)$. %%g
\item If $0 \to M' \to M \to M'' \to 0$ is exact there is a long exact sequence
$$\cdots \to H_i(\ux;\,M') \to H_i(\ux;\,M) \to  H_i(\ux;\,M'') \to  H_{i-1}(\ux;\,M') \to \cdots\,.$$%%h
\item If $y,\,z \in T$ then there is a long exact sequence
$$ \cdots \to H_i(\ux, \, y;\, M) \to H_i(\ux, \, yz;\, M) \to H_i(\ux,\, z;\, M) \to \cdots\,.$$  %%i
\item Let $\uy$ be the image of  $\ux$ under multiplication by an invertible matrix over $T$.  Then
$H_i(\uy;\, M) \cong H_i(\ux; \, M)$. Hence, the Koszul homology is unchanged, up to isomorphism,
by permuting the $x_i$, multiplying them by units, or adding a sum of multiples of the $x_j$ for
$j \not= i$ to $x_i$.   If $R$ is a local ring, $H_i(\ux;\, M)$ depends, up to
isomorphism, only on the ideal $I$ and the number of generators $d$ used for $I$, and not on
the specific choice of generators.  %%j
\item If $Q$ is a flat $T$-module (or $T$-algebra), $H_i(\ux; Q \otimes_TM) \cong Q \otimes_T H_i(\ux;\, M)$
In particular, calculation of Koszul homology commutes with localization and, when $T$ is local
and $M$ is finitely generated, with completion. %%k
\item Let $(T,\,\m,\,K)$ be local, $\vect x s \in \m$, and $M$ finitely generated,
or let  $T$ be $\N$-graded, $\vect x s$ forms of positive degree, and let $M$ be $\Z$-graded
with all degrees of nonzero components bounded below by a fixed integer.  Then $\vect xs$ is a regular
sequence if and only $H_1(\ux;\,M) = 0$, in which case all $H_i(\ux;\,M) =0$ for $i \geq 1$. %% l
\item Let $T$ be Noetherian and let $M$ be finitely generated.  If $H_i(\ux;\,M) = 0$ (respectively,
has support contained in a closed set $X$, respectively has finite length), then so does
$H_j(\ux;\,M)$ for all $j \geq i$. %%m
\item Let $(T,\,\m, \, K)$ be local, let $M$ be finite generated, and let
$\ux$ be a \sop\ for $T$.  Then  $\chi(\ux;\,M) = 0$ if $\dm(M) < \dm(T)$, while $\chi(\ux;\, M)$
is the Hilbert-Samuel multiplicity\footnote{See Discussion~\ref{mult}.}  $e\big((\ux);\,M\big)$ if $\dm(M) = \dm(T)$.   %%n
\item Let $(T,\,\m, \, K)$ be local, let $\ux$ be a system of parameters, and let $M$ be finitely generated.
Then for $i \geq 1$,  $\chi_i(M) \geq 0$,
and $\chi_i(M) = 0$ if and only $H_j(\ux;\, M) =  0$ for all $j \geq r$.  Moreover, $M$ is \CM if and only if
$\chi_1(\ux;\, M) = 0$ if and only if $\ell\big(H_0(\ux; \, M) \big)= e\big((\ux); \, M\big)$.    %%o
\item There is a long exact sequence:
$$
\cdots \to H_{i-1}(\ux^-; \Ann_M x_d) \to H_i(\ux;\,M) \to \qquad\qquad$$
$$\qquad\qquad\qquad\qquad\qquad H_i(\ux^-;\, M/x_dM)  \to H_{i-2}(\ux^-;\, \Ann_Mx_d) \to \cdots \, .
$$  %% (p)
\item If the first $k$ of the elements $\ux$ form a (possibly improper) regular sequence on $M$,
then for all $i$,  $H_i(\ux; M) \cong H_i\big(x_{k+1}, \ldots, x_d; \, M/(\vect x k)M\big)$. %% (q)
\een
\end{theorem}

\begin{proof}  For more detailed treatments we refer the reader to \cite{Ser65}, \cite{BruH93}, and
\cite{Licht66} but we make some remarks here.

 (a) and (b) are immediate from the definition. Part (c) is
straightforward from the two spectral sequences of the double complex obtained by tensoring
$\cK_\bu(\ux^-; T)$ with $\cK_\bu(x_d; \, M)$ (which is also the mapping cone of multiplication
by $x_d$ mapping $\cK_\bu(\ux^-; \, M)$ to itself), and (d) is immediate from (c).
Part (e) follows easily from (c) using mathematical
induction on $d$.  (f) follows from (e) because $\cK_\bu(\vect X d;\, A)$ is free resolution of $\ov{A}$
over $A$,  and when we apply $\blank\otimes_A M$ we get $\cK(\ux;\, M)$, since the action of every
$X_i$ on $M$ is the same as the action of $x_i$. (g) and (h) then follow from corresponding
properties of Tor.

To prove (i),  let  $A = \Z[\UX, \, Y, \, Z]$
be polynomial, where $\UX = \vect X d$, and map this ring
to $T$ so that $\UX, \, Y, \, Z \mapsto \ux,\, y, \, z$.  The result follows from
the fact that $\UX, V$  is a regular sequence in $A$ when $V$ is any of the three
elements, $Y$, $YZ$, or $Z$, the fact that the sequence
$$0 \to A/(\UX, \, Y) \overset{Z\cdot}{\longrightarrow} A/(\UX^-, \, YZ) \to A/(\UX^-,\, Z) \to 0$$
is exact, the long exact sequence for Tor, and part (f).

Part (j) follows from the characterization of $K(\ux;\, T)$ as the exterior algebra
$\bigwedge^\bu\bigl(\cK_1(\ux;\, T)\bigr)$ such that the differential $d_i$ is the unique extension
of $d_1$ to a derivation (in the sense that if $u$, $v$ are forms,
$d(uv) = (du)\wedge v + (-1)^{\mathrm{deg}(u)}u\wedge dv$:  see \cite{Ser65}.
The matrix $A$ induces an isomorphism of exterior algebras compatible with differentials by letting the map in degree
$i$ be $\bigwedge^i(A)$.  In the local case, there is an invertible matrix that sends any given set of
$d$ generators for $I$ to any other set of $d$ generators for $i$.

Part (k) is obvious from the definitions, and part (l) follows easily by induction from part (c) and
Nakayama's lemma:  cf.~\cite{Ser65}. For parts (m), (n), and (o) we refer to \cite{Ser65}. One can
prove (p) by viewing the Koszul homology as Tor as in part (f) and using the
spectral sequence for change of rings for Tor: we refer to
\cite{Ser65} and \cite{Licht66} for more detail, as well as to \cite{Ho81} for an application to intersection
theory in a hypersurface.  Part (q) reduces to the case where $k=1$ by
induction on $k$, and the case $k=1$ follows from part (p) if we take $x_d$ to be the nonzerodivisor (by
part (j),  the order of the elements $\ux$ does not matter).
\end{proof}

%% 3.5
\begin{corollary}\label{Koszlength}  Let $\ux, \, y, \, z \in T$ and let $M$ be a $T$-module.
\bena
\item  If $H_i(\ux,\, y;\, M)$  and $H_i(\ux,\, z; \,M)$ have finite lengths , then
so does $H(\ux,\, yz;\, M)$, and
$$\ell\bigl(H_i(\ux,\, yz; \,M)\bigr) \leq \ell\bigl(H_i(\ux,\, y;\, M)\bigr)+
\ell\bigl(H_i(\ux,\, z;\, M)\bigr).$$
\item If $H_i(\ux; \, M)$ has finite length, this remains true when each $x_i$ is replaced by some power,
and $$\ell\bigl(H_i(x_1^{t_1},\, \ldots, \, x_d^{t_d}; M)\bigr) \leq t_1\cdots t_d \ell\bigl(H_i(\ux;\,M)\bigr).$$
In particular $$\ell\bigl(H_i(x_1^t,\, \ldots, \, x_d^t; M)\bigr) \leq t^d \ell\bigl(H_i(\ux;\,M)\bigr).$$
\een
\end{corollary}
\begin{proof} Part (a) is immediate from Theorem~\ref{Koszul}(i) and the discussion in
\ref{subadd}, while part (b) simply follows from iterated application of part (a). \end{proof}

We need the following result in the proof  in \S\ref{limcmSerre} that the existence of \lCM\ sequences implies the Serre conjecture on positivity of Tor.

%%3.6
\begin{lemma}\label{Koszsurj} Let $\cG_\bu$ be a flat left complex in $D^{[a,b]}$.
Suppose that $\vect y s \in T$ kill
all the modules $H_i(\cG_\bu)$. Set $x_i={y_i^{b-a+1}}$. Then there is a surjection
$$\Tor_{i+s}(\cK_\bu(x_1;\,T),\, \ldots, \,\cK_\bu(x_s; T), \cG_\bu) \surj H_i(\cG_\bu).$$  
If $\vect x s$
is, in addition, a regular sequence in $T$ and $I = (\vect x s)T$, then there is a surjection
$\Tor^T_{i+s}\big(T/I,\, \cG) \surj H_i(\cG_\bu)$. \end{lemma}
\begin{proof}  We prove the first statement by induction on $s$.  First suppose that
$s=1$, and let $x = x_1$.  Then $\Tor_{i+1}(\cK_\bu(x;\,T),\, \cG_\bu)$ is the degree $i+1$ homology
of the mapping cone $\cC_\bu$ of $\cG_\bu \overset{x\cdot}{\longrightarrow} \cG_\bu$, i.e., the total complex of:
$$\CD
\cdots @>d>> \cG_{i+1} @>d>> \cG_i @>d>> \cG_{i-1} @>d>> \cdots\\
     @.           @V{x}VV           @V{x}VV           @V{x}VV      @.\\
\cdots @>d>> \cG_{i+1} @>d>> \cG_i @>d>> \cG_{i-1} @>d>> \cdots\\
  \endCD$$
The bottom row is a subcomplex, and we use the map which kills this subcomplex: $\cC_\bu/\cG_\bu
\cong \cG_{\bu-1}$,  i.e., there is a degree shift down by one.  The quotient map $\cC_\bu \to \cG_{\bu-1}$
yields, in degree $i+1$, maps of homology
$\Tor_{i+1}(\cK_\bu(x;\,T),\, \cG_\bu) \to H_i(\cG_\bu)$.  Given a cycle $z$ in $\cG_i$,
we have that $xz$ is a boundary, since $x$ kills $H_i(G_\bu)$, and we can choose
$v \in \cG_{i+1}$ such that  $dv = xz$.  Then for a suitable choice of sign,  $\pm v \oplus  z
\in G_{i+1} \oplus G_i$  is an $(i+1)$-cycle in the total complex whose class in
$\Tor_{i+1}(\cK_\bu(x;\,T),\, \cG_\bu)$ maps to the class of $z$ in $H_i(G_\bu)$.

We can now prove the general case by induction on $s$.  If $s >1$, the induction hypothesis
yields a surjection $$\pi:\Tor_{i+s-1}(\cK_\bu(x_2;\,T),\, \ldots,\, \cK_\bu(x_s; T), \cG_\bu) \surj H_i(\cG_\bu),$$
and the left hand side is the homology of the flat complex $\widetilde{\cG}$ obtained
by taking the total tensor product over $T$ of $\cK_\bu(x_2;\,T), \, \ldots, \, \cK_\bu(x_s; T)$, and
$\cG_\bu$.  By Remark~\ref{torann},  $x_1$ kills the homology of $\widetilde{\cG}$.  The case
where $s = 1$ yields a surjection  $\Tor_{i+s-1+1}(\cK(x_1;\, T), \widetilde{\cG}) \surj
H_{i+s-1}(\widetilde{\cG})$, and we may compose with the surjection $\pi$ to obtain
the required surjection.

The final statement follows because when $\vect x s$ is a regular sequence in $T$, the total tensor product
of the complexes $\cK_\bu(x_j;\, T)$ is a free resolution of $T/I$.
\end{proof}

\begin{remark}
Let us explain a derived perspective on obtaining the map in the second part of Lemma~\ref{Koszsurj}. Assume $y_1,...,y_s$ is a regular sequence on $T$. Then the lemma essentially proves the following: for any $G$ in the derived category $D^b(T)$ with the property that $(y_1,\dots,y_s) H_*(G) = 0$, the natural map
\[ \alpha:\mathrm{RHom}_T(T/I, G) \to \mathrm{RHom}_T(T,G) \simeq G\]
induces a surjection on all cohomology groups, where $I=(x_1,...,x_s)T$ and the $x_i$'s are large powers of $y_i$ (in fact, one only needs to choose $x_i$ so that $x_i$ annihilates $\Hom_D(G, G)$).  To connect this to the the lemma, observe that since $T/I$ is a perfect complex, the left hand side above is identified with $(T/I)^\vee \otimes_T G$. But $(T/I)^\vee = (T/I)[-n]$ by self-duality of the Koszul complex, so the left side above is simply $T/I \otimes_T^L G[-n]$, whence the induced map $H^*(\alpha)$ coincides with a map of the shape appearing in lemma; we leave it to the reader to check that the maps agree. 
\end{remark}

\section{Lim Cohen-Macaulay sequences}\label{limdef} %%4

Throughout this section, $\rmk$ denotes a local ring of Krull dimension $d$.

%%4.1
\begin{definition}
If $\ux = \vect x d$ is a \sop\ for the local ring $\rmk$ and $M$ is a finitely generated
$R$-module, we shall write $h_i(\ux; \, M)$ for $\ell\bigl(H_i(\ux;\, M)\bigr)$.  We write
$\sigma_i(\ux;\,M) = \sum_{j=i}^d h_i(\ux; M)$.  \end{definition}

Although we give a general definition, we are most interested in the case where
$R$ is a domain, which suffices for applications to positivity of Serre multiplicities and existence of
big \CM modules and algebras.

%%4.2
\begin{definition}\label{limcmdef} A sequence of modules $\cM = \{M_n\}_n$ is {\it lim Cohen-Macaulay}
if there exists a system of parameters $\ux$ such that for all $i \geq 1$,  $h_i(\ux; M_n) = \mathrm{o}\big(\nu(M_n)\bigr)$.
If, moreover, each $M_n$ is a module-finite $R$-algebra, we call $\cM$ a {\it lim \CM sequence of $R$-algebras}.
\end{definition}

%%4.3
\begin{remark} An obviously equivalent condition is that $\sigma_1(M_n) = \oo\big(\nu(M_n)\bigr)$.
\end{remark}

%%4.4
\begin{remark}  In the expository manuscript \cite{Ho17}, for the sake of simplicity the definition of lim \CM sequence
of modules is only given in the case where $R$ is a local domain.  We  do not make this restriction here.  However, for the main applications we have in mind, it would suffice to have the existence of lim \CM sequences for complete local
domains with algebraically closed residue field.  Moreover, by  Propostion~\ref{limcmres} below,
to construct a lim \CM sequence over $R$,
it suffices to construct such a sequence over the domain $R/\fp$,  where $\fp$ is a minimal prime ideal
of $R$ such that $\dim(R/\fp) = \dim(R)$. \end{remark}

%%4.5
\begin{remark}\label{net}  We can make the same definition for a {\it net} of modules, i.e., for a
family of  modules $M_\lambda$ each of which haas the same Krull dimension as the ring $R$ indexed
 by a directed set $\Lambda$, i.e., a set
with a {\it preorder} (transitive, reflexive binary relation) such that any two elements have
a common upper bound.  Given a net $a_\lambda$ of real numbers, $\lim_\Lambda a_\lambda = 0$
means that for all $\epsilon >0$ there exists $\lambda \in \Lambda$ such that for all
$\mu \geq \lambda$, $|a_\mu| < \epsilon$.  Hence,  $f(\lambda) = \OO\bigl(g(\lambda)\bigr)$
and $f(\lambda) = \oo\bigl(g(\lambda)\bigr)$ both have meanings in the more general context.
Without exception, the results on sequences of modules in this paper are valid for nets, with no
essential changes in the arguments. However, at this point we have no applications
of the more general notion except that it can be used to define integral closure of ideals
for a local domain:  see \S\ref{op}, Example~\ref{intclop}. \end{remark}

It turns out that if the condition in Definition~\ref{limcmdef} holds for one system of parameters
then it holds for every \sop. This follows at once from the following result, which is one of the main results of
this section.

%%4.6
\begin{theorem} Let $\rmk$ be a local ring and let $\ux = \vect xd$ and $\uy = \vect y d$
be two systems of parameters for $R$.  Then there exist positive constants $C,\, C'$ independent
of $i$, $M$
such that for all $i \leq 1 \leq d$ and for every $R$-module $M$,  $\sigma_i(\ux;\, M) \leq
C\sigma_i(\uy; \,M)$ and $\sigma_i(\uy; \,M) \leq C'\sigma_i(\ux; \,M)$.  \end{theorem}

\begin{proof} For $i > d$ we may use any constants, and, using reverse induction on $i$,
it suffices to show for $1 \leq i \leq d$ that if we have such constants for all $j > i$ then
we also get such constants for $\sigma_i$.  We therefore assume the result for all
$j >i$.  It then suffices to show that there exist positive constants $C$ and $C'$ such that
for all $M$,  $h_i(\ux;\, M) \leq C\sigma_i(\uy;\, M)$ and $h_i(\uy;\, M) \leq C'\sigma_i(\ux;\, M)$.
We can construct a finite chain of systems of parameters of length $2d-1$ joining $\ux$ to $\uy$ such that
any two consecutive elements overlap in $d-1$ elements.  It therefore suffices to consider
the case  where $\uy = \ux^-,\,y$,  where $\ux^- = \vect x {d-1}$, and it suffices to prove the existence of
$C$ such that $h_i(\uy;\,M) \leq C\sigma_i(\ux;\,M)$ for all $M$.  Note that $C^{2d-1}$ can then be used in
place of $C$ when there is a chain of systems of parameters of length $2d-1$ from one \sop to another.

Let $x = x_d$.
We can choose $s$ such that $x^s \in (\uy)R$, and so $x^s = yz + w$,  where
$w \in (\ux^-)R$.  Then  $$ (*) \quad H_i(\ux^-, yz;\, M) \cong H_i(\ux^-, yz+w;\, M) = H_i(\ux^-,\,x^s;\,M),$$
where the equality on the left follows from Theorem~\ref{Koszul}(j)
and the length of the latter is at most $sh_i(\ux;\, M)$,  by Corollary~\ref{Koszlength}(b).   Thus,
$$(**)\quad h_i(\ux^-, yz;\, M) \leq sh_i(\ux;\, M).$$

Theorem~\ref{Koszul}(i)  yields a long exact sequence
$$ (\dagger)\quad   \cdots \to H_{i+1}(\ux^-,z; M) \to  H_i(\uy; M)
\overset{\alpha}{\longrightarrow} H_i(\ux^-,yz; M) \to \cdots.$$
It follows from $(\dagger)$ that 
$$ h_i(\uy;\,M) \leq h_{i+1}(\ux^-,z;\,M) + h_i(\ux^-,yz; M) \leq \sigma_{i+1}(\ux^-,z;\,M) + h_i(\ux^-,yz; M).$$
By the induction hypothesis we have constant $C_0$ such that, independent of $M$, $\sigma_{i+1}(\ux^-,z;\,M) \leq C_0\sigma_{i+1}(\ux;\,M)$. Combining this with $(**)$, we have that 
$$h_i(\uy;\,M) \leq C_0\sigma_{i+1}(\ux;\,M) + sh_i(\ux;\, M) \leq C \sigma_{i}(\ux;\,M)$$
where $C=\max\{C_0, s\}$ is independent of $M$ as required.
\begin{comment}
Let $N = \Img(\alpha)$.
It follows that $h_i(\uy;\,M)$ is the  sum of $\ell(N)$ and alternating sum of lengths
of the at most $3(d-i)$  terms strictly to left of $H_i(\uy; M)$ in $(\dagger)$.
We have that $\ell(N) \leq h_i(\ux^-,yz; M)$,
and the sum of the absolute values of the other $3(d-i)$ terms is
$$ (\#)\quad \sigma_{i+1}(\ux^-,z;\,M) + \sigma_{i+1}(\ux^-,yz;M) + \sigma_{i+1}(\uy; M).$$
By the induction hypothesis we have constants $C_1, \, C_2, \, C_3 > 0$ such that, independent of
$M$,  the summands in $(\#)$ are less than the respective values of $C_\nu \sigma_{i+1}(\ux;\,M)$
for $\nu = 1, 2, 3$.  Combining the information in $(**)$ and $(\#)$, we have that
$$h_i(\uy;\,M) \leq (C_1+C_2+C_3)\sigma_{i+1}(\ux;\,M) + sh_i(\ux;\, M),$$ independent of $M$.
as required.
\end{comment}
\end{proof}

\subsection{Weakly lim \CM sequences of modules}\label{wklimdef}%%4.1
If $\rmk$ is local of dimension $d$ with \sop $\ux =
\vect x d$, recall tht $\chi_1(\ux; \, M) = \sum_{i=1}^d (-1)^{i-1} \ell\big(H_i(\ux;\, M)\big)$. We have at once:

%% 4.7
\begin{proposition}\label{wlcm}  Let $\rmk$ be local of Krull dimension $d$ and let $\{M_n\}_n$ be a lim \CM sequence of
$R$-modules of Krull dimension $d$.  Then for every \sop $\ux$ of $R$,
$\disp\lim_{n \to \infty} \frac{\chi_1(M_n)} {\nu(M_n)} = 0.$
\end{proposition}

We are therefore led to define a {\it weakly lim \CM} sequence of modules over the local ring $\rmk$ of
Krull dimension $d$ to be a sequence of finitely generated $R$-modules $\{M_n\}_n$ of Krull dimension
$d$ such that for every \sop $\ux$ of $R$, $\disp \lim_{n \to \infty} \frac{\chi_1(M_n)} {\nu(M_n)} = 0.$

Then we may restate Proposition~\ref{wlcm} as follows:

%% 4.8
\begin{proposition}  Every lim \CM sequence of modules over a local ring $R$ is weakly lim \CM
over $R$.  \end{proposition}

We note the following, but do not give a proof because it is proved in \cite[2.6]{Ma23}.

%% 4.9
\begin{theorem} If $\rmk$ is local and $\{M_n\}_n$ is a sequence of modules such that
for one \sop $\ux$ of $R$,  $\disp \lim_{n \to \infty} \frac{\chi_1(\ux; \,M_n)} {\nu(M_n)} = 0$,
then for every \sop $\uy$ for $R$, $\disp\lim_{n \to \infty} \frac{\chi_1(\uy; \,M_n)} {\nu(M_n)} = 0,$
i.e., $\{M_n\}_n$ is a weakly lim \CM. \end{theorem}

We also note the following result, which gives an analogue of the behavior of maximal Cohen-Macaulay modules.

%%4.10
\begin{proposition}\label{limcmres}
Let $\rmk\to (S,\n,L)$ be a module-finite local map such that both rings have Krull dimension $d$ and
$\m S$ is primary to $\n$.
\bena
\item For every finite length $S$-module $N$,   $\ell_R(N) = [L:K]\ell_S(N)$. %% (a)
\item For every finitely generated $S$-module $M$,  $\nu_S(M) \leq \nu_R(M) \leq \nu_R(S)\nu_S(M)$. %% (b)
\item  A  sequence $\{M_n\}$ of finitely generated modules over
$S$ is (weakly) lim \CM iff it is (weakly) lim \CM when
considered as a sequence of $R$-modules by restriction of scalars.
In particular, if $\fp$ is a minimal prime of $R$ such $\dim(R/\fp) = \dim(R)$, a (weakly)  lim \CM sequence over
$R/\fp$ is a (weakly)  lim \CM sequence over $R$.
\een
\end{proposition}
\begin{proof} Parts (a) and (b) are straightforward and well known.  They imply (c) at once. \end{proof}

The following remarks and propositions may help to give some feeling for what it means to be
a (weakly)  lim \CM sequence.

%% 4.11
\begin{remark} Of course, if $R$ is local and $M$ is a fixed finitely generated $R$-module,  the constant
sequence $M, \, M, \,M, \,\ldots, \, M,\, \ldots$ is (weakly) lim \CM if and only if $M$ is a maximal \CM module over $R$.
See Theorem~\ref{Koszul}, parts (l) and (o).
\end{remark}

%% 4.12
\begin{proposition}  Let $\rmk$ be local of Krull dimension $d$.    Let $\{M_n\}_n$ be a sequence
of $R$-modules of Krull dimension $d$.
\bena
\item If $R$ is regular and $\beta_i(M)$ denotes the $i\,$th Betti number of $M$ (the rank of the
$i\,$th $R$-free module in a minimal $R$-free resolution of $M$) , then $\{M_n\}_n$ is lim \CM over $R$ if and only
if for $1 \leq i \leq d$,  $\disp \lim_{n \to \infty} \frac{\beta_i(M_n)} {\beta_0(M_n)} = 0$. Note
that $\beta_0(M_n) = \nu_R(M_n)$.
\item If $\uy = \vect y k \in R$ are part of a \sop\ for $R$ and form a regular sequence on every $M_n$, then
$\{M_n\}_n$ is (weaakly) lim \CM over $R$ if and only if $\{M_n/(\uy)M_n\}_n$ is (weakly) lim \CM over $R/(\uy)R$.
\item If $\rmk \to (S,\,\n,\, L)$ is flat local and the closed fiber $S/\m S$ is \CM,  then $\{M_n\}_n$
is (weakly) lim \CM over $R$ iff the sequence $\{S \otimes_R M_n\}_n$ is (weakly)  lim \CM over $S$. In particular,
this holds when $R \to S$ is flat local and $\dim(S/\m S) = 0$, so that $\{M_n\}_n$ is lim \CM over $R$ iff
$\{\wh{M}_n\}_n$ is lim \CM over $\wh{R}$.
\een
\end{proposition}
\begin{proof} (a) Let $\ux = \vect x d$ be a regular \sop for $R$, so that $(\ux)R = \m$.  Then $\cK(\ux;\,R)$
is a free resolution of $K = R/(\ux)$,  and $\beta_i(M)$ is the dimension as a $K$-vector space of
$\Tor^R_i(K,M) \cong H_i(\ux; M)$.  The result follows at once.

(b) This follows because $\nu_R(M_n/(\ux)M_n) = \nu_R(M)$ (both are $\dim_K(M_n/\m M_n)$) while
if we extend $\ux$ to a full \sop $\ux,\uy$ for $R$ we have $H_i(\ux,\uy; M) \cong H_i(\uy;\, M/(\ux)M$
for all $i \geq 0$, by Theorem~\ref{Koszul}(q).

(c) Let $\uy$ be a sequence of elements in $S$ whose images give a \sop for $S/\m S$, and let
$\ux$ be a \sop in $\m$. The $\ux, \, \uy$ is a \sop for $S$
and  $\uy$ is a regular sequence on $S \otimes_R M$ for every $R$-module
$M$: see \cite[Lemma 7.10]{HH94b} and \cite[Corollary (20.F)]{Mat70}. By part (b), we have that
$H_i(\uy,\, \ux; S \otimes_R M) \cong H_i(\ux; (S/\uy) \otimes_R M)$.

 Moreover,
$R \to S/(\uy)S$ is flat local.  Thus, we may replace $S$ by $S/(\uy)S$ and assume that the closed
fiber of $R \to S$ has dimension 0.  In this case,  $H_i(\ux; S \otimes_R M) \cong S \otimes H_i(\ux; \, M)$,
and for any finite length $R$-module $N$,  $\ell_S(S \otimes N) = \ell_S(S/\m S)\ell_R(N)$, while
for every finitely generated $R$-module $M$, $\nu_S(S \otimes_R M) = \nu_R(M)$.  The result
is now clear. \end{proof}

%% 4.13
\begin{remark}\label{ineq} Before stating the next result, we recall that $H_0(\ux; M) \cong M/(\ux)M$ and note
the following.  When $\ux$ is a \sop for the local ring $\rmk$ and $M$ is a finitely generated $R$-module
we always have the inequalities  $\nu(M) \leq \ell\big(H_0(\ux;\,M)\big) \leq \ell\big(R/(\ux)\big) \nu(M)$. The inequality to the left is obvious, since  $\nu(M) = \ell(M/\m M)$, while the inequality to the right follows because there is a surjection
 $R^{\nu(M)} \surj M$ and we may tensor with $R/(\ux)R$. \end{remark}

%% 4.14
\begin{proposition}\label{HO} Let $\rmk$ be local of Krull dimension $d$, and let $\{M_n\}_n$ be a sequence
of finitely generated $R$-modules of Krull dimension $d$.  Then $\{M_n\}_n$ is lim \CM (respectively,
weakly lim \CM) if and only if for some (equivalently, every) \sop $\ux$ for $R$,
 $$\mathrm{for\ all\ } i \geq 1,\, \lim_{n\to \infty} \frac{\ell\big(H_i(\ux; \, M_n)\big)} {\ell\big(H_0(\ux; \, M_n)\big)} \to 0\
 (\mathrm{respectively},\, \lim_{n\to \infty} \frac{\ell\big(\chi_1(\ux; \, M_n)\big)} {\ell\big(H_0(\ux; \, M_n)\big)} \to 0).
 $$\smallskip

 Moreover,  $\{M_n\}_n$ is weakly lim \CM if and only if for some (equivalently, every) \sop
$\ux$ for $R$,  $\disp \lim_{n \to \infty} \frac{e\big((\ux), M_n)} {\ell\big(H_0(\ux; \, M_n)\big)} = 1$.
 \end{proposition}
 \begin{proof}   The characterizations of (weakly) lim
 \CM sequences in the first part of the proposition in which $\nu(M_n)$ is
 replaced by $\ell\big(H_0(\ux; \, M)\big) = \ell\big(M/(\ux)M\big)$ in the
 denominators are immediate from Remark~\ref{ineq}.

The final characterization of weakly lim \CM sequences then follows at once from the first part of this
proposition and the fact that $e(\ux; M) = H_0(\ux; \, M) - \chi_1(\ux; \, M)$, by Theorem~\ref{Koszul}.
\end{proof}

%% 4.15
\begin{example}  A weakly lim \CM sequence need not be lim \CM.  Here is an example, also given
in \cite[\S10]{Ho17}.
Let $R = K[[x,y]]$,  and let  $M_n := R^n \oplus R/\m^{n^n}$. This is a weakly lim Cohen-Macaulay sequence
since the multiplicity of $M_n$ with respect to $\m = (x,y)R$ is $n$, the length of
$M_n/(x,y)M_n$ is $n+1$,   but the length of  $H_2(x,y; M_n) \cong \m^{n^n-1}/\m^{n^n}$ is enormous compared to
$\nu(M_n) = n+1$. \end{example}

When $\{M_n\}_n$ is a sequence of modules over a local ring $R$ of dimension $d$
such that every $M_n$ has Krull dimension $d$ and the rank\footnote{See Subsection \ref{rank}.} of $M_n$ is defined,
which is always the case when $R$ is a domain, we prove below (see Theorem~\ref{rkchar})
that we can replace $\nu(M_n)$ by
$\rank(M_n)$ in the definitions of lim \CM and weakly lim \CM sequence.

We first observe:

%% 4.16
\begin{lemma}\label{Cbound} If $\rmk$ is a local ring  of Krull dimension $d$ and $\{M_n\}_n$ is a weakly lim \CM
sequence of $R$-modules for which rank is defined (it is then nonzero since every $M_n$ has dimension $d$), then
$$(\dagger)\quad\lim_{n\to \infty} \frac {\ell\big(H_0(\ux; M)\big)}  {\rank(M_n)} =  e(\ux; \, R).$$

Hence, for every \sop\ $\ux$ for $R$ there is a positive real constant $C_{\ux}$ such that for all $n$,
$$ \quad1 \leq \frac {\nu(M_n)}{\rank(M_n)} \leq \frac {\ell\big(H_0(\ux; M_n)\big)}{\rank(M_n)} \leq C_{\ux}.$$
\end{lemma}
\begin{proof} By Proposition~\ref{multrank},  $\rank(M_n) = e(\ux; M_n)/e(\ux; \, R)$.
When we substitute this into the denominator of the fraction on the left in the first statement of the lemma,
we see that the result follows from the final statement in Proposition~\ref{HO}.

The existence of $C_{\ux}$
such that  $\disp\frac {\ell\big(H_0(\ux; M_n)\big)}{\rank(M_n)} \leq C_{\ux}$ for all $n$
follows from $(\dagger)$. But, for all $n$,  $\rank(M_n) \leq \nu(M_n) \leq \ell\big(H_0(\ux;\, M_n)\big)$. \end{proof}

%% 4.17
\begin{theorem}\label{rkchar} Let $\rmk$ be a local ring of Krull dimension $d$ and let
$\{M_n\}_n$  be a sequence of finitely generated modules of Krull dimension $d$.
Assume that for all $n  \in \N_+$, that $\rank(M_n)$ is defined.  Let $\ux$ be a \sop for $R$.
Then $\ell\big(H_i(\ux; M_n)\big) = \oo\big(\nu(M_n)\big)$ for all $i \geq 1$ if and only
$\ell\big(H_i(\ux; M_n)\big) = \oo\big(\rank (M_n)\big)$ for all $i \geq 1$.
Moreover,  $\chi(\ux; \, M_n) = \oo\big(\nu(M_n)\big)$ if and only if $\chi(\ux; \, M_n) = \oo\big(\rank(M_n)\big)$,

Hence,  $\{M_n\}_n$ is lim \CM iff for some (equivalently, every) \sop $\ux$,
$\ell\big(H_i(\ux, \, M_n)\big) = \oo\big(\rank(M_n)\big)$  for all $i \geq 1$,
and  $\{M_n\}_n$ is weakly lim \CM iff for some  (equivalently, every) \sop $\ux$,
$\chi_1(M_n) = \oo\big(\rank(M_n)\big).$
\end{theorem}
\begin{proof} The statements in the second paragraph follow from those in the first paragraph.  Fix $\{M_n\}_n$
and $\ux$.  Since $\rank(M_n) \leq \nu(M_n)$,  the statements for $\rank$ imply those for $\nu$,  and thus
in all cases we may assume that $\{M_n\}_n$ is at least weakly lim \CM.  Therefore, we may apply  Lemma~\ref{Cbound},
and the required results follow because $\nu(M_n)/\rank(M_n)$ is bounded above and below by positive constants
that do not depend on $n$. \end{proof}

\subsection{Ulrich modules and (weakly) lim Ulrich sequences}

Let $\rmk$ be local.  A module over $R$ is called {\it Ulrich} if it is a maximal Cohen-Macaulay
module such that $\nu(M) = e(\m;\,M)$ (one always has $\nu(M) \leq e(\m;\,M)\,)$. Ulrich modules
have received a great deal of study (see, for example, \cite{Ul84, BHU87, HUB91, Hanes99, Hanes04, BRW05})
in part because their existence implies a famous conjecture of C. Lech (cf. \cite{Lech60, Lech64}
that if $(R,\, \m) \to (S, \n)$ is a flat local map then $e(\m;\,R) \leq e(\n;\,S)$. See also \cite{Hanes99, Hanes04, Ma17}.
Lech's conjecture has an appealingly simple statement, but remains unsettled despite enormous effort.

We conclude this section by mentioning that a sequence of modules $\{M_n\}_n$ over
a local ring $\rmk$ is {\it (weakly) lim Ulrich} if it is (weakly) lim \CM and
$\disp\lim_{n \to \infty}  \frac{e(\m;\, M_n)} {\nu(M_n)} = 1$.  The existence of weakly
lim Ulrich sequences has been used recently to prove cases of Lech's conjecture \cite{Ma23} and
related results \cite{IMW22}.  The notion is also studied in \cite{Yhee23}.
																																												
																							\section{Existence of lim \CM algebra sequences in the F-finite case} \label{Fflcm} %%5

In this section we prove that lim \CM sequences of algebras exist for every F-finite local ring.
Thus, the situation is very much better than for small \CM modules, which are known to exist
only in a handful of cases.
 For example, if $R$ is an F-finite local domain (or if $R$ is F-finite and reduced)
the sequence $\{R^{1/p^n}\}_n$ is a lim \CM sequence of  $R$-algebras. In contrast, a local
domain of prime characteristic $p>0$ may fail to have any module-finite extension that is
\CM:  see \cite{Bha14}.

%%5.1
\begin{notation}  If $R$ is a ring of prime characteristic $p>0$, then $F$ or $F_R$ denotes the Frobenius endomorphism
of $R$,  $F^n_R$ or $F^n$ denotes its $n$-fold composition with itself, and
$\cF_R^n$ or simply $\cF^n$  denotes the base change functor from $R$-modules to $R$-modules
using the homomorphism $F_R^n:R \to R$.  Thus, $\cF^n(R) \cong R$,  and $\cF^n\big(\Coker(r_{ij})\big)
\cong \Coker(r_{ij}^{p^n})$. If $M$ is an $R$-module, we shall write $\nn M$ or $F^n_*(M)$ for
the $R$-module obtained from $M$ by restriction of scalars under the homomorphism $F^n:R \to R$.  \end{notation}

Our existence theorems for lim \CM sequences are based on the following result, \cite[Theorem 6.2]{HH93},  which is a strengthening of a theorem proved by P.~Roberts \cite{Rob89} for the case $M = R$, $s = d$.  Related results and 
refinements may be found in \cite{Chang97, Du83a} and \cite{Sei89}.

%%5.2
\begin{theorem}\label{Kzbnd} Let $\rmk$ be a local ring of prime characteristic $p >0$ and let $G_\bu$ be a finite 
left complex $0 \to G_s \to \cdots \to G_0 \to 0$
of length $s$ such that each $G_i$ is a finitely generated free $R$-module. Also suppose that every
$H_i(G_\bu)$ has  finite length. Let $M$ be a finitely generated $R$-module and let $d = \dim(M)$. Then there
is a constant $C > 0$ such  that $\ell\big(H_{s-t}(M \otimes_R  \cF^n(G_\bu)\big) \leq Cp^{n\min\{d,t\}}$ for all
$t, n \in \N$. \end{theorem}

Before proving the next theorem, we need a preliminary result.

%% 5.3
\begin{lemma}\label{nu} Let $\rmk$ be an F-finite local ring of prime characteristic $p >0$ and
of Krull dimension $d$,  and let
$M$ be any $R$-module of Krull dimension $d$. Let $[K:K^p] = p^\alpha$, and let
$\gamma > 0$ be the Hilbert-Kunz multiplicity of $M$.
Then $$\lim_{n \to \infty} \frac{\nu\big(F^n_*(M)\big)}{p^{(\alpha+d)n}} = \gamma.$$
%In particular,  $\nu\big(F^n_*(M)\big) = \OO(p^{(\alpha+d)n})$.
\end{lemma}
\begin{proof}  By Nakayama's lemma,
 $\nu_R\big(F^n_*(M)\big) = \ell\big(F^n_*(M/\m^{[p^n]}M)\big)$.  The theory of Hilbert-Kunz multiplicities
 \cite{Mon83} yields that  $\ell(M/\m^{[p^n]}M) = \gamma p^{nd} + \OO(p^{d(n-1)})$,
 where $\gamma >0$ is the Hilbert-Kunz multiplicity of $M$. To get the length over $R$
 of $F^n_*(M/\m^{[p^n]}M)$ we simply multiply by $[K:K^{p^n}] = p^{\alpha n}$. \end{proof}

From Theorem~\ref{Kzbnd} and Lemma~\ref{nu} we obtain:

%%5.4
\begin{theorem}\label{nMlimcm}  Let $\rmk$ be an F-finite local ring of prime characteristic $p >0$
of Krull dimension $d$,  and let
$M$ be any $R$-module of Krull dimension $d$.  Then the sequence $\{F^n_*(M)\}_n$ is a lim \CM sequence of
$R$-modules.  Hence, if $S$ is any ring module-finite over $R$ of the same Krull dimension as $R$,
$\{F^n_*(S)\}_n$ is a lim \CM sequence of $R$-algebras.  In particular,  $\{F^n_*(R)\}_n$ is a lim \CM sequence
of $R$-algebras. \end{theorem}
 \begin{proof}  Let $\ux := \vect x d$ be a \sop for $R$. Let $\ux^h := x_1^h, \, \ldots, x_d^h$.
  The modules $F^n_*(M)$ all have the same Krull dimension
 as $R$. The Koszul homology module $H_i\big(\ux;\, F^n_*(M)\big)$ may be identified with
 $F^n_*\big(H_i(\ux^{p^n}; \,M)\big)$.  Hence, if we take $G_\bu := \cK_\bu(\ux;\, R)$,
 we have $H_i\big(\ux;\, F^n_*(M)\big) \cong F^n_*\Big(H_i\big(\cF^n(G_\bu) \otimes M\big)\Big)$.
 With $C$ as in Theorem~\ref{Kzbnd} we have that
 the length of $H_i\big(\cF^n(G_\bu) \otimes M\big)$ over $R$ is bounded by $Cp^{n(d-i)}$.
 Thus, the length of $H_i\big(\ux;\, F^n_*(M)\big)$ is bounded by
 $Cp^{n(d-i)}\ell_R\big(F^n_*(K)\big)$,  and $\ell_R\big(F^n_*(K)\big)=  [K:K^{p^n}]  = p^{\alpha n}$.  This yields
 $$
 \ell_R\big(H_i(\ux;\, F^n_*(M))\big) \leq Cp^{n(d-i)}p^{\alpha n}
 $$
 for $i \geq 1$.   Consequently, for $i \geq 1$,  by Lemma~\ref{nu}
$$\frac{\ell\big(H_i(\ux; F^n_*(M)\big)} {\nu\big(F^n_*(M)\big)} \leq
\frac {Cp^{n(d-i)}p^{\alpha n}}{p^{\alpha n} C'p^{nd}}$$
for a suitable positive constant $C'$,
and the limit on the right as  $n \to \infty$ is 0,  as required.
The last two statements of the theorem then follow immediately.
 \end{proof}

 %%5.4
 \begin{remark}  If $R$ is a reduced ring,  $F^n_*(R) \cong R^{1/p^n}$. Hence, as observed
 earlier,  $\{R^{1/p^n}\}_n$  is a lim \CM sequence of $R$-algebras in the F-finite case. \end{remark}

\section{Serre intersection multiplicities and lim Cohen-Macaulay sequences}\label{limcmSerre}%%6

Throughout this section $T$ is a regular local ring and tensor products and Tor modules
are taken over $T$ unless otherwise specified.

%%6.1
\begin{theorem}\label{pos} Let $(T,\,\m,\,K)$ be a regular local ring of Krull dimension $d$.
 Let $P$ and $Q$ be prime ideals of $T$ such that $\dm(T/P) + \dm(T/Q)  = d$
 and $P+Q$ is $\m$-primary.  If $R := T/P$ and $S := T/Q$ have lim Cohen-Macaulay
 sequences $\{M_n\}_n$ and $\{N_n\}_n$, respectively, then $\chi(R,\,S) > 0$. In fact,
 $$\chi(R,\,S) =  \lim_{n \to \infty} \frac{\ell(M_n \otimes_T N_n)}{\rank(M_n)\rank(N_n)} \geq 1.$$ \end{theorem}

 %% 6.2
 \begin{remark}\label{geq1} Note that since we have  $\ell(M_n \otimes _T N_n) \geq
 \ell\big((M_n/\m M_n) \otimes_K (N_n/\m N_n)\big)$
 $= \nu_T(M_n) \nu_T(N_n) \geq \rank(M_n) \rank(N_n)$, the limit on the right hand side
 of the display above is clearly at least 1. \end{remark}

%%  6.3
\begin{corollary} If complete local domains with perfect (or algebraically closed) residue class field of dimension
at most $h$ have lim Cohen-Macaulay sequences, then Serre's conjecture on intersection
multiplicities holds in dimension  up to $h+2$. \end{corollary}
\begin{proof} To prove the result for a regular local ring $(T,\, \m, \, K)$, it suffices to prove it for $\wh{T}$, and
it then suffices to prove it for $\wt{T}$, where $\wt{T}$ is a faithfully flat complete local extension
of $T$ whose residue field is the algebraic closure of $K$ and whose maximal ideal is
$\m \wt{T}$.  Thus, we may replace $T$ by a complete regular local ring of the same dimension whose
residue class field is algebraically closed.  It suffices to show that if $P$ and $Q$ are prime ideals of
$T$ such that $P + Q$ is $\m$-primary and $\dim(T/P) + \dim(T/Q) = \dim(T)$,  then
$\chi^T(T/P, T/Q) >0$.  We may assume without loss of generality that $\dim(T/P)
\leq \dim(T/Q)$.  If $P = (0)$ and $Q = \m$, the result is obvious.  If $\dim(T/P) = 1$ then the height of
$Q$ is $1$,  and $Q$ is principal.  This case is also known.  If $\dim(T/P) \geq 2$,
then $\dim(T/Q) \leq h$,  and we have $\dim(T/P) \leq \dim(T/Q) \leq h$, so that both $T/P$
and $T/Q$ have lim \CM sequences by hypothesis.  The result now follows from
Theorem~\ref{pos}.  \end{proof}

\noindent\emph{Proof of Theorem~\ref{pos}.}\  It will suffice to show that for all $i \geq 1$,
we have that
$$
(\dagger) \quad\ell\bigl(\Tor^T_i(M_n,\, N_n)\bigr) = \oo\bigl(\rank(M_n) \rank(N_n)\bigr).$$
 We then have that
$$
\chi^T(R,\, S) = \frac{\chi^T(M_n, \, N_n)}{\rank(M_n)\rank(N_n)} =
\frac{\ell(M_n \otimes N_n)}{\rank(M_n)\rank(N_n)} +
\sum_{i=1}^d \frac{(-1)^i \ell\big(\Tor^T_i(M_n,\,N_n)\big)}{\rank(M_n)\rank(N_n)}.
$$
As $n \to \infty$, the leftmost term is constant, the first summand on the rightmost expression
is the term whose limit is taken in the statement of the theorem, and
and the remaining terms in the rightmost expression approach 0 as $n \to \infty$.

To prove $(\dagger)$ we first want to choose parameters
$\ux ,\uy$  for   $T$  such that the  $x_i$   are in  $Q$
and their images form a \sop\ in  $T/P$  and the  $y_j$  are in  $P$
and their images form a \sop\ in  $T/Q$.   By elementary prime
avoidance we can recursively choose $\vect x h \in Q$ such that they are
part of  a system of parameters in $T$,  where $h = \hgt(Q) = \dm(T/P)$
and such that their images are also a \sop\ in $T/P$:  once $\vect x i \in Q$
have been chosen for $i < h$ such that they are part of a \sop\ in $T$ and
have images that are part of a \sop\ in $T/P$,  one chooses $x_{i+1} \in Q$
that is not in any minimal prime of $(\vect x i)$ nor any minimal prime
of $P+(\vect x i)$: no such minimal prime can contain $Q$ because
$\hgt(Q) = h$ and $\hgt(P) + \hgt(Q) = \dm(T) > \hgt(P) +i$.  Then one recursively chooses
$\vect y k \in P$,  where $k = \hgt(P)$.  Once $\vect xh, \vect y j$,  $j <k$
have been chosen so that $\vect x h, \vect yj$ are part of a \sop\ for $T$ and
the images of $\vect yj$ are part of a \sop\ for $T/Q$, one chooses $y_{j+1} \in P$
that is not in any minimal prime of $\vect x h, \vect y i$ nor any minimal prime
of $Q+(\vect y j)$.  A minimal prime of the former that contained $P$ would
contain $(\vect x h) + P$, which is primary to $\m$,  while $h + j < h+k = \dm(T)$,
while a minimal prime of the latter that contained $P$ would contain $P+Q$,
which is also primary to $\m$.

 Hence, we may assume without loss of generality that we have
parameters $\ux \in P$, $\uy \in Q$ as above such that $\ux, \uy \in P+Q$ is a
\sop in $T$.  The systems of parameters    $\ux$ and
$\uy$ will be fixed for the remainder of the proof, except that at one point all of the elements will be replaced
by powers.   We will complete the proof by showing
that for $i \geq 1$ and for all $n$,
$$(*)\quad \ell\bigl(\Tor_i(M_n,\, N_n)\bigr) \leq
\sum_{r+s+t = d+i} \ell\Bigl(\Tor^T_t\bigl(H_r(\ux; M_n),\, H_s(\uy; N_n)\bigr)\Bigr).$$
Here, $r,s,t$ are nonnegative integers and the terms in the sum on the right vanish if $t > d$, or if $r > h$,  or if $s >k$.
If $t \leq d$, the condition $r+s+t > d$ implies that $r +s > 0$.  Thus, for each of the at most
$(d+1)(h+1)(k+1)$
nonvanishing terms on the right, we have  $\ell\big(H_r(\ux; \, M_n)\big) =
 \OO\big(\rank(M_n)\big)$ and $\ell\big(H_s(\uy; \, M_n)\big) = \OO\big(\rank(N_n)\big)$ by Lemma~\ref{Cbound}, and since at least one of $r,s$
 must be positive, in at least one of the two terms we can replace $\OO$ by $\oo$.  The fact that
 $\ell\bigl(\Tor_i(M_n,\, N_n)\bigr) =\oo\bigl(\rank(M_n) \rank(N_n)\bigr)$ then follows from
 Corllary ~\ref{ltor}.

 In this paragraph we make repeated used of the spectral sequences described in Discussion~\ref{mulTor}.
Let $\cP_\bu$ denote the total complex
obtained by tensoring a free resolution $\cF_\bu$ of $M_n$ over $T$ with a free resolution $\cG_\bu$ of $N_n$
over $T$.  Because $\ux, \, \uy \in P+Q$,  these
elements kill the homology of $\cP_\bu$.  Since $\cP_\bu\in D^{[-d,0]}$, it follows from Lemma~\ref{Koszsurj} that, upon replacing the elements in $\ux, \,\uy$ by their $(d+1)\,$th powers, we have that 
$\Tor^T_{d+i}\bigl(T/(\ux,\uy), \, \cP_\bu\bigr)$ maps onto $H_i(\cP_\bu) \cong \Tor_i(M_n,\,N_n)$.
But
$$\Tor^T_{d+i}\bigl(T/(\ux, \uy), \, \cP_\bu\bigr) \cong \Tor^T_{d+i}\bigl(\cK_\bu(\ux; T), \cK_\bu(\uy; T), \cP_\bu\bigr)
\cong\qquad\qquad\qquad\qquad\qquad$$
$$\quad\Tor^T_{d+i} \bigl(\cK_\bu(\ux; T),\, \cK_\bu(\uy; T), \cF_\bu, \cG_\bu\bigr) \cong
\Tor^T_{d+i}\bigl(\cK_\bu(\ux; T) \otimes \cF_\bu, \cK_\bu(\uy;\,T)\otimes \cG_\bu\bigr).$$
We have a spectral sequence
$$\Tor^T_q\big(H_r(\ux; M_n), \cK_\bu(\uy;\,T)\otimes \cG_\bu\big) \imp
\Tor^T_{q+r}\bigl(\cK_\bu(\ux; T) \otimes \cF_\bu, \cK_\bu(\uy;\,T)\otimes \cG_\bu\bigr).$$ But for each term on
the left above
with $q+r = d+i$  we also have a spectral sequence
$$\Tor^T_t\bigl(H_r(\ux; M_n), \, H_s(\uy; N)\bigr) \imp
\Tor^T_{t+s}\bigl(H_r(\ux; M_n), \cK_\bu(\uy;\,T)\otimes \cG_\bu\bigr),$$
Using Remark~\ref{diag},
this shows that
$$\ell\Bigl(\Tor^T_{d+i}\bigl(T/(\ux,\uy), \, \cP_\bu\bigr)\Bigr) \leq
\sum_{r+s+t = d+i} \ell\Bigl(\Tor^T_t\bigl(H_r(\ux; M_n), H_s(\uy; N_n)\bigr)\Bigr),$$
and since the term on the left maps onto $\Tor^T_i(M_n,\,N_n)$, we have the required length estimate. \qed

%% 6.4
\begin{remark}
Let us give a derived category argument for the last paragraph of the proof above (i.e., the proof of $(\ast)$), with the same notation there. As each $M_n$ is a $T/(\underline{y})$-module and each $N_n$ is a $T/(\underline{x})$-module, the tensor product $M_n \otimes_T^L N_n$ is naturally linear over $T/(\underline{x},\underline{y}) = 
T/(\underline{x}) \otimes_T^L T/(\underline{y})$, i.e., in the image of $D(T/(\underline{x},\underline{y})) \to D(T)$. We can thus write
\[ G := (M_n \otimes_T^L N_n) \otimes_T^L T/(\underline{x},\underline{y}) \simeq (M_n \otimes_T^L N_n) 
\otimes_{T/(\underline{x},\underline{y})}^L  T/(\underline{x},\underline{y}) \otimes_T^L T/(\underline{x},\underline{y}).\]
By a standard calculation with Koszul complexes, we can write 
$T/(\underline{x},\underline{y}) \otimes_T^L T/(\underline{x},\underline{y}) = \oplus_i \wedge^i F[i]$, where $F$ is a free module over $T/(\underline{x},\underline{y})$ of rank $d$ (given canonically by the conormal bundle 
$(\underline{x},\underline{y})/(\underline{x},\underline{y})^2$). In particular, looking at the $i=d$ summand shows that $(M_n \otimes_T^L N_n)[d]$ appears as a direct summand of $G$, whence
\[ \ell(\mathrm{Tor}_i^T(M_n,N_n)) = \ell(H_{i+d}(M_n \otimes_T^L N_n[d])) \leq \ell(H_{i+d}(G)).\]
On the other hand, we can also rewrite $G$ as 
\[ G = \left(M_n \otimes_T^L T/(\underline{x}\right)) \otimes_T^L \left(N_n \otimes_T^L T/(\underline{y}\right)).\]
Filtering the bracketed terms by their cohomology and running the cohomology spectral sequence then shows that
\[ \ell(H_{i+d}(G)) \leq \sum_{r+s+t=d+i} \ell(\mathrm{Tor}^T_t(H_r(\underline{x};M_n), H_s(\underline{y};N_n))),\]
as wanted.
\end{remark}

\section{Asymptotic module closure operations}\label{op} %%7

If $R$ is a local ring, every sequence $\cM = \{M_n\}_n$ of nonzero finitely generated $R$-modules together with
an $\N_+$-valued function $\alpha$ on the modules in the sequence $\cM$ (e.g., $M_n \mapsto \nu(M_n)$, the 
least number of generators, or
$M_n \mapsto \rank(M_n)$)   defines
a closure operation on submodules of finitely generated modules over $R$.  A
surprising number of useful properties can be proved with little or no restriction on the
sequence.  Under mild hypotheses, one can show that the closure obtained for ideals is
contained in the integral closure: see Theorem~\ref{cloint}.  In case the sequence of modules is lim Cohen-Macaulay,
one obtains results similar to results on capturing parameter colon ideals and on the
behavior of Koszul homology (in degree at least one, the cycles are in the closure of boundaries
in the ambient chain module) similar to results originally proved for tight closure. These results
are presented in \S\ref{cap}: Theorems~\ref{thmcappar} and \ref{strcap} are specific results of this kind.

Note that there have been many attempts to extend tight closure theory:  see, for example,
\cite{Bre03, Heit01, HeitMa21, Ho94, Ho03, HV04, HZ18} and \cite{Jia21}.

%% 7.1
\begin{definition}\label{clo} Let $\cM$ denote $\{M_n\}_n$,  a sequence of nonzero finitely generated modules of over a local ring $\rmk$ of dimension $d$.   Let $\alpha$ be a function from the set of modules $\{M_n: n \in \N\}$ to $\N_+$, the
positive integers.  Of particular interest is the case where $\alpha$ is  $\nu$ (number of generators)
or $\rank$ (see Subsection~\ref{rank}), which is torsion-free rank when $R$ is a domain.  For each such $\cM$ and
$\alpha$ we define a closure operation on submodules
of finitely generated $R$-modules,  which we refer to as $\cM$-closure with respect to $\alpha$, or
$(\cM, \, \alpha)$-closure.   If $A \inc B$ are
$R$-modules, we use the notation $A\cla_B$ for the $\cM$-closure of $A$ in $B$ with
respect to $\alpha$.
 These closure operations on submodules $A$ of finitely generated
$R$-modules $B$ are defined as follows.  If $B/A$ has finite length,
we define the closure $A\cla_B$  to be the largest submodule
$A'$ of $B$ containing $A$
such that
$$ (\dagger_\alpha)\quad\ell\Big(\Img\big(M_n \otimes_R (A'/A) \to M_n \otimes_R (B/A)\big)\Big) =
\oo\bigl(\alpha(M_n)\bigr).$$
Since $0 \to \Img\big(M_n \otimes_R (A'/A)\big) \to M_n \otimes_R B/A \to M_n \otimes_R B/A' \to 0$
is exact, this is equivalent to the condition
$$(\ddagger_\alpha)\quad \ell(M_n \otimes B/A') - \ell(M_n \otimes B/A) = \oo\bigl(\alpha(M_n)\bigr).$$
There is a largest such submodule because the sum of two submodules satisfying $(\dagger_\alpha)$ again
satisfies $(\dagger_\alpha)$, and $B$ is Noetherian.
In general, we define the $\cM$-closure of $A$ in $B$ with respect to $\alpha$
as the intersection of the closures of the modules
$A +\m^t B$ for $0 \leq t < \infty$. This evidently
gives the same result as the original definition if $B/A$ has finite length.

In some cases when $\cM$ and $\alpha$ are fixed, we may simplify notation and write $A\cln_B$ instead
of $A\cla_B$, especially for closure  with respect to $\nu$ or $\rank$ in situations when the two possible
closures are known to coincide.  \end{definition}

%% 7.2
\begin{remark} The subscript $\blank_B$ may be omitted if $B$ is clear from context.
In particular, in discussing ideals of $R$,   we omit the subscript $R$ when discussing
the closure of an ideal in $R$ unless otherwise indicated. \end{remark}

%% 7.3
\begin{remark} We can make essentially the same definition as in \ref{clo} for a net of nonzero
$R$-modules as discussed
in Remark~\ref{net} of \S\ref{limdef}. In fact, the results of this section are entirely valid for
nets as well as for sequences, with only very straightforward modifications of the arguments.
For the sake of simplicity, we have chosen to state most results only in the case of sequences.
Example~\ref{intclop} is an exception. \end{remark}

Throughout this section, we make use of the notion of {\it rank} as defined in the second paragraph
of Subsection~\ref{rank}.  This agrees with torsion-free rank when the base ring is a domain.  We
also make free use of Propositions~\ref{ranknu} and~\ref{multrank}, which show that this notion of
rank has the usual relationships to $\nu$ and to Hilbert-Samuel multiplicity. \smallskip

The next result establishes some basic properties of the closures defined in \ref{clo}:  (a)--(i) correspond
to well-known properties of tight closure.

%% 7.4
\begin{proposition}\label{cloprop}  Let $\rmk$ be a local ring, let $\cM = \{M_n\}_n$, $\alpha$ be as in Definition~\ref{clo},
and let $A \inc B$ be finitely generated $R$-modules. Let $G$ denote a finitely
generated free $R$-module that maps onto $B$.  Let $H$ be the inverse image of $A$ in $G$.
Let $b \in B$ and let $\widetilde{b}$ denote an element of $G$ that maps
to $b$. Let $\ov{b}$ denote the image of $b$ in  $B/A$.

\bena
\item The following three statements are equivalent:
\benn
\item $b \in A\cla_B$.    %% 1
\item $\ov{b} \in 0\cla_{B/A}$.   %% 2
\item $\widetilde{b} \in H\cla_G$.  %% 3
\een
\een
Moreover:
\benr
\item   $A \inc A\cla_B$. %% (b)
\item If $\theta: B \to B'$ is $R$-linear and $\theta(A) \inc A' \inc B'$,
then $\theta(A\cla_B) \inc (A')\cla_{B'}$. %% (c)
\item If $A \inc A' \inc B$ then $A\cla_B \inc (A')\cla_B$. %% (d)
\item If $A \inc B \inc B'$ , then $A\cla_B \inc A\cla_{B'}$. %% (e)
\item $(A\cla_B)\cla_B = A\cla_B$.  %% (f)
\item $(A \oplus A')\cla_{B \oplus B'} \cong A\cla_{B} \oplus {A'}\cla_{B'}$.  %% (g)
\item The intersection of any family of $\cM$-closed submodules of $B$ with respect to $\alpha$
is $\cM$-closed with respect to $\alpha$. %% (h)
\item If $A\cla_B = A$ and $J$ is any ideal of $R$, then $(A:_BJ)\cla_B = A:_BJ$.  %% (i)
\item For any ideal $J \inc R$,  $JA\cla_B \inc (JA)\cla_B$. %% (j)
\item  If $\alpha$, $\beta$ are two functions from the set $\{M_n: n \geq 1\}$ to the positive
integers and there is a positive real constant $c$ such that $\alpha(M_n) \leq c\beta(M_n)$
for all $n \gg 0$,  then $A\cla_B \inc A\clb_B$.  %% (k)
\item If rank is defined and nonzero on $\{M_n:n \in \N_+\}$ then $A^{*\cM, \rank} \inc A^{*\cM, \nu}$.    %% (l)
\een
\end{proposition}
\begin{proof}  The equivalence of (1) and (2) in part (a) when $B/A$ has finite length is immediate from
the definition.  The general case  follows from the definition and the fact that $B/(A+\m^t) \cong (B/A)/\m^t(B/A)$.
The equivalence of (3) is then clear, since $G/H \cong B/A$ in such a way that the image of $\wt{b}$ is $\ov{b}$.

Part (b) is obvious from the definition.

 To prove part (c) we may replace the pairs
$(A,B)$ and $(A',B')$ by the pairs $(0,B/A)$ and $(0,B'/A')$ and $\theta$ by the map $B/A \to B'/A'$  that it induces.
Thus, we may assume that $A= 0$,  $A' =0$.  First suppose that $B$, $B'$ have finite length.
 The result now follows because
for any $C \inc B$,    with $\iota:C \inc B$, $\iota':\theta(C) \inc B'$,  and $\id:= \id_{M_n}$,
$$\ell\big((\id\otimes\iota)(M_n  \otimes_R C) \big) \geq
\ell\big((\id  \otimes_R \iota')(M_n \otimes_R \theta(C))\big),$$
since the surjection $C \surj \theta(C)$ induces a surjection
$$(\id\otimes\iota)(M_n  \otimes_R C) \surj (\id  \otimes_R \iota')(M_n \otimes_R \theta(C))$$
after one applies $M_n\otimes_R \blank$ and takes images. One may now apply this with
$C := 0\cla_B$.

To prove (c) in general it suffices to note that for all $t$,  $\theta(\m^tB) \inc \m^tB'$ and
$\theta(C + \m^tB) \inc \theta(C) + \m^tB'$,
and then to apply the case already proved to the induced map $\theta_t:B/\m^tB \to B'/\m^tB'$ and
the submodule $C_t := (C+\m^tB)/\m^tB$ of $B/\m^tB$ when $C = 0\cla_B$.

 Parts (d) and (e) are both special cases of part (c).

The statement in part (f) follows if we show that if $A \inc A' \inc A'' \inc B$ are such that $A'$ is in
the closure of $A$
in $B$ and $A''$ is in the closure of $A'$ in $B$, then $A''$ is in the closure of $A$ in $B$.  This in turn reduces
to the case where $B/A$ has finite length by considering $A + \m^tB \inc A'+\m^t B \inc A'' + \m^tB$ for
every value of $t$.  The case where $B/A$ has finite length may be proved as follows.
The image of $M_n \otimes_R (A''/A)$ in $M_n \otimes_R (B/A)$ has a filtration in which one factor
is the image of $M_n \otimes_R(A'/A)$ in $M_n \otimes_R (B/A)$, and the other may be identified
with the image of $M_n \otimes _R(A''/A') \inc M_n \otimes_R (B/A')$.  Since the lengths of both factors are
$\oo\big(\alpha(M_n)\big)$, the sum of their lengths is also $\oo\big(\alpha(M_n)\big)$.

The $\inc$ in part (g) follows at once, for each of the summands, from (b), while $\supseteq$
follows from (b) as well using the two projections $B \oplus B' \surj B$ and $B \oplus B' \surj B'$:
these carry  $(A \oplus A')\cla_{B \oplus B'}$ into $A\cla_B$ and ${A'}\cla_{B'}$, respectively.

Part (h) is clear, since if $u$ is in the closure of the intersection, (d) implies that it is in the closure
of each module in the family and, hence, in each module in the family.    To prove (i),  note
that $A:_B J = \bigcap_{f \in J} A:_B f$.  By (h), we may assume that  $J = fR$.  By (d),
since $0$ is closed in $B/A$,   we have that $0$ is closed in $f(B/A) \cong B/(A:_B fR)$.

To prove $(j)$ we first consider the case where $J = rR$ is principal. We apply (c) to the map $\theta$ given by
multiplication by $r$ from $B$ to $B$ to obtain  $rA\cla_B \inc (rA)\cla_B$.  In general,  if $J = (\vect r k)$ we
have that for every $i$,  $r_i A\cla_B \inc (r_i A)\cla_B \inc (JA)\cla_B$ by part (d), and so
$J(A\cla B) = \sum_{i=1}^k r_i A\cla_B \inc (JA)\cla_B$.

Part (k) is immediate from the observation that $(\dagger_\alpha)$ is a stronger condition than
$(\dagger_\beta)$, while (l) follows from (k),  since $\rank(M_n) \leq \nu(M_n)$ for all $n \geq 1$.
 \end{proof}

 %% 7.5
  \begin{examples}\label{exclo}  Let $\rmk$ be a local ring.
 \benn
\item  If all of the terms of the sequence $\cM$ are $R$,  then $0$ is closed in the $B$
 for any module $B$ of finite length, for if $C \inc B$ is not 0, we cannot have $\ell(C) = \oo\big(\alpha(R)\big)$, which
 is a constant, unless $\ell(C) = 0$.  It follows easily that for every pair of finitely generated modules $A \inc B$,
 $A\cla_B = A$, no matter what $\alpha$ is.  %% (1)

 \item If $\alpha$ is $\nu$ or if rank is defined and nonzero for the modules $M_n$ and $\alpha$ is rank,
 $\m\cla_R = \m$.  The only other possibility is that $\m\cla_R = R$,  which would imply
 that the closure of $0$ in $K$ is $K$,  so that $\ell(M_n/\m M_n) = \nu(M_n)$ is $\oo\big(\alpha(M_n)\big)$.
 This is clearly false since $\nu(M_n)/\alpha(M_n)$ is $\geq 1$ and so does not approach 0 as $n \to \infty$.  %% (2)
 \een
 \end{examples}

% 7.6
\begin{discussion}\label{tcintro} We shall soon prove in
Theorem~\ref{tclimcm}  that tight closure over a reduced equidimensional F-finite local ring
is a closure operation arising from a lim \CM sequence.  Studying tight closure in the F-finte case
is a central concern:  see Discussion~\ref{tc}.
We refer the reader to \cite{Bru96, HH90, HH93, HH94a, HH94b} for
tight closure background.  In particular, we need the characterization of tight closure related
to Hilbert-Kunz multiplicities given in \cite[Theorem 8.17]{HH90}.  Hilbert-Kunz multiplicities were defined by
Monsky \cite{Mon83}, and studied further, for example, in \cite{HaMon93, MonT04, MonT06}.   \end{discussion}

%% 7.7
\begin{remark}\label{froblength} If $R$ has prime characteristic $p>0$, we write $\cF_R^n= \cF^n$ for the base 
change functor from $R$-modules to $R$-modules given by  $F^n_R = F^n$,  the $n$-fold composition of the
Frobenius endomorphism with itself, so that $F^n:r \mapsto r^{p^n}$.  The $\cF^n_R$ are called the {\it Frobenius}
or  {\it Peskine-Szpiro} functors.  See, for example, \cite[Discussion 8.1]{HH90} for more detail.
Note that $\cF^n(R) = R$,  that $\cF^n\big(\Coker(r_{ij}) \big) \cong \Coker(r_{ij}^{p^n})$,  and that
$\cF^n(R/I) \cong R/I^{[p^n]}$,  where $I^{[p^n]}$ denotes the ideal generated by the $p^n\,$th powers of
all elements (equivalently, of a set of generators) of $I$, and is the same as the extension of $I$
under $F^n:R \to R$.

If $B$ is any $R$-module, note that   $F_*^n\big(\cF^n(B)\big) \cong F_*^n(R) \otimes_R B$.
If $\rmk$ is an F-finite local ring and $B$ is a finite length R-module,
$$
(\dagger) \quad \ell_R\big(F_*^n(R) \otimes_R B\big) = \ell_R\big(F_*^n\big(\cF^n(B)\big)\big) =
[K:K^{p^n}]\ell_R\big(\cF^n(B)\big).
$$
More generally, if $C \inc B$, the image of $F_*^n(R) \otimes C$ in $F_*^n(R) \otimes B$,
and the image of $\cF^n(C)$ in $\cF^n (B)$, which is denoted $C^{[p^n]}_B$ in \cite{HH90},
are related in the same way: $\Img \big(F_*^n(R) \otimes C \to F_*^n(R) \otimes B\big) \cong F_*^n(C^{[p^n]}_B)$,
and so
$$
(\ddagger) \quad \ell\Big(\Img \big(F_*^n(R) \otimes C \to F_*^n(R) \otimes B\big)\Big) =
[K:K^{p^n}]\ell_R(C^{[p^n]}_B).
$$
 \end{remark}

 %% 7.8
 \begin{remark}\label{Fnrank} If $\rmk$ is an F-finite local domain of Krull dimension $d$, and $[K:K^p] =p^\alpha$,
 the torsion free rank of
 $F_*(R)$ over $R$ is $p^{(\alpha+d)}$, by a
theorem of Kunz \cite{Kunz76},  but see also the footnote to \cite[Theorem (2.2)(ii)]{Tu12}. From this one
has at once that the torsion-free rank of  $F^n_*(R)$ over $R$ is $p^{(\alpha+d)n}$. \end {remark}

%% 7.9
\begin{theorem}\label{tclimcm} Let $\rmk$ be  an equidimensional reduced F-finite local ring
of prime characteristic $p >0$ and
Krull dimension $d$.  The closure operation obtained from the lim Cohen-Macaulay sequence
$\cM := \{R^{1/p^n}\}_n$  (or  $\{F^n_*(R)\}_n$) and either $\nu$ or $\rank$ \emph{is} the usual notion of
tight closure in prime characteristic $p >0$.  I.e., for this choice of $\cM$,  $A^{*\cM,\rank}_B =
A^{*\cM, \nu}  = A^*_B$.  \end{theorem}
\begin{proof}  Let $p^\alpha = [K:K^p]$.  We first note that when $R$ is F-finite reduced local and equidimensional,
$\rank_R F_*^n( R)$ is well defined.  In fact, for any minimal prime $\fp$ of $R$,  if $L_\fp$ denotes the fraction
field of the domain $R/\fp$, this will be the same as $[L_\fp:L_\fp^{p^n}] = [L_\fp: L_\fp^p]^n$,  provided that
all of these degrees are the same, and do not depend on $\fp$.  In fact, by Remark~\ref{Fnrank},
they are all $p^{(\alpha+d)n}$.  Note that by
Lemma~\ref{nu}, $\nu_R\big(F^n_*(R)\big)$ is asymptotic to $\gamma p^{(\alpha + d)n}$.
This verifies that it does not matter whether we use $\nu$ or $\rank$ in studying this closure, which was
already guaranteed by Theorems~\ref{rkchar} and \ref{nMlimcm}.

Every excellent local ring $R$ and, hence, every F-finite local ring $R$ has a completely stable test element
$c \in R\0$. See, for example, \cite[Theorems 5.10 and 6.1(b)]{HH94b}.
 From this it follows using \cite[Proposition 8.13(b)]{HH90} that for any two finitely generated modules $A \inc B$,
$A^*_B =  \bigcap_t(A + \m^tB)^*_B$.  It therefore suffices to prove the theorem in the case where $B/A$ has
finite length.  Moreover, we may assume that $A = 0$ and that $B$ has finite length.

Now let $b \in B$.  By the definition of closure with respect to $(\cM, \, \nu)$,   $bR  \in A\cla_B$  if and only if
$\ell(\Img\big((F_*^n(R) \otimes bR) \to (F_*^n(R) \otimes_R B)\big) = \oo(p^{(\alpha + d)n})$.
Using the displayed line $(\ddagger)$ in Remark~\ref{froblength}, this is equivalent to
$p^{\alpha n} \ell(Rb^{p^n}_B) \leq \oo(p^{(\alpha + d)n})$,
or $\ell(Rb^{p^n}_B) \leq \oo(p^{dn})$.  But this is precisely the criterion for $b$ to be in $0^*_B$ provided
by \cite[Theorem 8.17]{HH90}.
\end{proof}

%%7.10
\begin{discussion}\label{tc}  In this paragraph, $R$ denotes a Noetherian ring of prime characeristic $p >0$,
$A \inc B$ are finitely generated $R$-modules, and $b \in B$.
It was asserted in Remark~\ref{tcintro} that the study of tight closure in the
F-finite case is a central concern.  For the large class of rings $R$ that have a completely stable test element
(or even a locally stable test element) $c$, which
includes rings essentially of finite type over an excellent local ring \cite[Theorem 6.1(b)]{HH94b},
$b \in A^*_B$  if and only if $b/1 \in (A_\m)^*_{B_\m}$ for every maximal ideal  $\m$ of
$R$. (The key point is that if $cb^{p^n} \notin A^{[p^n]}_B$, we can preserve this while localizing
at a suitable maximal ideal.)  Moreover, because every excellent local ring has a completely stable test
element $c$, if $R$ is local, $b \in A^*_B$ if and only if $b \in \wh{A}^*_{\wh{B}}$ over $\wh{R}$ (again,
if $cb^{p^n} \notin A^{p^n}_B$, this will be preserved when we complete).
Consequently, the problem
of understanding tight closure reduces to understanding what happens when $R$ is a complete local
ring.  Moreover, $b \in A^*_B$ if and only if that is true after base change to $R/\fp$ for each minimal
prime of $R$.  Thus, the case where $(R, \m)$ is a complete local domain is central.  But then one can
pass to a faithfully flat purely inseparable extension $R^\Gamma$ of $R$  with maximal ideal
$\m R^\Gamma$ so that $R^\Gamma$ is an F-finite local domain.  It turns out that
$b \in A_B^*$ over $R$ if and only if  $1 \otimes b \in (R^\Gamma \otimes_R A)^*_{R^\Gamma \otimes_R B}$
over $R^\Gamma$. We refer the reader to \cite[\S6]{HH94b} for a detailed treatment.  Since we may
also complete $R^\Gamma$ and kill minimal primes, tight closure in rings with a completely stable test
element is determined by tight closure in complete F-finite local domains.
 \end{discussion}

 %% 7.11
 \begin{example}\label{nonequi}  Our purpose here is to show that the lim \CM closure used in
 Theorem~\ref{tclimcm} does  not agree with tight closure, in general, for F-finite reduced rings that are not equidimensional.

 Let $T$ be the power series ring $K[[X,\,Y,\, Z]]$, where $K$ is an algebraically closed field of prime characteristic
 $p >0$.  (The situation where $T$ is the localization of  the polynomial ring $K[X,\,Y,\,Z]$ at the maximal
 ideal $(X,\,Y,\,Z)$ is entirely similar.)  Let $R:= T/(XY,XZ)T = K[x,\, y,\, z]$, where $xy = xz = 0$. The minimal
 primes of $R$ are $\fp = xR$ and $\fq = (y,\,z)R$.  Choose an integer $k \geq 2$.
 Note that $x^k-y,\,z$ is  a \sop for $R$, even though
 $z$ is in the minimal prime $\fq$.  Observe that $\dim(R/\fp) =2$ and $\dim(R/\fq)= 1$.  It is easy to verify
 that $(x^k-y):_R z = (x^k-y,\, x) = (x,\,y)$.  The tight closure of $(x^k-y)$ does not contain $x$, since this fails
 even after a base change to $R/\fq$.   However,   $ x \in (x^k-y)^{*\cM,\nu}_R$ for $\cM = \{F_*^n(R)\}$.  To check
 this, we must show that $x \in (x^k-y, \m^t)^{*\cM,\nu}_R$  for all $t$, and it suffices to show that
 $x \in (x^k-y, z^t)^{*\cM,\nu}_R$
 for all $t$.

  Consequently, we need to compare the difference
 of the lengths of the quotients $F_*^n(R)/(x, x^k-y, z^t)F_*^n(R)$ and $F_*^n(R)/(x^k-y, z^t)F_*^n(R)$ for  $n \gg 0$.
 Since $K$ is perfect, this is the same as the difference of the lengths of $R/(x,y,z^t)^{[p^n]}$
 and $R/(x^k-y, z^t)^{[p^n]}$.  Let $\gamma_1$, $\gamma_2$ be the respective Hilbert-Kunz multiplicities of $\fA_ 1 =(x,y,z^t)$  and $\fA_2 = (x^k-y, z^t)$.  Then the length of $R/\fA_i^{[p^n]}$ is $\gamma_ip^{2n} + \OO(p^n)$.
 Because the only minimal prime of $R$ of dimension 2 is $\fp$, we may compute the $\gamma_i$
 for $\fA_i(R/\fp)$ working over $R/\fp$ instead.  But in $R/xR$ the $\fA_i$ both extend to
 $(y, z^t)$.  This shows that $\gamma_1 = \gamma_2$, and the difference of the lengths is $\OO(p^n)
 = \oo(p^{2n})$, as required.  \end{example}

%% 7.12
\begin{example}\label{intclop} Let $\cM$ be the net, indexed by itself, of all nonzero ideals of the Noetherian
local domain $R$, where $I \leq J$ means that there exists an ideal  $I'$ such that $II' = J$.  Thus,  $IJ$ is an
upper bound for $I$ and $J$.  The closure operation for $\cM$ and $\rank$, restricted to ideals, is
integral closure.  The reason is that if $I \inc J$, then $J$ is in the integral closure of $I$
if and only if there exists a nonzero ideal $\fA$ such that $J\fA = I\fA$,  and this will also
be true for all ideals of the form $\fA\fB$ in $\cM$, i.e., all ideals larger than $\fA$ with respect to $\leq$.  The
condition that $\ell(J \fB/I \fB) = \oo\bigl(\rank(\fB)\bigr) = \oo(1)$ in the case where $I$ is
$\m$-primary implies that $J\fB = I \fB$ for all sufficiently large $\fB$, since the length
cannot be smaller than 1 otherwise.
\end{example}

In the case of closure of an ideal, we note the following alternative characterization:

%% 7.13
\begin{proposition}\label{idealcl} Let $\rmk$  be a local ring and let $\cM = \{M_n\}_n$ be a
sequence of nonzero  modules over $R$, and let $\alpha$ be a function
from finitely generated $R$-modules to $\N_+$.   Let $I$ be an $\m$-primary ideal of $R$
and let $u \in R$.  Then $u \in I\cla_R$ if and only if
$$\ell\Bigl(\frac{M_n}{IM_n:_{M_n}u}\Bigr) = \oo\bigl(\alpha(M_n)\bigr).$$ \end{proposition}
\begin{proof}  From the definition of $\cla$, $u \in I\cla_R$ iff $\ell\big((I + Ru)M_n /IM_n)\big) =
\oo\big(\alpha(M_n)\big)$. Multiplication by $u$ yields a surjection
$M \surj uM$ that restricts to a surjection $IM_n:_{M_n} u \surj IM_n \cap uM_n$.  Both surjections have
as kernel $\Ann_{M_n}u \inc IM_n:_{M_n} u$.  Thus
$$\frac{(I+uR)M_n}{IM_n}  \cong  \frac{uM_n} {IM_n \cap uM_n} \cong
\frac{M_n/\Ann_{M_n} u} {(IM_n:_{M_n}u)/\Ann_{M_n} u} \cong\frac{M_n}{ IM_n:_{M_n} u}.$$
 \end{proof}

%% 7.14
\begin{proposition}\label{finlg}  Let $\rmk$  be a local ring and let $\cM = \{M_n\}_n$ be a
sequence of finitely modules over $R$. Let $A \inc A'  \inc B$ be finitely
generated $R$-modules.  Suppose that $A'/A$ happens to have finite length.  Then
a sufficient condition for $A' \inc A\cla_B$  is that
$$ \quad\ell\bigl(\Img( M_n \otimes_R A'/A \to M_n \otimes_R B/A)\bigr) =
\oo\bigl(\alpha(M_n)\bigr).$$
\end{proposition}
\begin{proof}  As usual, we may assume that $A =0$ and that $A' \inc B$ has finite length.  We
need to show that  for all $t \gg 0$,  $\ell\big(\Img (M_n \otimes_R (A' + \m^tB)/\m^tB \to M_n \otimes_R B/\m^t B)\big) =
 \oo\big(\alpha(M_n)\big)$.   But the left hand side is at most $\ell\big(\Img(M_n \otimes_R A' \to M_n \otimes_R B)\big)$,
 since there is a surjection
 $$\Img(M_n \otimes_R A' \to M_n \otimes_R B) \surj
 \Img\big(M_n \otimes_R (A' + \m^tB)/\m^tB \to M_n \otimes_R (B/\m^tB)\big),$$ induced by the surjection
 $M_n \otimes_R B \surj  M_n \otimes_R (B/\m^tB)$.
\end{proof}

%% 7.15
\begin{proposition}\label{resscal} Let $(R,\, \m) \to (S,\, \n)$ be a local homomorphism such
that $S$ is a module-finite
extension of the $R$,  let $\cM$ be a sequence of nonzero finitely generated modules over $S$
(hence, also, over $R$, by restriction of scalars), let $A \inc B$ be $R$-modules and let
$u \in B$. Let $\alpha_R$ and $\alpha_S$ be functions from the set $\{M_n: n \in \N_+\}$ to $\N_+$
such that there exist positive constants $c_1$ and $c_2$ such that
$$c_1\alpha_S(M_n) \leq \alpha_R(M_n) \leq c_2 \alpha_S(M_n)$$
for all $n \geq 1$.
Then $u$ is in the $\cM$-closure of $A$ in $B$ with respect to $\alpha_R$,
working over $R$, if and only if the image $1 \otimes u$ of $u$ in $S \otimes_RB$ is in the
$\cM$-closure of the image of $S \otimes_R A$ in $S \otimes_R B$, working over $S$.

In particular,  the result holds if
\bena
\item $\alpha_R =  \nu_R$ and $\alpha_S = \nu_S$,  or  %% (a)
\item $R$ and $S$ are domains of dimension $d$, the $M_n$ have dimension $d$,
$\alpha_R = \rank_R$, and $\alpha_S = \rank_S$. %% (b)
\een
\end{proposition}
\begin{proof} Since the powers of $\m S$ are cofinal with the powers of $\n$,  it suffices to
check this for every $t$ when $A$ is replaced by $A + \m^tB$.  Hence, we may assume that
$B/A$ has finite length over $R$.  Moreover, as usual, we may then assume that $A = 0$ and that $B$ itself has
finite length over $R$. Note that $M_n \otimes_R B \cong M_n \otimes_S (S \otimes_R B)$,  and
that the images of $N_n = M_n \otimes_S S\otimes_R Ru$ and $ M_n \otimes_S Ru$ may be identified.
The result now follows from the fact that $\ell_R(N_n) = [L:K]\ell_S(N_n)$ and
the existence of the constants $c_1, c_2$.

In case $\alpha = \nu$,  we have $\nu_S(M_n) \leq \nu_R(M_n) \leq \nu_R(S) \nu_S(M_n)$,  so
that we may take $c_1 = 1$ and $c_2 = \nu_R(S)$. If $R$ and $S$ are domains and $\alpha = \rank$,
we have that $\rank_R(M_n) = \rank_R(S)\rank_S(M_n)$.
\end{proof}

We remind the reader of the treatment of Hilbert-Samuel multiplicity in  Discussion~\ref{mult}. A local
ring $\rmk$ is called {\it equidimensional} if $\dim(R/\fp) = \dim(R)$ for every minimal prime $\fp$ of $R$,
and is called {\it formally equidimensional} (or {\it quasi-unmixed}) if its $\m$-adic completion
$\wh{R}$ is equidimensional.  An excellent equidimensional local ring, e.g., an excellent local domain,
is formally equidimensional.

%% 7.16
\begin{theorem}\label{cloint} Let $\rmk$  be a formally equidimensional local ring and let $\cM = \{M_n\}_n$ be a
sequence of finitely generated modules over $R$ for which $\rank$ is defined and nonzero.
Then for every ideal $I$ of  $R$, $I ^{*\cM, \rank} \subseteq\ov{I}$,  the integral closure of $I$.  
That is, integrally closed ideals of $I$
are $\cM$-closed with respect to rank.  Hence, radical ideals are $\cM$-closed with respect to
rank, and so prime ideals are $\cM$-closed with respect to rank.
\end{theorem}
\begin{proof}  The integral closure of $I \inc R$ is the intersection of the integral closures of the
$\m$-primary ideals containing $I$, by \cite[Cor. 6.8.5]{SwHu06}.
Hence, it suffices to prove the result when $I$ is $\m$-primary and integrally closed, $I = I ^{*\cM, \rank}$.
If not, we can choose an element $u \in I ^{*\cM, \rank} - I$. By Example~\ref{exclo}, we have that $u \in \m$.  Let
$J = I + Ru$.  It suffices to prove that if $J \inc I ^{*\cM, \rank}$,  then the multiplicities
of $R$ with respect to $I$ and $J$ are equal, for then $J \inc \ov{I}$ by \cite{Rees61} or \cite[Theorem 11.3.1]{SwHu06}.
Since $I \inc J$, $e(I;\,R) \geq e(J;\,R)$ while, by Proposition~\ref{multrank}
 $$e(I;\, R) - e(J; \, R) =
\frac{e(I;\, M_n) - e(J; M_n)}{\rank(M_n)} =
d!\lim_{t \to \infty} \frac{\ell(M_n/I^tM_n) - \ell(M_n/J^tM_n)}{\rank(M_n)\,t^d}.$$
Consequently, we have %using the second asymptotic formula for multiplicities in subsection~\ref{mult},
$$(\dagger)\quad  0 \leq e(I;\, R) - e(J; \, R) = d!\lim_{t \to \infty}\frac{\ell(J^tM_n/I^t M_n)}{\rank(M_n)\,t^d}.$$
Note that $$J^t = (I + uR)^t = I^t + I^{t-1}u + \cdots + I^{t-s}u^s + \cdots + Ru^t.$$
Let  $\fA_0 = I^t$ and
$\fA_s = \sum_{j=0}^s I^{t-j}u^j$, so that $\fA_t = J^t$.  Then
$\ell(J^t M_n/ I^t M_n) = \sum_{s= 1}^t \ell(\fA_s M_n/\fA_{s-1}M_n)$.
We next want to prove that
$$(*) \quad \ell\biggl(\frac{\fA_sM_n}{\fA^{s-1} M_n}\biggr) \leq \nu(I^{t-s})\ell\biggl(\frac{JM_n}{IM_n}\biggr).$$
First observe that $\fA_s$ is generated over $\fA_{s-1}$ by $I^{t-s}u^s$, an ideal with $\nu(I^{t-s})$
generators.  Hence, $\fA_sM_n/\fA_{s-1}M_n$ is generated by the image of
$I^{t-s}u^s M_n$.  This is contained in $I^{t-s}u^{s-1}JM_n$,  while the denominator
contains $I^{t-s}u^{s-1}(IM_{n}) = I^{t-(s-1)}u^{s-1}M_n$.  Consequently,
$$(**) \quad \ell\biggl(\frac{\fA_sM_n}{\fA_{s-1}M_n}\biggr)
\leq \ell\biggl(\frac{I^{t-s}u^{s-1}JM_n}{I^{t-s}u^{s-1}IM_n}\biggr) \leq \ell\biggl(\frac{I^{t-s} JM_n}{I^{t-s} IM_n}\biggr),$$
where the inequality on the right follows because the surjection $I^{t-s} JM_n \to I^{t-s}u^{s-1}JM_n$ given by multiplication
by $u^{s-1}$ maps $I^{t-s} IM_n$ onto $I^{t-s}u^{s-1}IM_n$ and so induces a surjection
$$\frac{I^{t-s} JM_n}{I^{t-s} IM_n} \surj  \frac{I^{t-s}u^{s-1}JM_n}{I^{t-s}u^{s-1}IM_n}.$$
We have a surjection $R^{\nu(I^{t-s})} \surj I^{t-s}$ and, hence, we have
$$R^{\nu(I^{t-s})} \otimes \biggl(\frac{JM_n}{IM_n}\biggr)\surj
I^{t-s} \otimes \frac{JM_n}{IM_n} \cong \frac {I^{t-s}\otimes (JM_n)}{\Img\bigl(I^{t-s} \otimes (IM_n)\bigr)} \surj
\frac{I^{t-s}JM_n}{I^{t-(s-1)}M_n},$$
which shows that
$$\ell\biggl(\frac{I^{t-s}JM_n}{I^{t-(s-1)}M_n}\biggr) \leq
\ell\biggl(R^{\nu(I^{t-s})} \otimes \frac{JM_n}{I M_n}\biggr)
= \nu(I^{t-s})\ell\biggl(\frac{JM_n}{IM_n}\biggr).$$
Along with $(**)$, this establishes the inequality $(*)$ asserted above.   Hence,
$$(***) \quad \ell\biggl(\frac{J^tM_n}{I^tM_n}\biggr) \leq
\biggl(\sum_{s=1}^t \nu(I^{t-s})\biggr)\ell\biggl(\frac{JM_n}{IM_n}\biggr).$$

We have that $\sum_{s=1}^t \nu(I^{t-s}) = \sum_{s=1}^t \dm_K (I^{t-s}/\m I^{t-s}) = \sum_{j=0}^{t-1}H(j)$,
where $H$ is  Hilbert function of $(R/\m) \otimes_R \gr_IR$, and so coincides with a polynomial of degree $d-1$
in $t$ for all $t \gg 0$.    It follows that there is a constant $C > 0$
such that $\sum_{s=1}^t \nu(I^{t-s}) \leq Ct^d$  for all $t$, so that from $(***)$ we obtain:
$$\ell\biggl(\frac{J^tM_n}{I^tM_n}\biggr) \leq Ct^d\,\ell\biggl(\frac{JM_n}{IM_n}\biggr),$$ for all $t \geq 1$, where
$C$ is independent of $t$ and $n$.  Using this fact and $(\dagger)$ we have that for all $n \geq 1$,
$$ 0 \leq e(I;\, R) - e(J;\, R) \leq d! \lim_{t \to \infty} \frac{Ct^d \,\ell\bigl(JM_n/IM_n)}{\rank(M_n)\,t^d}
 = Cd! \, \frac{\ell(JM_n/IM_n)}{\rank(M_n)}.$$
Since  $\ell(JM_n/IM_n) = \oo\bigl(\rank(M_n)\bigr)$, we have that  $e(I,\,R) = e(J,\, R)$, as required. \end{proof}

%% 7.17
\begin{proposition}\label{omcl} Let $\rmk$  be a reduced formally equidimensional local ring and let $\cM = \{M_n\}_n$
be a sequence of finitely generated $R$-modules for which rank is defined and nonzero. Then  $\m$ is $\cM$-closed
with respect to $\nu$ and $\rank$, while $(0)$ is $(\cM, \rank)$-closed. If $\cM$ is (weakly) lim \CM,  then $0$ is
$\cM$-closed with respect to both $\nu$ and rank.\end{proposition}
\begin{proof} The result for $\m$ was proved in Example~\ref{exclo}(2).  The result that $(0)$ is $(\cM, \rank)$ closed
follows from Theorem~\ref{cloint} and the fact that when $R$ is reduced, the ideal (0) is integrally closed
(more generally, its integral closure is the ideal of all nilpotent elements).  The final statement then follows
from Lemma~\ref{Cbound}.
 \end{proof}

\section{Capturing  Koszul homology and parameter colon ideals with lim Cohen-Macaulay closures.}\label{cap}
%% 8

In this section, we prove that the closures coming from lim \CM sequences of modules share many of the
``colon-capturing" properties that tight closure has.  The results we obtain are sufficiently strong to enable
us to prove, for example, that a domain that has a  lim \CM sequence of modules has a big \CM module.

%% Theorem 8.1
\begin{theorem}\label{thmcappar}  Let $\rmk$ be  local ring of
Krull dimension $d$. Let $\cM$ be a lim Cohen-Macaulay sequence of modules.
\bena
\item If $\ux  = \vect x k$ is part of a system of parameters for $R$ and $\cK_\bu(\ux;\, R)$ is the Koszul complex,
then the cycles in $\cK_i(\ux;\, R)$ are in the $\cM$-closure of the boundaries in $\cK_i(\ux;\, R)$ with respect to
$\nu$ for all $i \geq 1$.    %% (a)
\item If $\vect x {k+1}$ is part of a system of parameters for $R$, then the ideal $(\vect x k)R:_R x_{k+1}$ is contained
in the $\cM$-closure of $(\vect x k)R$ with respect to $\nu$.  %% (b)
\een
If $R$ is a domain or, more generally, if rank is defined for all of the modules in $\cM$, the same results hold for
 $\cM$-closure with respect to rank.
\end{theorem}
\begin{proof}
(a)  If we have a full system of parameters $\ux  = \vect x d$, we have that
 $\ell\bigr(H_i(\ux; M_n)\bigl) = \oo\bigl(\nu(M_n)\bigr)$ for all $i \geq 1$
by the definition of lim \CM sequence of modules.

If we have a part of a system of parameters, say $\vect xk$
where $k < d$, we wrtie $_k\cK_i$ for $\cK_i(\vect x k; \, R)$.
We may extend $\vect x k$ to a full system of parameters $\vect x d$. Let $t$ be
a positive integer and consider the
Koszul complex  $\cK^{(t)}_\bu = \cK_\bu(\vect x k, \,x_{k+1}^t, \, \ldots, \, x_d^t;\,R)$.
We may think of this Koszul complex as the exterior algebra over $R$ of  a free
module $\cK_1^{(t)} = Ru_1 \oplus \cdots \oplus Ru_d$ such that $u_i \mapsto x_i$
(respectively, $x_i^t$) if $i \leq k$ (respectively, $i > k$).  The sublagebra generated
by $Ru_1 \oplus \cdots \oplus Ru_k$  gives the Koszul complex $_k\cK_\bu = \cK_\bu(\vect x k;\, R)$, which
we think of as a subcomplex.    We want to show for $i \geq 1$ that every cycle $z \in {}{_k\cK_i}$
is in the $\cM$-closure of $B = \Img(_k\cK_{i+1})$ with respect to $\nu$.  We may think of the same
element as a cycle in $\cK_i^{(t)}$. Here, we know that it is in the $\cM$-closure of $\Img(\cK_{i+1}^{(t)})$ with
respect to $\nu$.  Consider the standard generators for $\bigwedge^{i+1}(Ru_1 + \cdots + Ru_d)$
consisting of elements $v = u_{j_1} \wedge \cdots \wedge u_{j_{i+1}}$.  If the $j_h$ occurring are all at most $k$,
the image of $v$ is an element of $B$, and we get generators of $B$ over $R$ this way.  If any of the $j_h$
is $>k$, the image of $v$ is in $(x_{k+1}^t, \, \ldots, x_d^t)\cK^{(t)}_i \inc \m^t\cK^{(t)}_i$.  Hence,
$\Img(\cK^{(t)}_{i+1}) \inc B + \m^t \cK^{(t)}_i$,  and so  $z \in \bigl(B + \m^t \cK^{(t)}_i\bigr) ^{*\cM, \nu}_{\cK^{(t)}_i}$.
Let $\theta: \cK^{(t)}_i \to {}{_k\cK_i}$ be the $R$-module retraction that fixes all of the standard
generators of $u_{j_1} \wedge \cdots \wedge u_{j_i+1}$ of $\cK^{(t)}_i$ such that all of the $j_h \leq k$
and sends all of the other standard generators to 0.  By Proposition~\ref{cloprop}(c), the image of $z$, which is $z$,
is in the $\cM$-closure of  $\theta(B + \m^t \cK^{(t)}_i)$ with respect to $\nu$, and so
$z \in (B + \m^t \, _k\cK_i) ^{*\cM, \nu}_{_k\cK_i}$.  Since this is true for every $t$,  it follows that
$c \in B^{*\cM, \nu}_{_k\cK_i}$, as claimed.

(b) Suppose that $rx_{k+1} = \sum_{j=1}^k r_jx_j$.  Then $(-r_1, \,\ldots, \, -r_k, r)$ is a cycle
in $\cK_1(\vect x {k+1};\, R)$,  and is in the $\cM$-closure of the trivial Koszul relations
with respect to $\nu$.  Let $\pi: \cK_1(\vect x {k+1};\, R) \to R$ be projection on the last
coordinate.  The image of the trivial Koszul relations under this map is the ideal
$(\vect x k)R$,  and so, by Proposition~\ref{cloprop}(c),
the image $r$ of $z$ maps into the $\cM$-closure of this ideal with
respect to $\nu$.

The final statement holds because the closure with respect to rank contains the closure with
respect to $\nu$.
\end{proof}

%% 8.2
\begin{remark} Colon-capturing results for tight closure typically require that the elements $\vect x {k+1}$ be part of
a system of parameters modulo every minimal prime of $R$.  But note that the closure given, for example, by the
sequence $F^n_*(R)$ in the F-finite case agrees with tight closure {\it when $R$ is equidimensional}, but
not in general.  In non-equidimensional cases lim \CM closures may have advantages over tight closure,
as in Example~\ref{nonequi}.
\end{remark}

%%  8.3
\begin{proposition}\label{torcap}  Let $\rmk$ be a local ring of Krull dimension $d$ and let $\cM_n$ be
a lim Cohen-Macaulay sequence of modules over $R$.  If $\ux = \vect x d$ is a
system of parameters for $R$, then
$\ell\bigl(\Tor_1^R\bigl(R/(\ux),\, M_n\bigr)\bigr) = \oo\bigl(\nu(M_n)\big)$. \end{proposition}
\begin{proof} We have a free resolution of $R/(\ux)$ that begins $R^{b} \to R^d \to R \to 0$
where the map $R^d \to R$ sends $(\vect r d) \mapsto \sum_{i=1}^d r_ix_i$ and the image
$B$ of the map $R^b \to R^d$ contains the submodule generated by the standard Koszul
relations on $\ux$.  $\Tor^R_1(R/(\ux),\,M_n)$ is the homology at the middle spot
of $M_n^b \to M_n^d \to M_n$, and this implies that we have a surjection $H_1(\ux;\,M_n)
\surj \Tor_1^R\bigl(R/(\ux),\,M_n\bigr)$.  The result is now immediate from the definition of
lim \CM sequence.
 \end{proof}

We can improve the results on colon ideals involving a system of parameters as follows.

%% Theorem 8.4
\begin{theorem}\label{strcap} Let $\rmk$ be a local ring and $\cM = \{M_n\}_n$ a lim Cohen-Macaulay
sequence over $R$.   Let $\vect x k$ be part of a system of parameters for $R$,
let $\vect a k$ be positive integers, and let $\vect b k$ be nonnegative integers.
Then  $(x_1^{a_1+b_1}, \, \ldots, x_k^{a_k+b_k}) ^{*\cM, \nu} : x_1^{b_1} \cdots x_k^{b_k} =
(x_1^{a_1}, \, \ldots, x_k^{a_k}) ^{*\cM, \nu}$.  The same result holds for $\cM$-closures with
respect to rank when all of the modules $M_n$ have a well-defined rank.  \end{theorem}
\begin{proof} The final statement about rank follows from Lemma~\ref{Cbound} once we have
established the result for $\nu$. 
%Let $\alpha$ be either $\nu$ or $\rho$.  
Throughout this proof we
will write $\blank\cln$ for $\blank^{*\cM, \nu}$.
We first want to reduce to the case where  $k = d$.  For the moment,
assume this case.  Extend $\vect x k$ to a full system of parameters $\ux = \vect x d$
and fix a positive integer $t$. Let $I = (x_1^{a_1+b_1}, \, \ldots, x_k^{a_k+b_k})R$ and
$\mu = x_1^{b_1} \cdots x_k^{b_k}$. Then
$I\cln:\mu \inc \bigl(I + (x_{k+1}^t, \, \ldots, x_d^t)\bigr)\cln : \mu$, which, by our assumption,
is contained
$J_t = (x_1^{a_1}, \, \ldots, x_k^{a_k}, x_{k+1}^t, \, \ldots, x_d^t)\cln$. By the definition of
${}\cln$,  the intersection of the $J_t$ as $t$ varies is  $(x_1^{a_1}, \, \ldots, x_k^{a_k})\cln$.
We have thus reduced to the case where $k = d$.

We next observe that we can reduce to the case where only one of the $x_j$ occurring
in $\mu$ has a positive exponent, and that exponent is one.  We use induction on the
total degree of $\mu$ in the $x_j$.  If the degree is larger than one, write $\mu = \mu_0 x_j$,
and then the result is immediate from the induction hypothesis and the fact that for
any ideal $\fA$ and two elements $f,\,g$ of the ring,  $\fA:_R fg = (\fA:_R f):_R g$,  applied
with $\fA = I\cln$,  $f = \mu_0$, and  $g = x_j$.  By renumbering the parameters, we may
assume that $\mu = x_d$.  Since $x_1^{a_1}, \, \ldots, \, x_{d-1}^{a_{d-1}},\,x_d$ is simply
another system of parameters, we may change notation, and assume that $a_i = 1$
for $i < d$.

Consequently, all we need to prove is that
$$(x_1,\, x_2, \, \ldots, \,x_{d-1}, \, x_d^{a+1})\cln:x_d =
(x_1,\, x_2, \, \ldots,\, x_{d-1}, x_d^a)\cln$$ for $a \geq 1$.
Since $\supseteq$  follows from part (j) of  Theorem~\ref{cloprop}
 it remains to show $\inc$.

 Let $J=(x_1,\dots,x_{d-1}, x_d^{a+1})$. We consider the following commutative diagram with exact rows and columns:
$$\CD
 0 @>>>{\disp \frac{R}{J:x_d}}  @>{\cdot x_d}>>  {\disp \frac{R}{J}}  @>>>
 {\disp \frac{R}{J+x_dR}} @>>> 0\\
  @.                                   @V{\eta}VV                       @V{\theta}VV           @VVV   @.\\
 0 @>>> {\disp \frac{R}{J\cln:x_d}}  @>{\cdot x_d}>> {\disp \frac{R}{J\cln}} @>>>
{\disp   \frac{R}{J\cln+x_dR}} @>>> 0\\
      @.                  @VVV                                           @VVV                      @VVV\\
   {}       @.           0                        @.                          0               @.                0
 \endCD$$
We apply $\blank \otimes_R M_n$ to the  commutative diagram above. Note that $J+x_dR =
(\ux)R$. Let $N_n$ denote $\Tor_1^R\big(R/(\ux), M_n\big)$.  Then we have the following commutative diagram with exact rows and columns:
$$\CD
{}       @.         0               @.              0\\
@.                 @VVV                   @VVV\\
{}     @.       \Ker(\eta \otimes M_n)    @>\gamma>>  \Ker(\theta \otimes M_n)\\
@.                  @VVV                     @VVV\\
 N_n @>>>{\disp \frac{R}{J:x_d}\otimes M_n}  @>{\cdot x_d\otimes M_n}>>  {\disp \frac{R}{J}\otimes M_n}  @>>>
 {\disp \frac{R}{J+x_dR}\otimes M_n} @>>> 0\\
  @.                                   @V{\eta\otimes M_n}VV                       @V{\theta\otimes M_n}VV           @VVV   @.\\
   {}  @. {\disp \frac{R}{J\cln:x_d}\otimes M_n}  @>{\cdot x_d\otimes M_n}>> {\disp \frac{R}{J\cln}\otimes M_n} @>>>
{\disp   \frac{R}{J\cln+x_dR}\otimes M_n} @>>> 0\\
      @.                  @VVV                                           @VVV                      @VVV\\
   {}       @.           0                        @.                          0               @.                0
 \endCD$$
\quad\\
We have that $\ell\bigl(\Ker(\eta\otimes M_n)\bigr) = \ell\bigl(\Img(\gamma)\bigr) + \ell\bigl(\Ker(\gamma)\bigr)
\leq \ell\bigl(\Ker(\theta\otimes M_n)\bigr) + \ell(N_n)$,  since $\Ker(\gamma)$ is the intersection of $\Ker(\eta\otimes M_n)$
with the image of $N_n$ and so is a subquotient of $N_n$.

By the definition of lim Cohen-Macaulay closure,
$\ell\bigl(\Ker(\theta\otimes M_n)\bigr)=\oo\bigl(\nu(M_n)\bigr)$, and by Corollary~\ref{torcap},
$\ell(N_n) = \ell\Big(\Tor_1^R\big(R/(\ux), M_n\big)\Big)=\oo\bigl(\nu(M_n)\bigr)$. Hence 
$\ell\bigl(\Ker(\eta\otimes M_n)\bigr)=\oo\bigl(\nu(M_n)\bigr)$, so that by the exactness of the leftmost column we have
$$ (*) \quad\ell\biggl(\frac{R}{J:x_d}\otimes M_n\biggr)-\ell\biggl(\frac{R}{J\cln:x_d}\otimes M_n\biggr)=\oo\bigl(\nu(M_n)\bigr).$$

Note that if $r \in J:x_d$,  then $x_dr \in (\vect x {d-1},\, x_d^{a+1})$, so there exists $r'\in R$ such that
$$
x_dr - x_d^{a+1}r' = x_d(r - x_d^ar') \in (\vect x {d-1}), \ \ \mathrm{and}
$$
$$r - (x_d^a r') \in (\vect x {d-1}):x_d \inc (\vect x {d-1})\cln$$
by Theorem~\ref{thmcappar}(b). Therefore, we have that
$$r \in (\vect x {d-1})\cln + x_d^aR \inc (\vect x {d-1},\, x_d^a)\cln.$$  Thus,
$J:x_d \inc (x_1,\dots,x_{d-1}, x_d^a)\cln$. Consequentlly,
$$(**)\quad\ell\bigg(\frac{R}{(x_1,\dots,x_{d-1}, x_d^a)}\otimes M_n\bigg)-
\ell\bigg(\frac{R}{J:x_d}\otimes M_n\bigg)=\oo\big(\nu(M_n)\big).$$
By subtracting  $(**)$ from $(*)$ we obtain:
$$\ell\bigg(\frac{R}{(x_1,\dots,x_{d-1}, x_d^a)}\otimes M_n\bigg)-\ell\bigg(\frac{R}{J\cln:x_d}\otimes M_n\bigg)= \oo\big(\nu(M_n)\big).$$
This shows that $J\cln:x_d\subseteq (x_1,\dots,x_{d-1}, x_d^a)\cln$, as required.
\end{proof}

As a corollary, we obtain a  proof of the direct summand conjecture for
rings that have a lim Cohen-Macaulay sequence, and we also obtain the result that
over a regular local ring, for every choice of $\cM$, every submodule of every
finitely generated module is $\cM$-closed with respect to $\nu$ and rank.

%%Corollary 8.5
\begin{corollary}\label{clsmall}  Let $\rmk$ be a local ring of Krull dimension $d$, and let $\cM = \{M_n\}_n$ be a
lim Cohen-Macaulay sequence of modules for $R$.  Let ${}\cln$ denote $\cM$-closure with respect to
either $\nu$ or rank (if rank is defined for the $M_n$). Let
$\vect x d$ be any system of parameters for $R$. Let $I_t =
(x_1^t,\, \ldots, x_d^t)R$.  Then for all $t \geq 1$,  $$I_t\cln: (x_1\cdots x_d)^{t-1} \inc (\vect x d)\cln \inc \m,$$
and so  $(x_1\cdots x_d)^{t-1} \notin I_t\cln$.
Hence:
\bena
\item If $T  \to R$ is a module-finite local map from a regular local ring $T$,  then $T$ is a direct summand of
$R$ as an $T$-module.
\item If $R$ is regular, the $\cM$-closure of every submodule $A$ of every finitely generated
module $B$ is $A$, both with respect to $\nu$ and $\rank$.
\een
\end{corollary}
\begin{proof}  The first inclusion in the displayed line is immediate from Theorem~\ref{strcap}, while
the second inclusion follows from  Proposition~\ref{omcl}.

(a) Choose a regular system of parameters $\vect x d$ for $T$:  it is also a system of parameters
for $R$.  By \cite{Ho73b}, to prove that $T \to R$ splits, it suffices to show that for all $t \geq 1$,
$(x_1\cdots x_d)^{t-1} \notin I_t$.

(b) It suffices to prove that $0$ is closed in each module $B$ of finite length, and $B$
will embed in a finite direct sum of copies of the injective hull of $K$ over $R$ and, hence, in the
direct sum of finitely modules of the form $R/(x_1^t, \cdots, x_d^t)$. By Proposition~\ref{cloprop}(d),
it suffices to show for each summand that a socle generator is not in the closure of 0. This is true
because for all $t$,  $(x_1\cdots x_d)^{t-1} \notin I_t\cln$ in $R$.
\end{proof}

\section{The Dietz axioms for lim \CM closures and big \CM modules}\label{big} %%9

In this section we show that the existence of a lim Cohen-Macaulay sequence of modules
for a complete local domain implies that the compete local domain has a big Cohen-Macaulay module.
We follow the strategy of \cite{Di10}, where it is shown that if a closure operation on submodules of
finitely generated modules satisfies a certain set of seven axioms,  it follows that the ring
has a big Cohen-Macaulay module.  If $R$ is local, we can choose a minimal prime $\fp$ of
$\wh{R}$ so that $\wh{R}/\fp$ has the same dimension of $R$, and then elements of $R$ form
a system of parameters only if their images in $\wh{R}/\fp$ form a system of parameters.
Moreover we can make a local extension $\wh{R}/\fp \to S$ so that $S$ is complete, has a perfect
(or even algebraically closed) residue class field, and the closed fiber is simply the residue class
field of $S$.  A big Cohen-Macaulay module for $S$ is also a big Cohen-Macaulay module for $R$.
Thus, our results show that if every complete local domain with perfect (or even algebraically closed)
residue class field has a lim Cohen-Macaulay sequence of modules, then every local ring
has a big Cohen-Macaulay module.

Given a closure operation for finitely generated modules and their submodules over $R$,  one
gets a notion of {\it phantom extension} of $R$ as follows (cf.~\cite{HH94a} for the case of tight closure
and \cite{Di10}).  An injection
$f:R \to B$, where $B$ is a finitely generated $R$-module, is called a {\it phantom extension} with respect
to the closure operation if the following
condition holds.  Let $C := B/\Img(f)$.  We have a short exact sequence $(*)\quad 0 \to R \to B \to C \to 0$.
Choose a projective resolution $G_\bu$ of $C$ by finitely generated free $R$-modules, say
$\cdots G_2 \to G_1 \to G_0 \to 0$ of $C = B/\Img(f)$,
so that $\Ext_R^1(C,\,R)$ may be viewed as $Z_1/B_1$  where
$$Z_1 = \Ker\bigl(\Hom_R(G_1,\, R) \to \Hom_R(G_2,R)\bigr)\ \ \mathrm{and}$$
$$B_1 =\Img\bigl(\Hom_R(G_0,\, R) \to \Hom_RG_1,R)\bigr).$$
Then $f:R \to B$ is a {\it phantom extension} with respect to the given closure operation
if an element of $Z_1$  that maps to the element $\eta \in \Ext^1_R(C,\,R)$ corresponding to the
exact sequence $(*)$ is in the closure of $B_1$ in $\Hom_R(G_1,\, R)$.  This condition turns out
to be independent of the choice of $G_\bu$ and of the choice of element that maps to  $\eta$.

%% 9.1
\begin{discussion} {\bf The Dietz axioms.}
Let $(R,\m)$ be a fixed complete local domain.  Let ${}\cln$ denote a closure operation
over $R$ which assigns to every  $R$-submodule $A$ of a finitely generated $R$-module $B$
a submodule $A\cln_B$ of $B$. Let $A$, $B$, and $C$ be arbitrary finitely generated $R$-modules with
$A \inc B $. By the {\it Dietz axioms} we mean the following seven conditions on ${}\cln$.
\ben
\item $A\cln_B$ is a submodule of $B$ containing $A$.
\item $(A\cln_B)\cln_B = A\cln_B$, i.e., the  ${}\cln$-closure of $A$ in $B$ is closed in $B$.
\item If $A \inc B \inc C$, then $A\cln_C\inc B\cln_C$.
\item Let $\theta : B \to C$ be an $R$-linear homomorphism. Then $\theta(A\cln_B) \inc \theta(A)\cln_C$.
\item If $A\cln_B = A$, then $0\cln_{B/A} = 0$.
\item The ideals $\m$ and $0$ are ${}\cln$-closed in $R$; i.e., $\m\cln_R = \m$ and $0\cln_R = 0$.
\item  Let $\vect x {k+1}$ be part of a system of parameters for $R$, and let $J =
(\vect x k)R$. Suppose that there exist a surjective $R$-linear homomorphism $f : B \surj R/J$  and $v \in B$
such that $f(v) = x_{k+1} + J \in R/J$. Then
$(Rv)\cln_B \cap \Ker(f) \inc (Jv)\cln_B$.
\een
\end{discussion}

The main result of \cite{Di10} is:

%% 9.2
\begin{theorem}[\bf Dietz] A complete local domain $R$ has a big Cohen-Macaulay module if and only
if it has a closure operation on submodules of finitely generated modules satisfying axioms
(1)--(7) above. \end{theorem}

We comment on the proof.  If there is a big Cohen-Macaulay module $\fB$, one may define
a closure operation satisfying (1)--(7) by letting $A\cln_B$ be the pullback to $B$ of
$\Ker\bigl(B/A \to (B/A) \otimes_R \fB\bigr)$.  In the other direction, one uses the same idea
as in \cite{HH94a}.  One starts with the phantom extension (with respect to ${}\cln$) $R \arwf{\id} R$.
One then shows:
\ben[label=$(\dagger)$]
\item If $R \to B$ is phantom with respect to
${}\cln$ and $x_{k+1}m = \sum_{i=1}^k x_im_i$
is a relation on part $\vect x {k+1}$ of a system of parameters for $R$ with coefficients in $B$,
then the composite map $$(*) \quad R \to B \to \frac{B \oplus R^k}{R(m,\, -x_1,\, \ldots, -x_k)}$$
is again phantom.
\een

See \cite{Di10} for a detailed treatment.
The rather subtle Dietz axiom (7) plays a critical role in the proof that condition $(\dagger)$ holds.
The map $$B \to \frac{B \oplus R^k}{R(m,\, -x_1,\, \ldots, -xk)}$$
is referred to as a {\it modification}.
One then shows that the direct limit $\fB$ of a carefully chosen family of
modules $B$, each obtained from a finite sequence of successive modifications $R \to B_1 \cdots \to B_n =:B$,
is a big \CM module for $R$ provided that the image of $1 \in R$ in $B$ is not in $\m B$.  For this, it suffices
to prove that the image of  $1 \in R$ is not in $\m B$ for every map $R \to B$ when $B$ is obtained
by a successive sequence of modifications.  From $(\dagger)$ and mathematical induction one knows that the maps
$R \to B$ obtained in this way are all phantom extensions, and one shows that for a phantom extension
$R \to B$ one cannot have that the image of $1 \in R$ is in $\m B$.

We shall prove below that if $\cM$ is a lim Cohen-Macaulay sequence of modules,
then the associated closure operation over the complete local domain $R$, either with respect to rank or with respect
to $\nu$,  satisfies the Dietz axioms (1)--(7).  It follows at once
that if $R$ has a lim Cohen-Macaulay sequence, then $R$ has a big Cohen-Macaulay
module.   In fact, we have already proved that the first six axioms hold:

%% 9.3
\begin{proposition}  Let $\rmk$ be a local ring.  The Dietz axioms (1) --- (5) inclusive hold for every
closure operation on $R$ with respect to a sequence of $R$-modules.  Moreover,
if $R$ is reduced  and formally equidimensional and $\cM$ is a (weakly)  lim \CM sequence
of modules for which rank is defined, then axiom (6) holds for $\cM$-closure with respect to $\nu$ or rank,
which are the same.  In particular, (6) holds $(\cM,\nu)$-closure (equivalently, $(\cM, \rank)$-closure
for every (weakly) lim \CM sequence of modules over a complete local domain. \end{proposition}
\begin{proof} Axioms (1), (2), (3), (4), and (5) follow from Theorem~\ref{cloprop} parts
(b), (f), (d), (c), and (a), respectively, while the statements about (6) follow from
Proposition~\ref{omcl}. \end{proof}

We are now ready to prove one of the main results of this section:

%% 9.4
\begin{theorem} Let $\rmk$ be a complete local ring and suppose that $R$ has a lim \CM sequence $\cM$ of
$R$-modules.  Then closure with respect to $\nu$ or rank (if rank is defined on the modules in $\cM$)
is a Dietz closure, and so $R$ has
a big \CM module. \end{theorem}
\begin{proof} In this argument, all tensor products are taken  over $R$.
Let $d$ be the Krull dimension of $R$ and let ${}\cln$ indicate $\cM$-closure with respect
to $\nu$ or rank (if it is defined on $\cM$), which are the same. It remains only to prove the Dietz axiom (7).
Let $\vect x{k+1}$ be part of a full system of parameters $\ux = \vect x d$ for $R$,
and let $J=(\vect x k)R$.
Suppose that we have a surjective homomorphism $f:B\to R/J$ and $v \in B$ such that
$f(v)=x_{k+1}+J$. We need to show that $(Rv)\cln_B\cap \Ker{f}\subseteq (Jv)\cln_B$

We first prove this for $k = d-1$. In this case $J=(x_1,\dots,x_{d-1})$and $f(v)=x_d$. Let $N=\Ker f$, so we have an exact sequence: \[0\to \frac{N+Rv}{Rv}\to \frac{B}{Rv}\to\frac{R}{(x_1,\dots,x_d)}\to 0.\] Since $(x_1,\dots,x_d)$ is $\m$-primary, for $t\gg 0$, $\m^tB$ maps to $0$ in
$R/(x_1,\dots,x_d)$. Let $y\in(Rv)\cln_B\cap N$. Since $y\in N$, the image of $y$ in $R/(x_1,\dots,x_d)$ is also $0$. Hence when $t\gg0$, we have the following commutative diagram with exact rows:
\[  \xymatrix{
   0 \ar[r] &{\disp \frac{N+Rv}{Rv+\m^tB}}  \ar[r]\ar@{->>}[d] & {\disp  \frac{B}{Rv+\m^tB} }  \ar[r]\ar@{->>}[d]
    & {\disp \frac{R}{(\ux)} } \ar[r]\ar[d]^{\id} & 0\\
   0 \ar[r] & \disp{\frac{N+Rv}{Ry+Rv+\m^tB}}  \ar[r] & \disp{\frac{B}{Ry+Rv+\m^tB}}  \ar[r]
   & \disp{\frac{R}{(\ux)}}  \ar[r] & 0%%
} \]
Applying $M_n \otimes \blank$ to the commutative diagram above, we have the following commutative diagram with exact rows:

\[  \xymatrix@C=9pt{  {\disp \Tor_1^R\!\Big(\!M_n,\frac{R}{(\ux)}\Big)}\ar[r]\ar[d]_{\id}&\disp{M_n \otimes \frac{N+Rv}{Rv+\m^tB}} \ar[r]\ar@{->>}[d]
     &\disp{M_n \otimes  \frac{B}{Rv+\m^tB}} \ar[r]\ar@{->>}[d]  &\disp{\frac{M_n}{(\ux)M_n}}\ar[r] \ar[d]^{\id} & 0\\
 \disp{ \Tor_1^R\!\Big(\!M_n,  \frac{R}{(\ux)}\Big)}\ar[r]&\disp{M_n \!\otimes\!  \frac{N+Rv}{Ry+Rv+\m^tB}}\ar[r]
     & \disp{M_n \!\otimes\! \frac{B}{Ry+Rv+\m^tB}}  \ar[r]  &\disp{\frac{M_n}{(\ux)M_n}} \ar[r] & 0%%
} \]
By Proposition~\ref{torcap}, $\ell\big(\Tor_1^R(R/(\ux), M_n)\big)=\oo\big(\nu(M_n)\big)$, so the length of any homomorphic image of
$\Tor_1^R(R/(\ux), M_n)$ is also $\oo\big(\nu(M_n)\big)$. Since the alternating sum of the lengths in a finite
exact sequence of modules of finite length is 0, we obtain an equation on lengths from each of the two rows:
\[\ell\Big(M_n \otimes \frac{B}{Rv+\m^tB}\Big)+\oo(\nu(M_n))=
\ell\Big(M_n\otimes\frac{N+Rv}{Rv+\m^tB}\Big)+\ell\Big(\frac{M_n}{(\ux)M_n}\Big)\ \ \mathrm{and} \]
\[\ell\Big(M_n \otimes\frac{B}{Ry+Rv+\m^tB}\Big)+\oo\big(\nu(M_n\big))=
\ell\Big(M_n \otimes\frac{N+Rv}{Ry+Rv+\m^tB}\Big)+\ell\Big(\frac{M_n}{(\ux)M_n}\Big)\]

Since $y\in(Rv)\cln_B\subseteq (Rv+\m^tB)\cln_B$, we also know that
\[\ell\Big(M_n \otimes\frac{B}{Rv+\m^tB}\Big)-\ell\Big(M_n \otimes \frac{B}{Ry+Rv+\m^tB}\Big)=\oo\big(\nu(M_n)\big).\]
Combining these three equations we have that for all $t \gg 0$
 \[\ell\Big(M_n \otimes\frac{N+Rv}{Rv+\m^tB}\Big)-\ell\Big(M_n \otimes\frac{N+Rv}{Ry+Rv+\m^tB}\Big)=
 \oo\big(\nu(M_n)\big).\]

 These equations imply that $y\in(Rv+\m^tB)\cln_{N+Rv}$ for all $t\gg0$.
 Since we know that $B/(N+Rv)$ has finite length, we have that is killed by $\m^{t_0}$ for some $t_0$.
 But then $\m^{t+t_0}B \inc \m^t(N+Rv)$,and so $y \in  (Rv + \m^t(N+Rv)\cln_{N+Rv}$ for all $t \gg 0$.
 From the definition of ${}\cln$,  it follows that $y \in (Rv)\cln_{N+Rv}$.  Since
 $(N+Rv)/Rv\cong N/(N\cap Rv)$ we have that  $y\in (N\cap Rv)\cln_N\subseteq (N\cap Rv)\cln_B.$

But since $f(v)=x_d$, it is easy to see that $N\cap Rv=\Ker f\cap Rv= (J:x_d)v$. By Theorem~\ref{thmcappar}(b),
we have that $J:x_d \inc J\cln_R$.  Consequently,
$$\big((J:x_d)v\big)\cln_B\subseteq (J\cln_R  v)\cln_B.$$
Consider the map $\theta: R \to B$  such that $r \mapsto rv$.  By the Dietz axiom (4),  which we have
already established, $ \theta(J\cln_R) \inc \theta(J)\cln_B$, i.e., $J\cln_R v \inc (Jv)\cln_B$.
But then we have $(J\cln_R  v)\cln_B \inc \big((Jv)\cln_B\big)\cln_B = (Jv)\cln_B$, by the Dietz axiom (2), which also
has been proved.  Combining this with the line displayed just above, we have $\big((J:x_d)v\big)\cln_B \inc (Jv)\cln_B$,
as required. This completes the proof of the case where $k=d-1$.

Now assume $k< d-1$. We consider the composite map
\[f_t: B\xrightarrow{f}R/J=R/(x_1,\dots,x_k)\surj R/(x_1,\dots, x_k, x_{k+2}^t,\dots, x_d^t)\]
where the rightmost map is just the natural surjection. We still have $f_t(v)=x_{k+1}$.  We now apply the
result for $k =d-1$ treating $x_{k+1}$ as the last parameter, which we may, since systems of parameters are
permutable. It follows that
\[(Rv)\cln_B\cap\Ker f_t\subseteq \big(Jv+(x_{k+2}^t,\dots, x_d^t)v\big)\cln_B\subseteq (Jv+\m^tB)\cln_B.\]
Finally, we have \[(Rv)\cln_B\cap\Ker f \subseteq \bigcap_t \big((Rv)\cln_B\cap \Ker f_t\big)\subseteq
\bigcap_t (Jv+\m^tB)\cln_B=(Jv)\cln_B. \]
\end{proof}

%% 9.5
\begin{remark} Let $\rmk$ be a complete local domain and, for simplicity, assume that $K$ is algebraically closed.
Let $\cM = \{M_n\}_n$ be  as sequence of module-finite extension algebras of $R$ that is a lim \CM sequence of
modules.  Even if all the $M_n$ are also domains, we do not know whether the closure operation associated with this sequence satisfies the algebra axiom of  \cite{RG18}.  Whenever that is so, we can prove that $R$ has a big \CM algebra. \end{remark}

\section{Strongly lim \CM sequences}\label{strong} %%10

In this section we define the notion of a {\it strongly lim \CM} sequence of modules, and prove
several theorems about their behavior. We use local duality and
spectral sequence arguments to prove some  length estimates for various homology and cohomology
modules. See Theorems~\ref{spseq}, \ref{boundhom}, and \ref{thmcap}.
In particular, we show that strongly lim \CM sequences are lim \CM, which is not obvious.  This is Corollary~\ref{strongimp}.  We note that if $R$ is F-finite local
and $M$ is a finitely generated $R$-module of dimension $d = \dim(R)$,  then $F^n_*(M)$
is a strongly lim \CM sequence of modules: see Theorem~\ref{Fnstrong}.  This result strengthens
Theorem~\ref{nMlimcm}.

Recall from \S\ref{alt} that $\ell_\cV(H)$ denotes the shortest
length of a finite filtration of $H$ in which all factors are modules in $\cV$, or $+\infty$ if no such filtration
exists.

%% 10.1
\begin{definition} Let $\rmk$ be a local ring of Krull dimension $d$.  We define a sequence\footnote{Again, there is
an obvious generalization to nets of modules.}  of modules
$\cM = \{M_n\}_n$ to be {\it strongly lim \CM} if the Krull dimension of every $M_n$ is $d$, and there exists a finite
set of Artinian modules $\cV = \{\vect V a\}$ such that for every $j < d$,  $\ell_\cV\big(H^j_\m(M_n)\big) =
 \oo\big(\nu(M_n)\big).$
\end{definition}.

%% 10.2
\begin{proposition}\label{stcomp} The sequence $\{M_n\}_n$ is strongly lim \CM over $R$ if and only if the sequence
$\{\wh{M}_n\}_n$ is strongly lim \CM over $\wh{R}$. \end{proposition}
\begin{proof} This is clear, since the local cohomology modules are the same. \end{proof}

Let $\rmk$ be a local ring of Krull dimension $d$.  In this section, tensor product, Hom, and Ext are all taken over
$R$ unless otherwise indicated by the use of subscripts. Let $\blank^{\lor}$
indicate Matlis dual over $R$, i.e., $\Hom_R(\blank, E)$ where $E = E_R(K)$ is the injective
hull of $K$ over $R$.  

%% 10.3
\begin{discussion}\label{dual} Let $M$ and $W$ denote finitely generated
$R$-modules.
Note that if $G_\bu$ is a left complex of finitely generated free $R$-modules over a local ring $R$ whose
augmentation $N = H_0(G_\bu) = \Coker(G_1 \to G_0)$ is locally free on the punctured spectrum
and such that $H_i(G_\nu)$ has finite length for $i \geq 1$,  then $H_i(G_\bu \otimes_R W)$
has finite length for every finitely generated $R$-module $W$ and $i \geq 1$. This is a consequence
of the fact that for any $f \in \m$,  $R_f \otimes_R G_\bu \to N_f \to 0$ is an exact sequence of
$R_f$-projective modules, so that $(G_\bu \otimes_R W) \otimes_R R_f$ is acyclic, and some
power of $f$ therefore kills $H_i(G_\bu \otimes_R W)$ for $i \geq 1$. Note that the hypotheses
are preserved if we make a base change to $\wh{R}$.   (The completion of $N$ is still locally free
on the punctured spectrum of $\wh{R}$:  a proper prime of $\wh{R}$ cannot contain $\m$, and $f$
is in $\m$ and not in the prime,  $N_f$ is projective over $R_f$ and the base change to $\wh{R}$
preserves this.) Hence, the homology for positive indices $i$ of
$G_\bu \otimes_R \wh{W} \cong (G_\bu \otimes_R \wh{R}) \otimes_{\wh{R}}\wh{W}$ has
the same finite length as $H_i(G \otimes_R W)$.   Let $\wh{G}_\bu = \wh{R} \otimes_R G_\bu$.
If $V$ is any Artinian $R$-module, then
$$H^i\bigl(\Hom(G_\bu,\, V)\bigr)^\lor \cong H^i\bigl(\Hom(G_\bu \otimes_R \wh{R},\, V)\bigr)^\lor
\cong H_i(\wh{G}_\bu \otimes_{\wh{R}} V^\lor),$$
 which will have finite length
homology since $V^\lor$ is Noetherian over $\wh{R}$.  Moreover, $H^i\bigl(\Hom(G_\bu,\, V)\bigr)$
and $H_i(\wh{G}_\bu \otimes_{\wh{R}} V^\lor)$  have the same length. \end{discussion}

%% 10.4
\begin{discussion} If $R = S/\fA$ is a homomorphic
image  of a Gorenstein ring $S$ of Krull dimension $d+h$, we let $\omega_i(M)$ denote
$\Ext_S^{d+h-i}(M,\,S)$, which is a finitely generated $R$-module
whose Matlis dual over $R$ is $H^{d-i}_\m(M)$:  the latter condition determines $\omega_i(M)$
up to isomorphism.  If $R$ is complete, we may take $\omega_i(M)$ to be the Matlis dual
of $H^{d-i}_\m(M)$.  In the sequel, we may take $\omega_i(M)$ to be a finitely generated module
over either $R$ or $\wh{R}$ whose Matlis dual is $H^{d-i}_\m(M)$.

In this situation we have a {\it dualizing complex} $\cI^\bu$ for $R$ obtained as follows.
 Let $\cJ^\bu$ be a minimal injective resolution of $S$, and let $\cI^i = \Hom_S(R, \, \cJ^{i+h})$.
 Note that $\Hom_S(R, \cJ^s) = 0$ for $s < h$,
since all of the associated primes of $\cJ^s$ have height $s$, so that $\cJ^s$ does not contain
an element killed by $\fA$.
The $\cI^i$ are injective over $R$, each module $H^i(\cI^\bu) = \omega_i(R)$ is
finitely generated, and its Matlis dual is $H^{d-i}_\m(R)$.
(If $R$ is Cohen-Macaulay, $\cI^\bu$ is an injective resolution of a canonical module $\omega =
\omega_0(R)$ for $R$.)  The dualizing complex gives an alternative way of calculating
$\omega_j(M)$: if $M$ is finitely generated, we may let $\omega_j(M):= H^j\bigl(\Hom_R(M,\, \cI^\bu)\bigr)$.
In fact, this gives the same modules as the calculation of $\Ext_S^{d'-i}(M, S)$.
For any finitely generated module $M$,  $H^i_\m(M) \cong \omega_{d-i}(M)^{\lor}$.
\end{discussion}

Note that in the theorem below, we can avoid referring to the modules $\omega_t(M)$ (and we do
so in the alternative statements), since
 by Discussion~\ref{dual}:  \vskip 3pt
\centerline{$ \disp H_i\bigl(G_\bu \otimes  H^j_\m(M)\bigr)^\lor \cong
H^i\Bigl(\Hom\bigl(G_\bu, \omega_{d-j}(M)^\lor\bigr)\Bigr) \cong
 H^i\Bigl(\Hom\bigl(G_\bu, \omega_{d-j}(M)\bigr)\Bigr)^\lor$}
\noindent and so $H_i\bigl(G_\bu \otimes  H^j_\m(M)\bigr)$  and
$H^i\Bigl(\Hom\bigl(G_\bu, \omega_{d-j}(M)\bigr)\Bigr)$
have the same length (finite or infinite)
for all $i$.  The two lengths are finite if $i\geq 1$.

%% 10.5
\begin{theorem}\label{spseq} Let $\rmk$ be a local ring of Krull dimension $d$, let $M$ be a finitely
generated $R$-module, and let $G_\bu$ be a left complex $ \cdots \to G_k \to \cdots \to G_0 \to 0$
of finitely generated free $R$-modules, which may have infinite length. Assume that $N = H_0(G_\bu)$ is locally free
on the punctured spectrum of $R$ and that $H_i(G_\bu)$ has finite length for $i \geq 1$.
Let $\omega_t(M)$ be a finitely generated module over $R$ or $\wh{R}$ whose Matlis
dual is $H^{d-t}_\m(M)$.
Then: \vskip 3 pt
\centerline{$\disp  \ell\Bigl(H^0_\m\bigl(H_0(G_\bu \otimes M)\bigr)\Bigr) \leq
\sum_{s,t \geq 0,\,\,s+t = d} \ell\Bigl(H^s\bigl(\Hom_R\bigl(G_\bu, \, \omega_t(M)\bigr)\bigr)\Bigr)$ \ and} \vskip 3pt
\centerline{$\disp \ell\bigl(H_i(G_\bu \otimes M)\bigr) \leq
\sum_{s,t \geq 0,\,\,s+t = d+i} \ell\Bigl(H^s\bigl(\Hom_R\bigl(G_\bu, \, \omega_t(M)\bigr)\bigr)\Bigl)$\ for $i \geq 1$.}

Alternatively: \vskip 3pt
\centerline{$\disp  \ell\Bigl(H^0_\m\bigl(H_0(G_\bu \otimes M)\bigr)\Bigr) \leq
\sum_{s,t \geq 0,\,\,s+t = d} \ell\bigl(H_s(G_\bu \otimes_R H^{d-t}_\m(M)\bigr)$.\ and} \vskip 3pt
\centerline{$\disp \ell\bigl(H_i(G_\bu \otimes M)\bigr) \leq
\sum_{s,t \geq 0,\,\,s+t = d+i} \ell\bigl(H_s(G_\bu \otimes_R H^{d-t}_\m(M)\bigr)$\ for $i \geq 1$.}
\end{theorem}

%% 10.6
\begin{remark} Before giving the proof of Theorem~\ref{spseq}, we observe that
the first two statements in the conclusion of the theorem can be combined into a single statement:  since
the module on the left in the second statement has finite length, it can be replaced by
$H^0_\m\bigl(H_i(G_\bu \otimes M)\bigr)$.  The resulting version of the second statement is then
true when $i =0$:  that is the first statement.  The same remark applies to the alternative forms.  Note also
that if any module on the right hand side of the one of these statements has infinite length, the
statement is obviously true.  \end{remark}

\begin{proof} The hypotheses are preserved if we replace $R$ by its completion.  (If $R = S/\fA$ with
$S$ a Gorenstein local ring, this step is not needed, since $R$ has
a dualizing complex).
Let $0 \to \cI^0 \to \cdots \to \cI^d$ be a dualizing complex for $R$.
We consider the spectral sequences associated with double complex $\Hom_R(G_\bu \otimes _R M, \, \cI^\bu)$:

$$\CD
{}   @.          0                                                      @.    {}      @.        0                @. {} \\
@.               @VVV                                                @.                      @VVV              @.\\
0 @>>> \Hom_R(G_0 \otimes M,\, \cI^0) @>>> \cdots @>>> \Hom_R(G_0 \otimes M, \,\cI^d) @>>> 0\\
@.               @VVV                                                @.                      @VVV              @.\\
{}     @.        \vdots                                      @.        \vdots   @.         \vdots                      @. {} \\
@.               @VVV                                                @.                      @VVV              @.\\
0 @>>> \Hom_R(G_k \otimes M,\, \cI^0) @>>> \cdots @>>> \Hom_R(G_k \otimes M, \,\cI^d) @>>> 0\\
@.               @VVV                                                @.                      @VVV              @.\\
{}   @.          0                                                      @.    {}      @.        0                @. {} \\
      \endCD$$

If we calculate the homology of the $j\,$th column,  then since $\Hom_R(\blank, \cI^j)$ is
exact we get  $\Hom_R\bigl(H_\bu(G_\bu \otimes M), \cI^j)$.  Thus, the iterated cohomology
at the $i,j\,$ spot is dual to $H^{d-j}_\m \bigl(H_i(G_\bu \otimes M)\bigr)$.  Since the modules
$H_\bu(G_\bu \otimes M)$ have finite length unless $i =0$, the terms in the array vanish unless
$i =0$ or $j = d$.  Thus, in the $E_2$ array, only the top row and the rightmost column are nonzero.
Let $N := H_0(G_\bu)$.  By the right exactness of tensor, $H_0(G_\bu \otimes M) \cong N \otimes M$.

The top row ($i=0$)  consists of
$$
H^d_\m(N \otimes M)^{\lor} \ \ \ldots \ \ H^{d-j}_\m (N \otimes M)^{\lor} \ \ \ldots \ \  H^0_\m(N \otimes M)^{\lor}.$$
The remaining terms ($i \geq 1$) in the rightmost column  ($j = d$)  are the modules
$$\Bigl(H^0_\m\bigl(H_i(G_\bu \otimes M)\bigr)\Bigr)^\lor \cong H_i(G_\bu \otimes M)^\lor.$$
Note that $d_r:E_r^{i,j} \to E_r^{i-1,j+2}.$
It follows at once that the target of each of these maps is 0 if $i =0$ or $j = d$ for
$r \geq 2$,  and so $E_2 = E_\infty$.

We next calculate the cohomology first with respect to rows and then with respect to columns.
Since $\Hom_R(G \otimes M, \, \blank) \cong \Hom_R\bigl(G,\, \Hom_R(M,\, \blank)\bigr)$ as functors
of two variables, after we take cohomology with respect to rows the $j\,$th column
is $\Hom_R\bigl(G_\bu,\, \omega_j(M)\bigr)$,  so that the $i,j$ term in the $E_2$ array
is $H^i\bigl(\Hom_R\bigl(G_\bu, \, \omega_j(M)\bigr)\bigr)$.

The terms  on the diagonals $s+t = d+i$ for $i \geq 0$ in $E_r$ are replaced by subquotients
each time $r$ increases by 1, and these converge to an associated graded of $E_\infty$.   But
for these diagonals, $E_\infty$ has only one term on each such diagonal.  The length of this
unique term is therefore bounded by the sum of the lengths of the terms on the diagonal in the $E_2$ term.
Coupled with the fact that a module and its Matlis dual have the same length, whether finite
or infinite, this yields the statements in the theorem.
\end{proof}

\begin{remark}
Let us give another explanation of Theorem~\ref{spseq} via derived category. First note that $G_\bullet \otimes_R^LM\in D^{\leq 0}$ and thus we have canonical map
$$G_\bullet \otimes_R^LM\to H^0(G_\bullet \otimes_R^LM) \to H^0(G_\bullet \otimes_R^LM) /H_\m^0(H^0(G_\bullet \otimes_R^LM)) =:M'$$
Consider the exact triangle defining $X$:
$$X \to G_\bullet \otimes_R^LM \to M' \xrightarrow{+1}.$$
Note that $X\in D^{\leq 0}$ and by Discussion~\ref{dual}, we know that $H^i(X)$ has finite length for all $i$, in particular, $R\Gamma_\m(X)\cong X\in D^{\leq 0}$. Since $R\Gamma_\m(M')\in D^{\geq 1}$, it follows that 
$$H^{-i}(X) \cong H^{-i}(R\Gamma_\m(X)) \cong H^{-i} (R\Gamma_\m(G_\bullet \otimes_R^LM)) \cong H^{-i} (G_\bullet\otimes_R^L R\Gamma_\m(M))$$
for all $i\geq 0$. Now by looking at the $E_2$ page of the standard spectral sequence for the last term, we obtain that 
$$\ell\big(H^{-i}(X)\big) \leq \sum_{-s+d-t=-i} \ell\big(H^{-s}(G_\bullet\otimes_R^L H_\m^{d-t}(M))\big)$$
for all $i\geq 0$. Rewriting the above using homology we have that 
$$\ell\big(H_{i}(X)\big) \leq \sum_{s+t=d+i} \ell\big(H_s(G_\bullet\otimes_R^L H_\m^{d-t}(M))\big)$$
Since $H_i(X)=H_i(G_\bullet \otimes_R^LM)$ for $i\geq 1$ and $H_0(X) = H_\m^0(H_0(G_\bullet \otimes_R^LM))$, this gives the desired result.
\end{remark}

%% 10.7
\begin{theorem}\label{boundhom} Let $\rmk$ be a local domain of Krull dimension $d$ and
let $G_\bu$ be a complex of finitely generated free $R$-modules such $H_0(G_\bu)$ is locally
free on the punctured spectrum of $R$ and $H_i(G_\bu)$ has finite length for $i \geq 1$.
Let $s, j \in \N$ be such that $j \geq 1$ and $s < d$.  Let $\cM$ be a strongly lim Cohen-Macaulay
sequence over $R$. Then
$$ \ell\Bigl(H_j\bigl(G_ \bu \otimes H^s_\m(M_n)\bigr)\Bigr) = \oo\bigl(\nu(M_n)\bigr).$$
\end{theorem}
\begin{proof}  Let $\cV = \{\vect V a\}$ be a finite family of Artinian modules as in the definition of
strongly lim Cohen-Macaulay sequence.    Then

$$(\#) \quad\ell_\cV\bigl(H^s_\m(M_n)\bigr) = \oo\bigl(\nu(M_n)\bigr)$$ for $s < d$.
Let $B$ denote an upper bound
for the lengths of the finitely many modules  $H_j\bigl(G_\bu \otimes V_h\bigr)$, $1 \leq h \leq a$ (these
have finite length by Discssion~\ref{dual}).
Since $H_\bu(G_\bu \otimes \blank)$ has a long exact sequence, we may apply
$(*)$ of \ref{subadd} with $\lambda$ equal to length $\ell$ to obtain the estimate
$$ \ell\Bigl(H_j\bigl(G_ \bu \otimes H^s_\m(M_n)\bigr)\Bigr)\leq \ell_\cV\bigl(H^s_\m(M_n)\bigr) B.$$
Here, we have used $\ell_\cV$ to give an upper bound for the number of factors in a filtration of $H^s_\m(M_n)$ and
$B$ to bound the contribution from each factor.
Coupled with $(\#)$ above, this yields the desired conclusion. \end{proof}

%% 10.8
\begin{theorem}\label{thmcap}  Let $\rmk$ be a local domain of
Krull dimension $d$. Let $\cM$ be a strongly lim Cohen-Macaulay sequence of modules.
 If $G_\bu$ is a finite free complex $0 \to G_d \to  \cdots \to G_0 \to 0$  of finitely generated free modules
with finite length homology, $d_i:G_i \to G_{i-1}$,  $Z_i = \Ker(d_i) \inc G_i$ and $B_i = \Img(d_{i+1}) \inc G_i$,
then $Z_i \inc (B_i)\cl_{G_i}$ for all $i \geq 1$.  Moreover, $\ell\bigl(H_i(G_\bu \otimes M_n)\bigr) =
\oo\big(\nu(M_n)\big)$.
If rank is defined on the modules in $\cM$, then the same result holds for $\cM$-closure with respect to rank.
\end{theorem}
\begin{proof}  By the second statement in the alternative form of Theorem~\ref{spseq}, we have that for all $n$,
$$\ell\bigl(H_i(G_\bu \otimes M_n)\bigr) \leq
\sum_{s,t \geq 0,\,\,s+t = d+i} \ell\bigl(H_s(G_\bu \otimes_R H^{d-t}_\m(M_n)\bigr)$$ for $i \geq 1$.
Because $G_j$ is 0 for $j > d$,  we may assume that $s \leq d$, and then $s+t = d+i$ forces
$t \geq i \geq 1$ and so $d-t < d$.  It follows from Theorem~\ref{boundhom}
that as $s$ varies from $0$ to $d$, each of the $d+1$ possibly nonzero terms on the right
is $\oo\bigl(\nu(M_n)\bigr)$,  and so for $i \geq 1$ we have
$$\ell\bigl(H_i(G_\bu \otimes M_n)\bigr) = \oo\bigl(\nu(M_n)\bigr),$$
which establishes the second statement.
Let $Z_i = \Ker(G_i \to G_{i-1})$ and
let $B_i = \Img(G_{i+1} \to G_i)$.  Let $\wt{Z}_i^{(n)}$ be the image of $Z_i \otimes M_i$  in $G_i \otimes M_n$
and  $\cB_i^{(n)}$ be the image of $G_{i+1} \otimes M_n \to G_i \otimes M_n$, which is the
same as the image of $B_i \otimes M_n \to G_i \otimes M_n$.  Let
$\cZ_i^{(n)} = \Ker(G_i \otimes M_n \to G_{i-1} \otimes M_n)$.  Then $\wt{Z}_i^{(n)}/\cB_i^{(n)}$
injects into $\cZ_i^{(n)}/\cB_i^{(n)} = H_i(G_\bu \otimes M_n)$,  which has finite length.
Moreover $$\ell\bigl(\wt{Z}_i^{(n)}/\cB_i^{(n)}) \leq \ell\bigl(H_i(G_\bu \otimes M_n)\bigr) = \oo\bigl(\nu(M_n)\bigr).$$
By Remark~\ref{finlg}, this suffices to show that $Z_i \inc (B_i)\cl_{G_i}$.
\end{proof}

%% 10.9
\begin{corollary}\label{strongimp} A strongly lim \CM sequence of modules over a local ring is lim \CM.
\end{corollary}
\begin{proof} Simply apply Theorem~\ref{thmcap} when $\ux$ is a \sop for $R$ and $G_\bu = \cK_\bu(\ux;\,R)$.
\end{proof}

In order to prove Theorem~\ref{Fnstrong}, we need the following result, which uses an idea of \cite{Mon83} and
is close to results in \cite{Du83a, Sei89}, but does not appear to follow from these references in the generality we need
here.

%%10.10
\begin{theorem}\label{Fnfilt} Let $\rmk$ be an F-finite local ring and let $M$ be a finitely generated $R$-module of Krull dimension $d$.
Let $[K:K^p] = p^\alpha$.
Then there exists a finite set of primes $\vect P k$ in the support of $M$ such that for all $n \in \N$,
$F^n_*(M)$  has a filtration in which all factors have the form $R/P_i$, $1\leq i \leq k$, and the length of
the filtration is
$\OO(p^{(\alpha+ d)n})$, i.e., $\ell_{\{R/P_1, \, \ldots, \, R/P_k\}}\big(F^n_*(M)\big) = \OO(p^{(\alpha+d)n})$. \end{theorem}
\begin{proof} 
% We use the following notation in this proof:  if $M$ has a finite filtration with factors $\vect N k$ such
%that $N_i$ occurs $a_i$ times, we write $M \sim \sum_{i=1}^k a_iN_i$, while in this
%situation, if $b_i \geq a_i$ for every $i$ we write
%$M \prec \sum_{i=1}^k b_iN_i$.  Since $F_*$ is exact, we then
%get $F_*(M) \sim \sum_{i=1}^k a_iF_*(N_i)$.

%We proceed by Noetherian induction on $M$.   
Since $M$ has a finite filtration by prime cyclic
modules  $R/P$ and  $F^n_*$ is an exact functor, we immediately reduce to the case where $M$ itself
is prime cyclic.  Note that for every factor, the Krull dimension has not increased. Thus,
there is no loss of generality in assuming that $M = R$ is a domain. In this case, the result is well-known, see \cite[Lemma 2.2]{Pol18}.
\begin{comment}
We have an exact sequence
$0 \to R^b \to F_*(R) \to N \to 0$, where $N$ has strictly smaller Krull dimension than $M$,
where $b := p^{\alpha+d}$ by Remark~\ref{Fnrank}.
 Then
$$(\dagger_1) \quad F_*(R) \sim  bR + N.$$
By a straightforward induction on $n$,
$$(\dagger_n) \quad F^n_*(R) \sim b^nR + b^{n-1}N + b^{n-2}F_*(N) + \cdots + F^{n-1}_*(N).$$
At the inductive step, apply $F_*$ to formula $(\dagger_n)$ and use $(\dagger_1)$ to replace
$b^nF^*(R)$  by $b^n(bR) + b^nN$.

 Apply the induction hypothesis to $N$.  That is, there is a finite set of primes $\{\vect Pk\}$ in the support of
 $N$ such that $F^t_*(N) \prec  Cp^{(\alpha +d-1)t}\sum_{i=1}^k(R/P_i)$ for a constant $C > 0$.
 Since $p^{(\alpha +d-1)t} = (b/p)^t$,
using $(\dagger_n)$, we see that
$$F^n_*(R) \prec b^nR + \sum_{t=1}^{n-1} b^{n-t}C(b/p)^t\big((R/P_1) + \cdots + (R/P_k)\big).$$

This shows that $\ell_{\{R, \, R/P_1, \, \ldots, R/P_k\}}(F^n_*(R))$ is bounded by
$$b^n + kC \sum_{t=1}^{n-1} b^{n-t}(b/p)^t =
b^n(1+ kC\sum_{t=1}^{n-1} p^{-t}),$$ and the coefficient of $b^n$ is bounded independent of $n$
since $\sum_{t=1}^\infty p^{-t}$ converges.
\end{comment}
\end{proof}

 %% 10.11
 \begin{remark}\label{stlem}
 Let  $f:R \to S$ be a module-finite local map and let $M, \, W$ be finitely generated $S$-modules.
 \bena
 \item  The completion of $W$ with respect to the maximal ideal $\m_R$ of $R$ is the same as with respect to
 the maximal ideal $\m_S$ of $S$, since $\m_R S$ is primary to $\m_S$.  It follows that
 if $\wh{f}$ is the induced map $\wh{R} \to \wh{S}$, then $\wh{f}_*(\wh{M}) \cong \wh{f_*(M)}$. In particular,
 when $R$ is F-finite
 $(F_{\wh{R}}^n)_*(\wh{M}) \cong \wh{(F_R^n)_*(M)}$. %% (a)

\item  Let $E_R$ and $E_S$ denote the injective hulls of the residue class fields for $R$ and $S$.
Then $f_*\big(\Hom_S(W, E_S)\big) \cong \Hom_R(f_*(W), E_R)$ as $S$-modules: these are naturally isomorpic
functors of the $S$-module $W$.  Note that $E_S \cong \Hom_R(S, \, E_R)$.  In particular we may
apply this when $R$ is F-finite, $S = R$, and $f = F^n$ to obtain $F^n_*\big(\Hom_R(W, E_R)\big)\cong \Hom(F^n_*(W), E_R)$.
 \een
\end{remark}

%% 10.12
\begin{remark}\label{Hsdim} Let $\rmk$ be a complete local ring and let $M$ be a finitely generated $R$-module of
dimension $d$.  Let $E_R$ denote the injective hull of $K$ over $R$.  Then for $i <d$,
$\Hom_R(H^i_\m(M), \, E_R)$ has Krull dimension at most $d-1$ (where this is interpreted to mean
that it is 0 when $d=0$), see \cite[Tag 0A7U]{StaProj}.
\begin{comment}
To see this,  we may use induction
on $d$ while allowing $R$ to vary.  The case where $d = 0$ is obvious. We may also filter $M$ by
prime cyclic modules and use the long exact sequence for local cohomology along with mathematical
induction on the number of factors in the filtration to reduce to the case where $M$ is a prime cyclic module. Thus,
we might as well assume that $M = R$ is a domain and that the result holds for all choices of $M$
of dimension at most $d-1$ over any ring.  Then $R$ is module-finite over a complete regular local ring $R_0$,
and we may replace $R$ by $R_0$.  When we take a prime cyclic filtration of $R$ as a module over $R_0$ the factors
are either copies of $R_0$, with local cohomology 0 if $i < d$, or else of dimension smaller than $d$, so that
we may apply the induction hypothesis.
\end{comment}
\end{remark}

%% 10. 13
\begin{theorem}\label{Fnstrong} Let $\rmk$ be an F-finite local ring of Krull dimension $d$, and let $M$ be any
finitely generated $R$-module of Krull dimension $d$.  Then $\cM := \{F^n_*(M)\}_n$ is a strongly
lim \CM sequence of $R$-modules.  \end{theorem}
\begin{proof}  By Proposition~\ref{stcomp} and Remark~\ref{stlem}(a), we may assume that
$R$ is complete.   By Lemma~\ref{nu},
$\nu\big(F^n_*(M)\big)$ is asymptotic to  $\gamma p^{(\alpha+d)n}$,
where $p^\alpha := [K:K^p]$.

Hence, it suffices to filter the local cohomology
modules  $H^i_\m\big(F^n_*(M)\big)$  for $0 \leq i \leq d-1$
with numbers of factors that are $\oo\big(p^{(\alpha+d)n}\big)$.  Let $W_i$ denote
the Matlis dual of $H^i_\m(M)$ for $0 \leq i \leq d-1$.
By Remark~\ref{stlem}(b), it suffices to filter the modules $F^n_*(W_i)$ for $0 \leq i \leq d-1$ using a fixed
set of prime cyclic factors $R/P_1, \, \ldots, \, R/P_k$ in such a way that the length of the filtration
of each $F^n_*(W_i)$ is  $\OO(p^{(\alpha + d-1)n})$, for this will make it $\oo(p^{(\alpha + d)n})$.
By Remark~\ref{Hsdim}, the $W_i$ all have dimension at most $d-1$,  and so we do have
such filtrations by Theorem~\ref{Fnfilt}.  \end{proof}

\section{Mixed characteristic results}\label{mixed} %%11

In this section, we shall give some examples of lim Cohen-Macaulay sequences in mixed characteristic. A relatively simple source of examples is Frobenius lifts, which allows one to imitate the construction in characteristic $p$; see Examples~\ref{AffFrobLift} and \ref{ProjFrobLift}. We also include another class of examples (see \S \ref{ss:FrobNonLift}) where Frobenius does not lift. Throughout this section, we fix a prime $p>0$.

%% 11.1
\begin{notation}
Let $k$ be a perfect field of characteristic $p$.  Given a $k$-scheme $E$, write $F_{E/k}:E^{(-1)} \to E$ for the $k$-linear relative Frobenius for $E/k$ (obtained from the standard relative Frobenius $E \to E^{(1)}$ by twisting by the inverse Frobenius on $k$). Iterating this construction gives a tower
\[ \{\cdots \to E^{(-n-1)} \to E^{(-n)} \to ... \to E^{(-1)} \to E\} \]
of $k$-schemes that we call the Frobenius tower of $E/k$; note that the composite map $E^{(-n)} \to E$ is isomorphic to the $n$th power $F^n:E \to E$ of the Frobenius if we ignore $k$-linearity.

Let $V$ be a $p$-complete $p$-torsion-free DVR with  residue field $k$.

Given a closed immersion $Z \subset X$, we shall write $I_{Z\subset X}$ for the ideal sheaf of $Z$ in $X$; when $Z$ is a divisor, this is also $\mathcal{O}_X(-Z)$.
\end{notation}

%% 11.1
\subsection{Frobenius-liftable examples}

Let us first discuss some examples of lim Cohen-Macaulay sequences in mixed characteristic provided by Frobenius lifts.

%% 11.2
\begin{example}[Affine Frobenius lifts]
\label{AffFrobLift}
Let $(R,\mathfrak{m})$ be a $p$-adically complete and $p$-torsion-free noetherian local ring with perfect residue field $k$. Assume we are given an endomorphism $\phi:R \to R$ lifting the Frobenius on $R/p$. Then the sequence $\{\phi^n_* R\}$ is lim-Cohen-Macaulay. Indeed, we have $\mathrm{rank}_R(\phi^n_* R) = \mathrm{rank}_{R/p}(F^n_* R/p)$. Moreover, we can choose a system of parameters $(p,x_2,...,x_d)$ on $R$ such that $x_2,...,x_d$ is a system of parameters on $R/p$. We then have natural identifications $H_i(p, x_2,\dots,x_d; \phi^n_* R) \simeq H_i(x_2,\dots,x_d; F^n_* R/p)$. So the lim Cohen-Macaulay property for $\{\phi^n_* R\}$ over $R$ follows from that of $\{F^n_* R/p\}$ over $R/p$.
\end{example}

%% 11.3
\begin{example}[Projective Frobenius lifts]
\label{ProjFrobLift}
Let $\{X_n\}$ be a tower of smooth projective $V$-schemes whose base change to $k$ gives the Frobenius tower $\{E^{(-n)}\}$ of a smooth projective $k$-variety. Let $L$ be an ample line bundle on $X_0$, and write $L_n \in \mathrm{Pic}(X_n)$ for its pullback to $X_n$. Consider the homogeneous co-ordinate ring $A_n := \Gamma_*(X_n,L_n) := \oplus_{i \geq 0} H^0(X_n, L_n^i)$. Then we claim that $\{A_n\}$ forms a lim Cohen-Macaulay sequence over $A_0$ after localization at the homogeneous maximal ideal $\mathfrak{m} = \mathfrak{m}_V + (A_0)_{>0}$. To check this, by reducing modulo a uniformizer of $V$ and using the Koszul homology definition of lim Cohen-Macaulayness, we reduce to checking that $\{R_n := \Gamma_*(E^{(-n)}, L_n)\}$ is a lim-Cohen-Macaulay sequence over $R_0$ after localization at $(R_0)_{>0}$. Ignoring $k$-linearity, we can identify $R_n \simeq F^n_* \Gamma_*(E, L^{p^n})$ as $R_0$-modules. As $E$ is smooth, Serre vanishing and a standard calculation of local cohomology of affine cones shows that for any $j < \dim(R_0)$, the function $n \mapsto \ell(H^j_{\mathfrak{m}}(R_n))$ is constant for $n \gg 0$. In particular, the sequence $\{R_n\}$ is even strongly lim Cohen-Macaulay (and hence lim Cohen-Macaulay) after localization at $(R_0)_{>0}$.
\end{example}

%% 11.2
\subsection{Frobenius-non-liftable examples}
\label{ss:FrobNonLift}

We now give an example where Frobenius lifts are not available. For this, our strategy is to deform the Frobenius {\em as a map (and not as an endomorphism)} on exceptional divisors of carefully constructed resolutions of isolated singularities:

%% 11.4
\begin{proposition}[Deforming Frobenius on exceptional divisors]
\label{prop:FindFrobTower}
Let $(R,\mathfrak{m},k)$ be a complete local normal domain. Let $f:X \to \mathrm{Spec}(R)$ be a proper birational map of integral schemes which is an isomorphism outside $\mathfrak{m} \in \mathrm{Spec}(R)$, and such that the reduced preimage $E \subset X$ of $\mathfrak{m}$ is a Cartier divisor with $E$ being a smooth projective variety over $k$. Assume that the conormal bundle $L = I_{E \subset X}|_E \in \mathrm{Pic}(E)$ satisfies the following: for any $n,m \geq 1$, we have
\[ H^2(E, T_{E/k} \otimes (F^n)^* L^m) = 0 \quad \text{and} \quad  H^1(E, (F^n)^* (T_{E/k} \otimes L^m)) = 0,\]
where $F^n:E \to E$ is the $n$-fold Frobenius, and $T_{E/k} = (\Omega^1_{E/k})^\vee$ is the tangent bundle. Then one can find a tower $\{X_n\}$ of finite flat maps with base $X_0 = X$ whose pullback to $E$ agrees with the Frobenius tower of $E$.
\end{proposition}

\begin{proof}
Write $\mathfrak{X} = \mathfrak{X}_0$ for the formal scheme obtained by taking the $\mathfrak{m}$-adic formal completion of $X=X_0$. As $X \to \mathrm{Spec}(R)$ is proper, formal GAGA implies that it sufffices to solve the problem after $\mathfrak{m}$-adic formal completion, i.e., it suffices to construct a tower $\{\mathfrak{X}_n\}$ of $\mathfrak{m}$-adic formal schemes over $\mathfrak{X}$ with finite flat transition maps and such that the pullback to $E$ gives the Frobenius tower of $E$. For this, we proceed via deformation theory. We shall explain the construction when $n=1$; the general case proceeds exactly the same way, replacing $\mathfrak{X}_0$ in our construction with $\mathfrak{X}_n$ constructed previously by induction.

In the $n=1$ case, we must explain why the Frobenius map $F_{E/k}:E^{(-1)} \to E$ admits a finite flat lift $\mathfrak{X}_1 \to \mathfrak{X}_0$ across the inclusion $E \hookrightarrow \mathfrak{X}_0$. For an integer $m \geq 1$, write $\mathfrak{X}_{0,m} := V(I_E^{m})$ for the displayed infinitesimal neighbourhood of $E$, so $\mathfrak{X}_{0,1} = E$ and $\varinjlim_m \mathfrak{X}_{0,m} = \mathfrak{X}_0$. We shall construct a compatible system of finite flat maps $f_m:\mathfrak{X}_{1,m} \to \mathfrak{X}_{0,m}$ whose pullback to $\mathfrak{X}_{0,1} = E$ is the Frobenius $E^{(-1)} \to E$. We proceed by induction on $m$. For $m=1$, we simply take $\mathfrak{X}_{1,m} = E^{(-1)}$ with $f_1$ being the Frobenius $E^{(-1)} \to E$. Fix an integer $m \geq 1$ and assume by induction that we have found a finite flat map $f_{m}:\mathfrak{X}_{1,m} \to \mathfrak{X}_{0,m}$ inducing the Frobenius over $\mathfrak{X}_{0,1} \subset \mathfrak{X}_{0,m}$; it suffices to explain why $f_m$ admits a finite flat lift $f_{m+1}:\mathfrak{X}_{1,m+1} \to \mathfrak{X}_{0,m+1}$ across the square-zero thickening $\mathfrak{X}_{0,m} \subset \mathfrak{X}_{0,m+1}$.  Since $I_E^m/I_E^{m+1} = (I_E/I_E^2)^{\otimes m} = L^m$, the obstruction to finding such a deformation lies $\mathrm{Ext}^2(L_{E^{(-1)}/E}, F_{E/k}^* L^m)$. As Frobenius induces the $0$ map on differential forms, the transitivity triangle for the cotangent complex collapses to give $L_{E^{(-1)}/E} \simeq \Omega^1_{E^{(-1)}/k} \oplus F_{E/k}^* \Omega^1_{E/k}[1]$. Consequently, we can write
\[ \mathrm{Ext}^2(L_{E^{(-1)}/E}, F^* L^m) \simeq H^2(E^{(-1)}, T_{E^{(-1)}/k} \otimes F_{E/k}^* L^m) \oplus H^1(E^{(-1)}, F_{E/k}^*(T_{E/k} \otimes L^m)). \]
Identifying $E^{(-1)}$ with $E$ (not $k$-linearly!), we learn that the right side is isomorphic (as an abelian group) to
\[ H^2(E, T_{E/k} \otimes F^* (L^m)) \oplus H^1(E, F^*(T_{E/k} \otimes L^m)),\]
which vanishes by assumption.
\end{proof}

%% 11.5
\begin{remark}
The assumptions in Proposition~\ref{prop:FindFrobTower} do not include the constraint that the line bundle $L$ is ample on $E$. However, it is harmless to impose this for our purposes as it will be satisfied in cases of interest (Examples~\ref{EllCurveSing} and \ref{BadNonFLiftSing}). Moreover, once $L = I_E|_E \in \mathrm{Pic}(E)$ is ample, it follows that $I_E \in \mathrm{Pic}(X)$ is itself ample. To see this, first note that $I_E$ is ample on $X\times_{\Spec(R)}\Spec{(k)}$ by \cite[Tag 09MS]{StaProj}, and then \cite[Corollary 9.6.4]{EGAIV} tells us that $I_E$ is ample on $X$. 
\end{remark}

\begin{comment}
\begin{remark}
The assumptions in Proposition~\ref{prop:FindFrobTower} do not include the constraint that the line bundle $L$ is ample on $E$. However, it is harmless to impose this for our purposes as it will be satisfied in cases of interest (Examples~\ref{EllCurveSing} and \ref{BadNonFLiftSing}). Moreover, once $L = I_E|_E \in \mathrm{Pic}(E)$ is ample, it follows that $I_E \in \mathrm{Pic}(X)$ is itself ample. Indeed, by \cite[Tag 01Q3]{StaProj}, it suffices to show that for any coherent sheaf $F$ on $X$, the sheaf $F \otimes I_E^{N}$ is globally generated for  $N \gg 0$.  To verify this property, write $i:E \to X$ for the inclusion. The formal functions theorem and the ampleness of $L$ imply that the canonical map $F  \otimes I_E^n \to i_* (F|_E) \otimes I_E^n$ gives a surjection on $H^0(X,-)$ for all $n \gg 0$. Using the ampleness of $L$ on $E$, we may choose a surjection $(L^{-N})^{\oplus r} \to F|_E$ for $N \ll 0$. Regarding this as a section of $i_*(F|_E) \otimes (I_E^N)^{\oplus r}$, we can find a map $(I_E^{-N})^{\oplus r} \to F$ that is surjective after restriction to $E$. But $E$ contains all closed points of $X$, so the map $(I_E^{-N})^{\oplus r} \to F$ must be surjective, whence $F \otimes I_E^N$ is globally generated, as wanted.
\end{remark}
\end{comment}

We next give some examples where the hypotheses in Proposition~\ref{prop:FindFrobTower} are satisfied; note that the exceptional divisors $E$ appearing in these examples do not admit a lift to $V$ together with a lift of the Frobenius.

%% 11.6
\begin{example}[Elliptic curves]
\label{EllCurveSing}
Let $(R,\mathfrak{m},k)$ be a normal $2$-dimensional complete local flat $V$-algebra with minimal resolution $f:X \to \mathrm{Spec}(R)$. Assume that the reduced exceptional divisor of $X$ is an elliptic curve $E$ over $k$. Then the hypotheses of Proposition~\ref{prop:FindFrobTower} are satisfied: as the line bundle $L$ is ample (as $E$ is contracted by $f$, its normal bundle must be negative) and $T_{E/k} = \mathcal{O}_E$ is trivial (as $E$ has genus $1$), this follows from Riemann-Roch.

An explicit example of such an $R$ is given by $V\llbracket x,y\rrbracket/(f(x,y,\pi))$, where $f(x,y,z) \in V[x,y,z]$ is a homogenous cubic lifting the defining equation of a smooth curve $E := V(f(x,y,z)) \in \mathbf{P}^2_k$ and $\pi \in V$ is a uniformizer; for instance, we may take $f(x,y,z) = x^3 + y^3 + z^3$ when $p \neq 3$. The resolution $f:X \to \mathrm{Spec}(R)$ is given by the blowup $X=\mathrm{Bl}_{(\pi,x,y)}(\mathrm{Spec}(R))$ and has exceptional divisor $E$.
\end{example}

 %% 11.7
\begin{example}[Arbitrary liftable varieties]
\label{BadNonFLiftSing}
 Let $E/k$ be a smooth projective variety that admits a flat projective lift to $V$. Then we shall construct a map $f:X \to \mathrm{Spec}(R)$ as in Proposition~\ref{prop:FindFrobTower} with exceptional fiber $E$ and such that $R$ (and hence $X$) is $V$-flat. Moreover, in our construction, we can also arrange that the line bundle $L = I_E|_E \in \mathrm{Pic}(E)$ is as ample as desired.

By the liftability assumption, we can find a flat projective lift $\widetilde{E}/V$ of $E/k$; in particular, the conormal bundle of $E \subset \widetilde{E}$ is trivial. Let $M$ be a very ample line bundle on $\widetilde{E}$ whose restriction to $E$ satisfies the vanishing required in Proposition~\ref{prop:FindFrobTower}.

We will construct a blowup $\widetilde{X} \to \widetilde{E}$ where $E$ lifts and such that $E \subset \widetilde{X}$ has conormal bundle $M|_E$; this implies that $E \subset \widetilde{X}$ can be contracted, at least in algebraic spaces, which then yields the desired map $f$ by base changing to the complete local ring of the contraction.

To find the desired blowup, choose a smooth divisor $Z \subset E$ in the linear system $M$, so $I_{Z \subset E} \simeq M|_E$. In particular, $Z \subset \widetilde{E}$ is a codimension $2$ regularly immersed closed subscheme.  Let $\widetilde{X} = \mathrm{Bl}_Z(\widetilde{E})$ be the blowup of $\widetilde{E}$ along $Z$. As $Z$ was already a Cartier divisor on $E$, the strict transform of $E$ in $\widetilde{X}$ is isomorphic to $E$, so we can view $E$ as a divisor on $\widetilde{X}$. Moreover, as the conormal bundle of $E \subset \widetilde{E}$ was trivial, a calculation shows that the conormal bundle $I_{E \subset \widetilde{X}}|_{E}$ of $E \subset \widetilde{X}$ identifies with $M|_E$. By Artin's theorems \cite[Theorem 3.1, Theorem 6.2]{Art70}, there exists a proper birational contraction $\widetilde{f}:\widetilde{X} \to \widetilde{Y}$ of $E$ in the category of algebraic spaces. Take $f:X \to \mathrm{Spec}(R)$ to be the base of $\widetilde{f}$ to the complete local ring of $\widetilde{Y}$ at its singular point. The inclusion $E \subset \widetilde{X}$ refines to $E \subset X$ as $E$ is contracted to the singular point of $\widetilde{Y}$ under $\widetilde{f}$. The conormal bundle of $E \subset X$ is the same as that of $E \subset \widetilde{X}$, i.e., it equals $M|_E$. As the vanishing conditions in Proposition~\ref{prop:FindFrobTower} are satisfied by assumption, we win.
\end{example}

Our desired examples will be obtained from Proposition~\ref{prop:FindFrobTower} by passing to section rings. For this, we need the following lemma bounding the length of the local cohomology of the resulting rings:

%% 11.8
\begin{lemma}
\label{EstLenRes}
Let $(R,\m)$ be a complete normal local domain of dimension $d$. Let $Y\to \mathrm{Spec}(R)$ be a projective birational morphism of normal schemes that is an isomorphism outside $\{\m\}$. Suppose $Y$ is Cohen-Macaulay and the reduced pre-image $E\subseteq Y$ of $\{\m\}$ is a prime Cartier divisor, and that $L=I_E=\mathcal{O}_Y(-E)$ is ample. Let $S=\oplus_{j\geq 0}H^0(Y, L^j)$ be the section ring of $Y$ with respect to $L$, with $\m_S=\m+ S_{>0}$ (note that $S_0=R$). Suppose $N\geq 0$ is such that
$$H^{>0}(E, \omega_E\otimes L^j)=0=H^{>0}(E, \mathcal{O}_E\otimes L^j) \text{ for all $j>N$}.$$
Then $S$ is normal and we have
$$\ell(H_{\m_S}^i(S)) \leq (N-1)\cdot\sum_{n=2}^{N}\ell(H^{d-i+1}(E, \omega_E(n))) + (N+1)\cdot \sum_{n=0}^{N}\ell(H^{i-1}(E, \mathcal{O}_E(n)))$$
for all $2\leq i\leq d$. In particular, $\ell(H_{\m_S}^i(S))$ is bounded by data depending only on $E$ and $L|_E$, and the same bound remains true if $L$ is replaced by $L^{p^c}$ for any $c \geq 0$.
\end{lemma}

%% 11.9
\begin{remark}
If $N=0$ (in which case the first term should be interpreted as $0$), then the lemma implies $ \ell(H_{\m_S}^i(S)) \leq \ell(H^{i-1}(E, \mathcal{O}_E))$ for all $2\leq i\leq d$, and in fact, it follows from the proof that we have equality in this case.
\end{remark}

\begin{proof}
Since $L$ is ample, we have that $Y=\mathrm{Proj}(S)$ and $\mathcal{O}_Y(1):= \widetilde{S(1)}=L$. Next note that $S$ is normal in $R[T]$, since $S$ can be viewed as the Rees ring associated to the divisorial valuation $\mathrm{ord}_E(-)$. Since $R$ is normal, it follows that $S$ is normal. 

Now we consider the Sancho de Salas exact sequence
$$\cdots \to \oplus_{j\in\mathbb{Z}} H_E^{i-1}(Y, \mathcal{O}_Y(j))\to H_{\m_S}^i(S) \to \oplus_{j\geq 0} H_\m^i(S_j) \to \oplus_{j\in\mathbb{Z}} H_E^i(Y, \mathcal{O}_Y(j)) \to \cdots.$$
It follows that
\begin{equation}
\label{eqn 1}
H_{\m_S}^i(S)_{-j}  \cong  H_E^{i-1}(Y, \mathcal{O}_Y(-j)) \cong H^{d-i+1}(Y, \omega_Y(j))^\vee \text{ for all $i$ and all $j>0$.}
\end{equation}
where the second isomorphism follows from duality and that $Y$ is Cohen-Macaulay (so $\omega_Y^{\bullet}\cong \omega_Y[d]$). 
From the exact sequence
$$0\to \omega_Y \to \omega_Y(E)\cong \omega_Y(-1) \to \omega_E\to 0,$$
after twisting and taking global sections we obtain an exact sequence
$$H^{d-i+1}(Y, \omega_Y(j)) \to H^{d-i+1}(Y, \omega_Y(j-1)) \to H^{d-i+1}(E,\omega_E(j)).$$
Our assumption says that, when $i\leq d$, $H^{d-i+1}(E, \omega_E(j))=0$ for all $j>N$. This, together with Serre vanishing, which gives that $H^{d-i+1}(Y, \omega_Y(j))=0$ when $i\leq d$ and $j\gg0$ (and an obvious descending induction), shows that
\begin{equation}
\label{eqn 2}
H^{d-i+1}(Y, \omega_Y(j))=0 \text{ when $i\leq d$ and $j\geq N$}
\end{equation}
and that
\begin{align}
\label{eqn 3}
  \ell(H^{d-i+1}(Y,\omega_Y(j)))  & \leq \sum_{n=j+1}^{N}\ell(H^{d-i+1}(E, \omega_E(n))) \\
   &  \leq \sum_{n=2}^{N}\ell(H^{d-i+1}(E, \omega_E(n))) \text{ when $i\leq d$ and $0<j< N$}. \notag
\end{align}
Putting (\ref{eqn 1}), (\ref{eqn 2}), (\ref{eqn 3}) together we obtain
\begin{equation}\label{eqn 4}
\ell(H_{\m_S}^i(S)_{<0}) \leq (N-1)\cdot\sum_{n=2}^{N}\ell(H^{d-i+1}(E, \omega_E(n))) \text{ for all $i\leq d$. }
\end{equation}

Now we investigate $H_{\m_S}^i(S)_{\geq 0}$. Note that for each $j\geq 0$, we have the following commutative diagram with exact rows for all $i\geq 2$:
\[\xymatrix{
H_E^{i-1}(Y, \mathcal{O}_Y(j))\ar[r] \ar[d]^= & H_{\m_S}^i(S)_j \ar[r] \ar[d] & H_\m^i(S_j) \ar[r] \ar[d]^\cong  & H_E^i(Y, \mathcal{O}_Y(j)) \ar[d]^= \\
H_E^{i-1}(Y, \mathcal{O}_Y(j))\ar[r] & H^{i-1}(Y, \mathcal{O}_Y(j)) \ar[r] & H^{i-1}(Y-E, \mathcal{O}_{Y-E}(j)) \ar[r]  & H_E^i(Y, \mathcal{O}_Y(j))
}
\]
where the vertical isomorphism follows from the fact that when $i\geq 2$ we have (note that we have $S_j=H^0(Y, \mathcal{O}_Y(-jE))\hookrightarrow S_0=H^0(Y, \mathcal{O}_Y)$ whose cokernel has finite length):
$$H_\m^i(S_j)\cong H_\m^i(R) \cong H^{i-1}(Y-E, \mathcal{O}_{Y-E})\cong H^{i-1}(Y-E, \mathcal{O}_{Y-E}(j)).$$
Thus we have
\begin{equation}\label{eqn 5}
H_{\m_S}^i(S)_j \cong H^{i-1}(Y, \mathcal{O}_Y(j)) \text{  for all $i\geq 2$ and $j\geq 0$. }
\end{equation}
From the short exact sequence
$$0\to \mathcal{O}_Y(-E)=\mathcal{O}_Y(1) \to \mathcal{O}_Y\to \mathcal{O}_E\to 0,$$
after twisting and taking global sections we obtain an exact sequence
$$ H^{i-1}(Y, \mathcal{O}_Y(j+1)) \to H^{i-1}(Y, \mathcal{O}_Y(j)) \to H^{i-1}(E, \mathcal{O}_E(j)) $$
Our assumption says that, when $i\geq 2$, $H^{i-1}(E, \mathcal{O}_E(j))=0$ for all $j>N$. This together with Serre vanishing that $H^{i-1}(Y, \mathcal{O}_Y(j))=0$ when $i\geq 2$ and $j\gg0$ (and an obvious descending induction) shows that
\begin{equation}\label{eqn 6}
H^{i-1}(Y, \mathcal{O}_Y(j))=0 \text{ when $i\geq 2$ and $j>N$}
\end{equation}
and that
\begin{align}\label{eqn 7}
H^{i-1}(Y, \mathcal{O}_Y(j)) & \leq \sum_{n=j}^{N}\ell(H^{i-1}(E, \mathcal{O}_E(n)))  \\
 & \leq \sum_{n=0}^{N}\ell(H^{i-1}(E, \mathcal{O}_E(n))) \text{ when $i\geq 2$ and $0\leq j\leq N$} \notag
\end{align}
Putting (\ref{eqn 5}), (\ref{eqn 6}), (\ref{eqn 7}) together we obtain
\begin{equation}\label{eqn 8}
\ell(H_{\m_S}^i(S)_{\geq 0}) \leq (N+1)\cdot \sum_{n=0}^{N}\ell(H^{i-1}(E, \mathcal{O}_E(n))) \text{ for all $i\geq 2$. }
\end{equation}

Finally, by (\ref{eqn 4}) and (\ref{eqn 8}) we have
$$ \ell(H_{\m_S}^i(S)) \leq (N-1)\cdot\sum_{n=2}^{N}\ell(H^{d-i+1}(E, \omega_E(n))) + (N+1)\cdot \sum_{n=0}^{N}\ell(H^{i-1}(E, \mathcal{O}_E(n)))$$
for every $2\leq i\leq d$ as wanted.
\end{proof}

%% 11.10
\begin{example}[A lim Cohen-Macaulay sequence without Frobenius lifts]
\label{ex:LimCMNoFrob}
Pick $f:X \to \mathrm{Spec}(R)$, $E$ and $L$ as in Proposition~\ref{prop:FindFrobTower} with $L$ ample (see Examples~\ref{EllCurveSing} and \ref{BadNonFLiftSing} for explicit examples). Let $\{X_n\}$ be the tower provided by Proposition~\ref{prop:FindFrobTower}. Let $S_n = \Gamma_*(X_n, I_{E_n \subset X_n}) = \oplus_{j \geq 0} H^0(X_n, I_{E_n \subset X_n}^j)$ be the homogenous co-ordinate ring of $X_n$ with respect to the ample line bundle $I_{E_n \subset X_n}$. Then we claim that $\{S_n\}$ forms a lim Cohen-Macaulay sequence over $S_0$ after localization at the graded maximal ideal $\mathfrak{m}_{S_0} = \mathfrak{m}_R + (S_0)_{>0}$.

To see this, let $E_n \subset X_n$ be the preimage of $E \subset X$, so each $E_n$ is identified with $E$ (not $k$-linearly), the induced map $E_n \to E$ is identified with $F^n:E \to E$, and thus the conormal bundle  $I_{E_n \subset X_n}|_{E_n}$ is identified with $(F^n)^* L = L^{p^n}$. Writing $R_n = H^0(X_n, \mathcal{O}_{X_n}) = (S_n)_0$, the induced map $X_n \to \mathrm{Spec}(R_n)$ is the Stein factorization of the composition $X_n \to X \to \mathrm{Spec}(R)$; as the latter is an alteration (being a composition of a proper birational map with a finite flat map), its Stein factorization $X_n \to \mathrm{Spec}(R_n)$ is a proper birational map and the induced map $R = R_0 \to R_n$ is a finite extension of complete local normal domains, whence $\sqrt{\mathfrak{m}_R R_n} = \mathfrak{m}_{R_n}$. It is clear from these descriptions that the map $X_n \to \mathrm{Spec}(R_n)$ satisfies the hypotheses of Lemma~\ref{EstLenRes}, that each $S_n$ is a finite $S_0$-algebra, and moreover that the graded maximal ideal $\mathfrak{m}_{S_n}$ of $S_n$ is given by $\sqrt{\mathfrak{m}_{S_0}}$. Applying the conclusion of Lemma~\ref{EstLenRes}, we learn that the function $n \mapsto \ell(H^i_{\mathfrak{m}_{S_0}}(S_n))$ is bounded for $i \leq \dim(R_n) = \dim(S_n)-1$, proving that $\{S_n\}$ is strongly lim Cohen-Macaulay.
\end{example}

We end this section by explaining why some of these examples (e.g., Example~\ref{EllCurveSing}) admit no small Cohen-Macaulay algebras. Our argument mirrors that in \cite{Bha14} via the Witt vectors using the following:

%% 11.11
\begin{lemma}
\label{LocalCohBigEx}
With notation as in Lemma~\ref{EstLenRes}, suppose $H^{>0}(E, L|_E^j)=0$ for all $j>0$ (e.g., when $N=0$, such as Example~\ref{EllCurveSing}) and that the residue characteristic is $p > 0$. Then $H^i(E, W_n\mathcal{O}_E)$ is a direct summand of $H^{i+1}_{\m_S}(W_nS)$ for all $n$ and all $0<i<d$.
\end{lemma}
\begin{proof}
First of all, since $H^{>0}(E, L|_E^j)=0$ for all $j>0$, it follows by Serre vanishing and descending induction on $j$ (as in the proof of Lemma~\ref{EstLenRes}) that $H^{>0}(Y, L^j)=0$ for all $j>0$. In particular, taking $j=1$ we have $H^{>0}(Y, I_E)=0$. This implies $H^{>0}(Y, W_nI_E)=0$ since $W_nI_E$ is an iterated extension of $I_E$ as abelian sheaves. Consider the long exact sequence induced by
$0\to W_nI_E\to W_n\mathcal{O}_Y \to W_n\mathcal{O}_E\to 0,$
we obtain that $H^{>0}(Y, W_n\mathcal{O}_Y)\cong H^{>0}(E, W_n\mathcal{O}_E)$.

Let $U:=\mathrm{Spec}(S)-V(S_{>0})$. The projection map $\pi$: $U\to Y=\mathrm{Proj}(S)$ is a $\mathbf{G}_m$-torsor.   As $Y$ is proper over a $p$-complete ring, the complex $R\Gamma(Y, W_n \mathcal{O}_Y)$ is already $p$-complete and thus agrees with $R\Gamma(\widehat{Y}, W_n \mathcal{O}_{\widehat{Y}})$, where $\widehat{Y}$ is the $p$-adic formal completion of $Y$. Lemma~\ref{WittSummand} implies that the composition
\[ R\Gamma(Y, W_n \mathcal{O}_Y) \to R\Gamma(U, W_n \mathcal{O}_U) \to R\Gamma(\widehat{U}, W_n \mathcal{O}_{\widehat{U}}) \]
is a split monomorphism (where $\widehat{U}$ is the $p$-adic formal completion of $U$), whence the first map is also a split monomorphism. We then learn that for all $i >0$
\begin{equation}\label{eqn 9}
H^i(E, W_n\mathcal{O}_E) \cong H^i(Y, W_n\mathcal{O}_Y)\to H^i(U, W_n\mathcal{O}_U) \text{ is a split monomorphism.}
\end{equation}
We next note that $H^{>0}(\mathrm{Spec}(S), W_nS)=0$ since $W_nS$ is an iterated extension of $S$. Thus, by the standard distinguished triangle $$R\Gamma_{V(S_{>0})}(\mathrm{Spec}(S), -)\to R\Gamma(\mathrm{Spec}(S), -)\to R\Gamma(U, -) \xrightarrow{+1},$$ we have
\begin{equation}\label{eqn 10}
H^i(U, W_n\mathcal{O}_U)\cong H^{i+1}_{V(S_{>0})}(W_nS) \text{ for all $i>0$}.
\end{equation}
Next we claim that $H^i_{V(S_{>0})}(W_nS)$ is supported only on $\{\m_S\}$ for all $i\leq d$, and thus
\begin{equation}\label{eqn 11}
H^i_{V(S_{>0})}(W_nS)\cong H^i_{\m_S}(W_nS) \text{ when $i\leq d$. }
\end{equation}
To see this, again, since $W_nS$ is an iterated extension of $S$, it is enough to show that $H^i_{V(S_{>0})}(S)$ is supported only at $\{\m\}$ as an $R$-module when $i\leq d$. But the standard Cech complex after rotating gives us a distinguished triangle
$$R\Gamma_{V(S_{>0})}(S) \to S \to \oplus_{j\in\mathbb{Z}}R\Gamma(Y,\mathcal{O}_Y(j))\xrightarrow{+1}$$
Now $S =\oplus_{j\in\mathbb{Z}}\Gamma(Y,\mathcal{O}_Y(j))$ by construction, $H^i(Y, \mathcal{O}_Y(j))=0$ for all $|j|\gg0$ when $1\leq i\leq d-1$ by Serre vanishing (and that $Y$ is Cohen-Macaulay), and each $H^{>0}(Y, \mathcal{O}_Y(j))$ is supported only at $\{\m\}$ since $Y\to\mathrm{Spec}(R)$ is an isomorphism outside $\{\m\}$. Putting these together we see that $H^i_{V(S_{>0})}(S)$ is supported only at $\{\m\}$ as an $R$-module when $i\leq d$. Finally, by (\ref{eqn 9}), (\ref{eqn 10}), (\ref{eqn 11}), we obtain that
$H^i(E, W_n\mathcal{O}_E)$ is a direct summand of $H^{i+1}_{\m_S}(W_nS)$ for all $0<i<d$ as wanted.
\end{proof}

The next lemma was used above.

%% 11.12
\begin{lemma}
\label{WittSummand}
  Let $f:V \to X$ be a $\mathbf{G}_m$-torsor in noetherian $p$-adic formal schemes. Then the natural map $W_n \mathcal{O}_X \to Rf_* W_n \mathcal{O}_V$ is a split monomorphism.
\end{lemma}

A similar assertion was claimed in \cite[Lemma 3.8]{Bha14}. However, the proof given there is incorrect: a $\mathbf{G}_m$-action on an $\mathbf{F}_p$-algebra $R$ does not obviously induce a $\mathbf{G}_m$-action on $W_n(R)$ as there is no obvious map $W_n(R[t^{\pm 1}]) \to W_n(R)[t^{\pm 1}]$. However, a $\mathbf{G}_m^{\mathrm{perf}}$-action on $R$ does induce a $\mathbf{G}_m^{\mathrm{perf}}$-action on $W_n(R)$ as $W_n(-)$ behaves predictably with relatively perfect maps: we have $W_n(R[t^{\pm \frac{1}{p^\infty}}]) \simeq W_n(R)[t^{\pm \frac{1}{p^\infty}}]$. This is enough to run the argument, and is what we do below.

\begin{proof}
First, observe that since $f$ is affine, we have $f_* W_n \mathcal{O}_V \simeq Rf_* W_n \mathcal{O}_V$, so everything is in degree $0$. Next, recall that specifying a $\mathbf{G}_m$-action on a $p$-complete abelian group is the same thing as specifying a $\Z$-grading (in the $p$-complete sense), with the functor of taking $\mathbf{G}_m$-invariants corresponding to passage to the degree $0$ summand. In particular, $f_* \mathcal{O}_V$ has a natural $\Z$-grading with degree $0$ summand $\mathcal{O}_X$. To extend this to $W_n(-)$, it is convenient to pass to the action of a perfect group scheme. Let $G = \mathbf{G}_m^{\mathrm{perf}} = \lim_{x \mapsto x^p} \mathbf{G}_m$, regarded as a group scheme; one has a natural projection map $G \to \mathbf{G}_m$ coming from the last term of the inverse limit. Then  $G$-actions on $p$-complete abelian groups are the same thing as $\Z[1/p]$-gradings; the functor of regarding a $\mathbf{G}_m$-action as a $G$-action corresponds to regarding a $\Z$-grading as a $\Z[1/p]$-grading. Now the $\mathbf{G}_m$-action on $f$ induces a $G$-action on $f$ via the map $G \to \mathbf{G}_m$. As $G$ is relatively perfect over $\Z_p$ as a $p$-adic formal scheme, we have $W_n(Z \times_{\mathrm{Spf}(\Z_p)} G) \simeq W_n(Z) \times_{\mathrm{Spf}(\Z_p)} G$ for any $p$-adic formal scheme $Z$. In particular, the map $W_n(f)$ has a natural $G$-action, whence $f_* W_n \mathcal{O}_V$ carries a $\Z[1/p]$-grading. Filtering $W_n \mathcal{O}_V$ by (Frobenius twisted) copies of $\mathcal{O}_V$, one then checks that the degree $0$ subring of $f_* W_n \mathcal{O}_V$ coincides with $W_n \mathcal{O}_X$, which gives the desired direct summand property.
\end{proof}

%% 11.13
\begin{corollary}[Non-existence of small Cohen-Macaulay algebras]
\label{NonExistSmallCM}
Consider Example~\ref{ex:LimCMNoFrob} with $E$ taken to be an elliptic curve (see Example~\ref{EllCurveSing}). Then the resulting ring $S$ admits no small Cohen-Macaulay algebra, i.e., there is no finite injective map $S \to T$ with $T$ being Cohen-Macaulay.
\end{corollary}
\begin{proof}
Note that $\dim(S) = \dim(Y) + 1 = \dim(E) + 1 + 1 =3$. Assume towards contradiction such a $T$ does exist. Following the trace argument of \cite[Lemma 3.4]{Bha14}, we learn that there is some integer $d > 0$ such that $d \cdot H^i_{\mathfrak{m}_S}(W_n S) = 0$ for $i < \dim(S)$ and all $n$. But Lemma~\ref{LocalCohBigEx} implies that $H^1(E, W_n \mathcal{O}_E)$ is a summand of $H^2_{\mathfrak{m}_S}(W_n S)$. As $E$ is an elliptic curve, we can find $n \gg 0$ such that the integer $d$ does not annihilate $H^1(E, W_n \mathcal{O}_E)$: indeed, the inverse limit $H^1(E, W \mathcal{O}_E) = \lim_n H^1(E, W_n \mathcal{O}_E)$ is nonzero after tensoring with $\Q$ by \cite{BBE07}. This is a contradiction, so no such $T$ can exist.
\end{proof}

\quad\bigskip

$\begin{array}{ll}
\textrm{Department of Mathematics}           &\quad\textrm{Department of Mathematics}\\
\textrm{Institute for Advanced Study}              &\quad\textrm{University of Michigan}\\
\textrm{Princeton, NJ}                                &\quad\textrm{Ann Arbor, MI 48109--1043}\\
\textrm{USA}                                 &\quad \textrm{USA}\\
\smallskip
\textrm{E-mail:}\  \textrm{bhargav.bhatt@gmail.com}         &\quad \textrm{E-mail:}\ \textrm{hochster@umich.edu}\\
\smallskip

\textrm{Department of Mathematics}\\
\textrm{Purdue University}\\
\textrm{West Lafayette, IN}\\
\textrm{USA}\\
\smallskip
\textrm{E-mail:}\ \textrm{ma326@purdue.edu}\\
\end{array}$

\bigskip\bigskip

\end{document}